\documentclass[11pt, oneside, reqno]{amsart}
\usepackage[utf8]{inputenc}
\usepackage{mathrsfs}
\usepackage{bm}
\usepackage{amstext}
\usepackage{amsthm}
\usepackage{amssymb}
\usepackage{mathtools}
\usepackage[all]{xy}
\usepackage{cleveref}
\usepackage{graphics}
\usepackage{ytableau}
\usepackage{version}
\usepackage{color}

\makeatletter
\@namedef{subjclassname@2020}{\textup{2020} Mathematics Subject Classification}
\makeatother
\numberwithin{equation}{section}
\numberwithin{figure}{section}
\theoremstyle{plain}
\newtheorem{thm}{Theorem}[section]
  \crefname{thm}{Theorem}{Theorems}
  \newtheorem{lem}[thm]{Lemma}
  \crefname{lem}{Lemma}{Lemmas}
  \newtheorem{prop}[thm]{Proposition}
  \crefname{prop}{Proposition}{Propositions}
  \newtheorem{cor}[thm]{Corollary}
	\crefname{cor}{Corollary}{Corollaries}
  \newtheorem*{prop1}{Proposition A}
  \newtheorem*{theor2}{Theorem B}
  \newtheorem*{theor3}{Theorem C}
  \newtheorem*{theor4}{Theorem D}
  \newtheorem*{cor5}{Corollary E}
  \newtheorem*{ack}{Acknowledgments}
\theoremstyle{definition}
  \newtheorem{defi}[thm]{Definition}
  \crefname{defi}{Definition}{Definitions}
  \theoremstyle{remark}
  \newtheorem{ntn}[thm]{Notation}
  \crefname{ntn}{Notation}{Notations}
	 \theoremstyle{remark}
  \newtheorem{rem}[thm]{Remark}
  \crefname{rem}{Remark}{Remarks}
  \newtheorem{ex}[thm]{Example}
  \crefname{ex}{Example}{Examples}
  
\usepackage{a4wide}

\def\r{\mathbb{R}}
\def\c{\mathbb{C}}
\def\q{\mathbb{Q}}
\def\z{\mathbb{Z}}

\newcommand{\Mat}{\operatorname{Mat}}
\newcommand{\Spec}{\operatorname{Spec}}
\newcommand{\rank}{\operatorname{rank}}
\newcommand{\sgn}{\operatorname{sgn}}

\newcommand{\Image}{\operatorname{Im}}

\makeatother

\makeatletter
\ifcsname phantomsection\endcsname
    \newcommand*{\qrr@gobblenexttocentry}[5]{}
\else
    \newcommand*{\qrr@gobblenexttocentry}[4]{}
\fi
\newcommand*{\addsubsection}{%
    \addtocontents{toc}{\protect\qrr@gobblenexttocentry}%
    \subsection}
\makeatother
 \makeatletter
    
    \@addtoreset{equation}{section}
  \makeatother
\begin{document}
\title[Newton--Okounkov polytopes arising from cluster structures]{Newton--Okounkov polytopes of Schubert varieties arising from cluster structures}

\author{Naoki FUJITA}

\address[Naoki FUJITA]{Faculty of Advanced Science and Technology, Kumamoto University, 2-39-1 Kurokami, Chuo-ku, Kumamoto 860-8555, Japan.}

\email{fnaoki@kumamoto-u.ac.jp}

\author{Hironori OYA}

\address[Hironori OYA]{Department of Mathematics, Institute of Science Tokyo, 2-12-1 Ookayama, Meguro-ku, Tokyo 152-8551, Japan.}

\email{hoya@math.titech.ac.jp}



\thanks{The work of the first named author was supported by Grant-in-Aid for JSPS Fellows (No.\ 19J00123), by JSPS Grant-in-Aid for Early-Career Scientists (No.\ 20K14281, 24K16902), and by MEXT Japan Leading Initiative for Excellent Young Researchers (LEADER) Project. 
The work of the second named author was supported by JSPS Grant-in-Aid for Early-Career Scientists (No.\ 19K14515).}

\date{}

\begin{abstract}
The theory of Newton--Okounkov bodies is a generalization of that of Newton polytopes for toric varieties, and it gives a systematic method of constructing toric degenerations of projective varieties. In this paper, we study Newton--Okounkov bodies of Schubert varieties from the theory of cluster algebras. We construct Newton--Okounkov bodies using specific valuations which generalize extended g-vectors in cluster theory, and discuss how these bodies are related to string polytopes and Nakashima--Zelevinsky polytopes. 
\end{abstract}
\maketitle
\tableofcontents 
\section{Introduction}\label{s:section}

A Newton--Okounkov body $\Delta(X, \mathcal{L}, v)$ is a convex body constructed from a polarized variety $(X, \mathcal{L})$ with a higher rank valuation $v$ on the function field $\mathbb{C}(X)$, which was introduced by Okounkov \cite{Oko1, Oko2, Oko3}, and afterward developed independently by Kaveh--Khovanskii \cite{KK1,KK2} and by Lazarsfeld--Mustata \cite{LM}. It generalizes the notion of Newton polytopes for toric varieties to arbitrary projective varieties. A remarkable fact is that the theory of Newton--Okounkov bodies gives a systematic method of constructing toric degenerations (see \cite[Theorem 1]{And}). In the case of flag varieties and Schubert varieties, their Newton--Okounkov bodies realize the following representation-theoretic polytopes:
\begin{enumerate}
\item[{\rm (i)}] Berenstein--Littelmann--Zelevinsky's string polytopes \cite{Kav},
\item[{\rm (ii)}] Nakashima--Zelevinsky polytopes \cite{FN},
\item[{\rm (iii)}] Feigin--Fourier--Littelmann--Vinberg polytopes \cite{FeFL, Kir}.
\end{enumerate}
The set of lattice points in each polytope of {\rm (i)}--{\rm (iii)} parametrizes a specific basis of an irreducible highest weight representation of a semisimple Lie algebra. In particular, string polytopes and Nakashima--Zelevinsky polytopes give polyhedral parametrizations of crystal bases (see \cite{BZ2, Lit, Nak1, NZ}). One motivating problem is to relate these polytopes using the framework of cluster algebras. 
Following Gross--Hacking--Keel--Kontsevich \cite{GHKK}, the theory of cluster algebras also can be used to obtain toric degenerations of projective varieties. Cluster algebras were introduced by Fomin--Zelevinsky \cite{FZ:ClusterI, FZ:ClusterII, FZ:ClusterIV} to develop a combinatorial approach to total positivity and to the dual canonical basis. Fock--Goncharov \cite{FG} introduced a cluster ensemble $(\mathcal{A}, \mathcal{X})$ which gives a more geometric point of view to the theory of cluster algebras. In the paper \cite{GHKK}, the notion of positive polytopes was introduced, which induces toric degenerations of compactified $\mathcal{A}$-cluster varieties. In this paper, we study relations between these two constructions of toric degenerations: Newton--Okounkov bodies and cluster algebras. In addition, we relate string polytopes with Nakashima--Zelevinsky polytopes by tropicalized cluster mutations.

To be more precise, let \[
\mathscr{U}(\mathcal{S}) = \bigcap_{\mathbf{s}} \c[A_{j; \mathbf{s}} ^{\pm 1} \mid j \in J]
\]
be an upper cluster algebra of geometric type, where $\mathbf{s} = (\mathbf{A} = (A_{j; \mathbf{s}})_{j \in J}, \varepsilon)$ runs over a set of Fomin--Zelevinsky seeds which are all mutation equivalent and stable under the mutations. Assuming that $X$ is birational to the corresponding cluster variety $\mathcal{A}$, we have the following identification for each $\mathbf{s}$: 
\[\mathbb{C}(X) \simeq \mathbb{C}(A_{j; \mathbf{s}} \mid j \in J).\] 
We consider the case that the exchange matrix $\varepsilon$ is of full rank for all $\mathbf{s}$. Fix a total order $\le_\varepsilon$ on $\mathbb{Z}^J$ refining the opposite order $\preceq_{\varepsilon} ^{\rm op}$ of Qin's dominance order $\preceq_{\varepsilon}$ \cite{Qin}; this $\le_\varepsilon$ induces a total order on the set of Laurent monomials in $A_{j; \mathbf{s}}$, $j \in J$. We define a valuation 
\[v_\mathbf{s} \colon \mathbb{C}(X) \setminus \{0\} \rightarrow \mathbb{Z}^J\]
to be the associated lowest term valuation (see \cref{ex:lowest_term_valuation}). By definition, this valuation $v_\mathbf{s}$ generalizes the extended $g$-vectors of cluster monomials and of the theta function basis. 
This observation implies that under Gross--Hacking--Keel--Kontsevich's toric degenerations \cite{GHKK} of compactified $\mathcal{A}$-cluster varieties, the Newton polytopes of the toric varieties can be realized as Newton--Okounkov bodies (see Section \ref{s:GHKK} for more details). 
Such valuation generalizing the extended $g$-vectors was also studied by Kaveh--Manon \cite[Section 7]{KavMan} in a different context. 
In the present paper, our valuation comes from $\mathcal{A}$-cluster structures, but we can also use $\mathcal{X}$-cluster structures to obtain valuations. 
Such valuation coming from an $\mathcal{X}$-cluster structure was previously introduced by Rietsch--Williams \cite{RW} for Grassmannians. 
They studied the associated Newton--Okounkov bodies of Grassmannians from the viewpoint of mirror symmetry.
Bossinger--Frias-Medina--Magee--N$\acute{\rm a}$jera Ch$\acute{\rm a}$vez \cite{BFMN} gave a different construction of toric degenerations using $\mathcal{X}$-cluster structures.
They developed the theory of $\mathcal{X}$-cluster varieties \emph{with coefficients}, and used it to construct a toric degeneration of a partially compactified $\mathcal{X}$-cluster variety as a certain cluster dual family of the degeneration given in \cite{GHKK}. 
Since an $\mathcal{X}$-cluster variety is not necessary separated (see \cite[Remark 4.2]{GHK}), toric degenerations constructed in \cite{BFMN} cannot be obtained from Newton--Okounkov bodies in general (see \cite[Section 1.3]{BFMN} for more details). 

To state our main results explicitly, let $G$ be a simply-connected semisimple algebraic group over $\mathbb{C}$, $B$ a Borel subgroup of $G$, $W$ the Weyl group, and $P_+$ the set of dominant integral weights. We denote by $X(w) \subset G/B$ the Schubert variety corresponding to $w \in W$, and by $\mathcal{L}_\lambda$ the globally generated line bundle on $X(w)$ associated with $\lambda \in P_+$. Let $U_w ^- \subset G$ be the unipotent cell associated with $w \in W$, which is naturally regarded as an open subvariety of $X(w)$. 
Berenstein--Fomin--Zelevinsky \cite{BFZ2} gave an upper cluster algebra structure on the coordinate ring $\c[U_w ^-]$. Then $U_w ^-$ is birational to the corresponding $\mathcal{A}$-cluster variety, and we have $\c(X(w)) = \c(U_w ^-) = \c(\mathcal{A})$, which gives a valuation $v_\mathbf{s}$ on $\c(X(w))$ for each seed $\mathbf{s}$. 
Let $\mathcal{A}^\vee$ denote the Fock--Goncharov dual \cite{FG} of $\mathcal{A}$. 
By tropicalizing the cluster mutations $\mu_k$ for $\mathcal{A}^\vee$, we obtain the tropicalized cluster mutations $\mu_k ^T$. 
Kashiwara--Kim \cite{KasKim} developed the theory of extended $g$-vectors of the upper global basis by using Kang--Kashiwara--Kim--Oh's monoidal categorification \cite{KKKO} of the unipotent quantum coordinate ring, which is valid only when the Cartan matrix is symmetric. 
After the first version of the present paper was put on arXiv, Qin's preprint \cite{Qin3} appeared, which proves that the upper global basis is a common triangular basis when the Cartan matrix is symmetrizable. 
The following proposition is a straightforward consequence of this property of the upper global basis. 

\begin{prop1}[{\cref{c:relation_of_NO_by_tropicalized_mutations}}]
Let $w \in W$, $\lambda \in P_+$, and $\mathbf{s}$ a seed for $\c[U_w ^-]$.
\begin{enumerate}
\item[{\rm (1)}] The Newton--Okounkov body $\Delta(X(w), \mathcal{L}_\lambda, v_\mathbf{s})$ is independent of the choice of a refinement of the opposite dominance order $\preceq_{\varepsilon} ^{\rm op}$.
\item[{\rm (2)}] If $\mathbf{s}^\prime = \mu_k (\mathbf{s})$, then 
\[\Delta(X(w), \mathcal{L}_\lambda, v_{\mathbf{s}^\prime}) = \mu_k ^T (\Delta(X(w), \mathcal{L}_\lambda, v_\mathbf{s})).\]
\end{enumerate}
\end{prop1}

Let $R(w)$ denote the set of reduced words for $w \in W$. 
Then there exists a specific class $\{\mathbf{s}_{\bm i} = (\mathbf{D}_{\bm{i}}, \varepsilon^{\bm{i}})
 \mid {\bm i} \in R(w)\}$ of seeds for $\c[U_w ^-]$ whose cluster $\mathbf{D}_{\bm{i}}$ is given by  generalized minors. 
Let $\Delta_{\bm i} (\lambda)$ (resp., $\widetilde{\Delta}_{\bm i} (\lambda)$) be the string polytope (resp., the Nakashima--Zelevinsky polytope) associated with ${\bm i} \in R(w)$ and $\lambda \in P_+$. 
We consider a lexicographic order $\prec$ on $\mathbb{Z}^{\ell (w)}$ (see \cref{d:lowest_term_valuation_Schubert} for its  precise definition), where $\ell (w)$ is the length of $w$. 
Our first main result relates the Newton--Okounkov bodies arising from cluster structures with the string polytopes. 

\begin{theor2}[{see \cref{p:refine} and \cref{c:relstring}}]
Let $w \in W$, $\lambda \in P_+$, and ${\bm i} \in R(w)$. 
\begin{enumerate}
\item[{\rm (1)}] The lexicographic order $\prec$ refines the opposite dominance order $\preceq_{\varepsilon^{\bm{i}}} ^{\rm op}$.
\item[{\rm (2)}] The Newton--Okounkov body $\Delta(X(w), \mathcal{L}_\lambda, v_{\mathbf{s}_{\bm i}})$ is unimodularly equivalent to the string polytope $\Delta_{\bm i} (\lambda)$.
\end{enumerate}
\end{theor2}

We note that unimodular equivalence between GHKK (Gross--Hacking--Keel--Kontsevich) superpotential polytopes in cluster theory and string polytopes was previously given by Magee \cite{Mag} and Bossinger--Fourier \cite{BF} in type $A$, and by Genz--Koshevoy--Schumann \cite{GKS} in simply-laced case. 
In their works, they dealt with double Bruhat cells and base affine spaces, while our main objects in the present paper are Schubert cells and unipotent cells. 
In Section \ref{s:superpotential_polytopes}, we relate GHKK superpotential polytopes with our Newton--Okounkov bodies in type $A$.

As an application of Theorem B, we prove that Anderson's construction \cite{And} of toric degenerations can be applied to the Newton--Okounkov body $\Delta(X(w), \mathcal{L}_\lambda, v_{\mathbf{s}})$, and we obtain the following. 

\begin{theor3}[{see \cref{t:main_result_toric_degenerations}}]
Let $w \in W$, $\lambda \in P_+$, and $\mathbf{s}$ a seed for $\c[U_w ^-]$.
\begin{enumerate}
\item[{\rm (1)}] The Newton--Okounkov body $\Delta(X(w), \mathcal{L}_\lambda, v_\mathbf{s})$ is a rational convex polytope.
\item[{\rm (2)}] If $\mathcal{L}_\lambda$ is very ample on $X(w)$, then there exists a flat degeneration of $X(w)$ to the normal toric variety corresponding to the rational convex polytope $\Delta(X(w), \mathcal{L}_\lambda, v_\mathbf{s})$.
\end{enumerate}
\end{theor3}

We then relate the Newton--Okounkov bodies arising from cluster structures with the Nakashima--Zelevinsky polytopes. 
Desired relations are revealed when we take a seed ${\bf s}_{\bm i} ^{\rm mut} = (\mathbf{D}_{\bm{i}} ^{\rm mut}, \varepsilon^{\bm{i}, {\rm mut}})$ obtained from ${\bf s}_{\bm i}$ through an appropriate mutation sequence $\overrightarrow{\boldsymbol{\mu}_{\bm{i}}}^{\vee}$ (see Section \ref{s:mutation_sequence} for the precise definition). 
We should note that the idea of our choice of seeds comes from \emph{the Chamber Ansatz formulas}, established by Berenstein--Fomin--Zelevinsky \cite{BFZ1} and Berenstein--Zelevinsky \cite{BZ1} (see \cref{t:Chamber_Ansatz}). 

\begin{theor4}[{see \cref{c:relNZ}}]
Let $w \in W$, $\lambda \in P_+$, and ${\bm i} \in R(w)$. 
Then the Newton--Okounkov body $\Delta(X(w), \mathcal{L}_\lambda, v_{{\bf s}_{\bm i} ^{\rm mut}})$ is unimodularly equivalent to the Nakashima--Zelevinsky polytope $\widetilde{\Delta}_{\bm i} (\lambda)$.
\end{theor4}

Combining Theorems B and D with Proposition A, we obtain the following.

\begin{cor5}[{\cref{c:relation_string_NZ}}]
For $w \in W$ and $\lambda \in P_+$, the string polytopes $\Delta_{\bm i} (\lambda)$ and the Nakashima--Zelevinsky polytopes $\widetilde{\Delta}_{\bm i} (\lambda)$ associated with ${\bm i} \in R(w)$ are all related by tropicalized cluster mutations up to unimodular transformations.
\end{cor5}

For simplicity, we deal with only simply-connected semisimple algebraic groups $G$ in the present paper, but Theorems B, C, D and Corollary E can be extended to symmetrizable Kac--Moody groups $G$ without much difficulty.

Finally, we mention some related works. An analogous application of $g$-vectors to representation theory was given by Fei \cite{Fei17,Fei:tensor} under the assumption that $G$ is of symmetric type. 
He dealt with an upper cluster algebra structure of the ring of regular functions on the base affine space $U^-\backslash G$, and constructed a polytope by using the $g$-vectors of generic characters whose lattice points count a weight multiplicity of an irreducible representation of $G$ (see \cite[Example 3.10]{Fei17} and \cite[Theorem 10.2]{Fei:tensor}). 
When $G$ is of symmetric type, Genz--Koshevoy--Schumann \cite{GKS} proved that, for the longest element $w_0$ in $W$, the string polytopes $\Delta_{\bm i} (\lambda)$, ${\bm i} \in R(w_0)$, are all related by tropicalized cluster mutations. 
Casbi \cite{Cas} also studied relations between monoidal categorifications and Newton--Okounkov bodies, using monomial valuations for cluster variables (see \cite[Remark 4.5]{Cas}). Kanakubo--Nakashima \cite{KanNak1, KanNak2} gave relations between cluster algebras and crystal bases, especially focusing on cluster variables and cluster tori. Note that our approach in this paper is different from theirs.
It is an interesting problem to study relations between Newton--Okounkov bodies of flag varieties arising from cluster structures and Feigin--Fourier--Littelmann--Vinberg polytopes. 
Bossinger \cite{Bos} gave a partial answer to the question whether these polytopes are unimodularly equivalent. 

This paper is organized as follows. In Section \ref{s:NO_bodies}, we recall the definition of Newton--Okounkov bodies, and review results of \cite{FO}. 
In Section \ref{s:Clusterval}, we recall some basic facts on cluster algebras, and define our main valuations. 
In Section \ref{s:GHKK}, we review Gross--Hacking--Keel--Kontsevich's toric degenerations, and relate them with Newton--Okounkov bodies. 
Section \ref{s:GHKK} is independent of the rest of this paper; hence readers who are mainly interested in the results concerning Newton--Okounkov bodies of Schubert varieties may skip this section. 
In Section \ref{s:unipcell}, we recall Berenstein--Fomin--Zelevinsky's cluster structure on unipotent cells, and show Proposition A above by using the existence of a $\c$-basis of $\c[U_w ^-]$ with desirable properties. 
In Section \ref{s:rel_with_stringNZ}, we realize string polytopes and Nakashima--Zelevinsky polytopes as Newton--Okounkov bodies arising from cluster structures. 
Theorem C above is also proved in this section. 
In Section \ref{s:superpotential_polytopes}, we relate our Newton--Okounkov bodies arising from cluster structures with superpotential polytopes introduced in \cite{GHKK}.
Section \ref{s:cluster_cone} is devoted to studying ray generators of cluster cones in some specific cases.

\begin{ack}\normalfont
The authors are grateful to Lara Bossinger, Elie Casbi, Xin Fang, Akihiro Higashitani, Tsukasa Ishibashi, Yoshiyuki Kimura, and Bea Schumann for helpful comments and fruitful discussions. 
The authors thank Jiarui Fei for informing them of his results in \cite{Fei17,Fei:tensor}. 
The authors would also like to thank anonymous referees for reading the manuscript carefully and for suggesting many improvements. 
\end{ack}

\section{Newton--Okounkov bodies}\label{s:NO_bodies}

\subsection{Basic definitions and properties}

First of all, we review the definition of Newton--Okounkov bodies, following \cite{HK, Kav, KK1, KK2}. Let $R$ be a $\mathbb{C}$-algebra without nonzero zero-divisors, and fix a total order $\lhd$ on $\mathbb{Z}^m$, $m \in \z_{>0}$, respecting the addition. 

\begin{defi}\label{defval}
	A map $v \colon R \setminus \{0\} \rightarrow \mathbb{Z}^m$ is called a \emph{valuation} on $R$ if the following holds: for every $\sigma, \tau \in R \setminus \{0\}$ and $c \in \mathbb{C}^\times \coloneqq \mathbb{C} \setminus \{0\}$,
	\begin{enumerate}
		\item[{\rm (i)}] $v(\sigma \cdot \tau) = v(\sigma) + v(\tau)$,
		\item[{\rm (ii)}] $v(c \cdot \sigma) = v(\sigma)$, 
		\item[{\rm (iii)}] $v (\sigma + \tau) \unrhd {\rm min} \{v(\sigma), v(\tau) \}$ unless $\sigma + \tau = 0$. 
	\end{enumerate}
\end{defi}

We need to fix a total order on $\z^m$ whenever we consider a valuation. The following is a fundamental property of valuations.

 \begin{prop}[{see, for instance, \cite[Proposition 1.8 (2)]{Kav}}]\label{prop1_val}
	Let $v$ be a valuation on $R$, $\sigma_1, \ldots, \sigma_s \in R \setminus \{0\}$, and $V \coloneqq \c \sigma_1 + \cdots + \c \sigma_s$. Assume that $v(\sigma_1), \ldots, v(\sigma_s)$ are all distinct. Then for $c_1, \ldots, c_s \in \mathbb{C}$ such that $\sigma \coloneqq c_1 \sigma_1 + \cdots + c_s \sigma_s \neq 0$, the following equality holds: 
	\[
	v(\sigma) = \min\{v(\sigma_t ) \mid 1 \le t \le s,\ c_t \neq 0 \}.
	\]
	In particular, it follows that 
\[v (V \setminus \{0\}) = \{v(\sigma_1), \ldots, v(\sigma_s)\}.\]
\end{prop} 

For ${\bm a} \in \z^m$ and a valuation $v$ on $R$ with values in $\z^m$, we set 
\[
R_{\bm a} \coloneqq \{\sigma \in R \setminus \{0\} \mid v(\sigma) \unrhd {\bm a}\} \cup \{0\};
\]
this is a $\c$-subspace of $R$. Then the \emph{leaf} above ${\bm a} \in \z^m$ is defined to be the quotient space $R[{\bm a}] \coloneqq R_{\bm a}/\bigcup_{{\bm a}\lhd {\bm b}} R_{\bm b}$. A valuation $v$ is said to have \emph{1-dimensional leaves} if $\dim_\c(R[{\bm a}]) = 0\ {\rm or}\ 1$ for all ${\bm a} \in \z^m$. 

\begin{ex}\label{ex:lowest_term_valuation}
Let $\lhd$ be a total order on $\mathbb{Z}^m$, $m \in \mathbb{Z}_{>0}$, respecting the addition, and $\mathbb{C}(t_1, \ldots, t_m)$ the field of rational functions in $m$ variables. The total order $\lhd$ on $\mathbb{Z}^m$ induces a total order (denoted by the same symbol $\lhd$) on the set of Laurent monomials in $t_1, \ldots, t_m$ as follows: 
\begin{center}
$t_1 ^{a_1} \cdots t_m ^{a_m} \lhd t_1 ^{a_1 ^\prime} \cdots t_m ^{a_m ^\prime}$ if and only if $(a_1, \ldots, a_m) \lhd (a_1 ^\prime, \ldots, a_m ^\prime)$. 
\end{center}
Let us define a map $v^{\rm low}_{\lhd} \colon \mathbb{C}(t_1, \ldots, t_m) \setminus \{0\} \rightarrow \mathbb{Z}^m$ as follows: 
\begin{itemize}
    \item $v^{\rm low} _{\lhd} (f) \coloneqq (a_1, \ldots, a_m)$ for
\[
f = c t_1 ^{a_1} \cdots t_m ^{a_m} + ({\rm higher\ terms}) \in \mathbb{C}[t_1, \ldots, t_m] \setminus \{0\},
\]
where $c \in \mathbb{C}^\times$, and the summand ``(higher terms)'' stands for a linear combination of monomials bigger than $t_1 ^{a_1} \cdots t_m ^{a_m}$ with respect to $\lhd$.
\item $v^{\rm low}_{\lhd} (f/g) \coloneqq v^{\rm low} _{\lhd} (f) - v^{\rm low} _{\lhd} (g)$ for $f, g \in \mathbb{C}[t_1, \ldots, t_m] \setminus \{0\}$.
\end{itemize}
It is obvious that this map $v^{\rm low} _{\lhd}$ is a well-defined valuation with $1$-dimensional leaves with respect to the total order $\lhd$. We call $v^{\rm low} _{\lhd}$ the {\it lowest term valuation} with respect to $\lhd$.
\end{ex}

\begin{defi}[{see \cite[Section 1.2]{Kav} and \cite[Definition 1.10]{KK2}}]\label{Newton--Okounkov body}
Let $X$ be an irreducible normal projective variety over $\c$ of complex dimension $m$, and $\mathcal{L}$ a line bundle on $X$ generated by global sections. Take a valuation $v \colon \mathbb{C}(X) \setminus \{0\} \rightarrow \z^m$ with 1-dimensional leaves, and fix a nonzero section $\tau \in H^0 (X, \mathcal{L})$. We define a subset $S(X, \mathcal{L}, v, \tau) \subset \z_{>0} \times \z^m$ by \[S(X, \mathcal{L}, v, \tau) \coloneqq \bigcup_{k \in \z_{>0}} \{(k, v(\sigma / \tau^k)) \mid \sigma \in H^0(X, \mathcal{L}^{\otimes k}) \setminus \{0\}\},\] and denote by $C(X, \mathcal{L}, v, \tau) \subset \r_{\ge 0} \times \r^m$ the smallest real closed cone containing $S(X, \mathcal{L}, v, \tau)$. Let us define a subset $\Delta(X, \mathcal{L}, v, \tau) \subset \r^m$ by \[\Delta(X, \mathcal{L}, v, \tau) \coloneqq \{{\bm a} \in \r^m \mid (1, {\bm a}) \in C(X, \mathcal{L}, v, \tau)\};\] this is called the {\it Newton--Okounkov body} of $X$ associated with $(\mathcal{L}, v, \tau)$. If the set $\Delta(X, \mathcal{L}, v, \tau)$ is a polytope, i.e., the convex hull of a finite number of points, then we call it a \emph{Newton--Okounkov polytope}.
\end{defi}

It follows by the definition of valuations that $S(X, \mathcal{L}, v, \tau)$ is a semigroup. Hence we see that $C(X, \mathcal{L}, v, \tau)$ is a closed convex cone, and that $\Delta(X, \mathcal{L}, v, \tau)$ is a convex set. Moreover, we deduce by \cite[Theorem 2.30]{KK2} that $\Delta(X, \mathcal{L}, v, \tau)$ is a convex body, i.e., a compact convex set. If $\mathcal{L}$ is very ample, then it follows from \cite[Corollary 3.2]{KK2} that the real dimension of $\Delta(X, \mathcal{L}, v, \tau)$ equals $m$; this is not necessarily the case if $\mathcal{L}$ is not very ample. If the semigroup $S(X, \mathcal{L}, v, \tau)$ is finitely generated, then $\Delta(X, \mathcal{L}, v, \tau)$ is a rational convex polytope, i.e., the convex hull of a finite number of rational points.

\begin{rem}\label{independence}
	If we take another section $\tau^\prime \in H^0 (X, \mathcal{L}) \setminus \{0\}$, then $S(X, \mathcal{L}, v, \tau^\prime)$ is the shift of $S(X, \mathcal{L}, v, \tau)$ by $k v(\tau/\tau^\prime)$ in $\{k\} \times \mathbb{Z}^m$ for $k \in \mathbb{Z}_{>0}$, that is, 
	\[
	S(X, \mathcal{L}, v, \tau^\prime) \cap (\{k\} \times \mathbb{Z}^m) = S(X, \mathcal{L}, v, \tau) \cap (\{k\} \times \mathbb{Z}^m) + (0, k v(\tau/\tau^\prime)).
	\] 
	Hence we obtain that $\Delta(X, \mathcal{L}, v, \tau^\prime) = \Delta(X, \mathcal{L}, v, \tau) + v(\tau/\tau^\prime)$. Thus, the Newton--Okounkov body $\Delta(X, \mathcal{L}, v, \tau)$ does not essentially depend on the choice of $\tau \in H^0 (X, \mathcal{L}) \setminus \{0\}$; hence it is also denoted simply by $\Delta(X, \mathcal{L}, v)$.
\end{rem} 

\begin{rem}\label{r:homogeneous_coordinate}
If $\mathcal{L}$ is a very ample line bundle on $X$, then it gives a closed embedding $X \hookrightarrow \mathbb{P}(H^0(X, \mathcal{L})^\ast)$. We denote by $R = \bigoplus_{k \in \z_{\ge 0}} R_k$ the corresponding homogeneous coordinate ring. Newton--Okounkov bodies are sometimes defined by using $R_k$ instead of $H^0(X, \mathcal{L}^{\otimes k})$ (see, for instance, \cite{HK}). However, since $X$ is normal, we deduce by \cite[Ch.\ I\hspace{-.1em}I Ex.\ 5.14]{Hart} that $R_k = H^0(X, \mathcal{L}^{\otimes k})$ for all $k \gg 0$. In addition, since $S(X, \mathcal{L}, v, \tau)$ is a semigroup, the cone $C(X, \mathcal{L}, v, \tau)$ coincides with the smallest real closed cone containing \[\bigcup_{k>k^\prime} \{(k, v(\sigma / \tau^k)) \mid \sigma \in H^0(X, \mathcal{L}^{\otimes k}) \setminus \{0\}\}\] for each $k^\prime \in \z_{\ge 0}$. Thus, $R_k$ and $H^0(X, \mathcal{L}^{\otimes k})$ are interchangeable in the definition of Newton--Okounkov bodies.
\end{rem}

We say that $X$ admits a \emph{flat degeneration} to a variety $X_0$ if there exists a flat morphism 
\[\pi \colon \mathfrak{X} \rightarrow \Spec(\c[z])\] 
of schemes such that the scheme-theoretic fiber $\pi^{-1}(z)$ (resp., $\pi^{-1}(0)$) over a closed point $z \in \c^\times$ (resp., the origin $0 \in \c$) is isomorphic to $X$ (resp., $X_0$). If $X_0$ is a toric variety, then this degeneration is called a \emph{toric degeneration}.

\begin{thm}[{see \cite[Theorem 1]{And} and \cite[Corollary 3.14]{HK}}]\label{t:toric deg}
Assume that $\mathcal{L}$ is very ample, and that the semigroup $S(X, \mathcal{L}, v, \tau)$ is finitely generated; hence $\Delta(X, \mathcal{L}, v, \tau)$ is a rational convex polytope. Then there exists a flat degeneration of $X$ to a (not necessarily normal) toric variety 
\[X_0 \coloneqq {\rm Proj} (\c[S(X, \mathcal{L}, v, \tau)]),\] 
where the $\z_{>0}$-grading of $S(X, \mathcal{L}, v, \tau)$ induces a $\z_{\ge 0}$-grading of $\c[S(X, \mathcal{L}, v, \tau)]$; note that the normalization of $X_0$ is the normal toric variety corresponding to the Newton--Okounkov polytope $\Delta(X, \mathcal{L}, v, \tau)$.
\end{thm}

\subsection{Newton--Okounkov bodies of Schubert varieties}\label{ss:NOSchu}
In this subsection, we fix our notation concerning Lie theory, and review results of \cite{FO}. 

Let $G$ be a connected, simply-connected semisimple algebraic group over $\mathbb{C}$, and $\mathfrak{g}$ its Lie algebra. Choose a Borel subgroup $B \subset G$ and a maximal torus $H \subset B$. Then the full flag variety is defined to be a quotient space $G/B$, which is a nonsingular projective variety. 
Denote by $\mathfrak{h} \subset \mathfrak{g}$ the Lie algebra of $H$, by $\mathfrak{h}^\ast \coloneqq {\rm Hom}_\mathbb{C} (\mathfrak{h}, \mathbb{C})$ its dual space, and by $\langle \cdot, \cdot \rangle \colon \mathfrak{h}^\ast \times \mathfrak{h} \rightarrow \mathbb{C}$ the canonical pairing. Denote by $I$ an index set for the vertices of the Dynkin diagram of $\mathfrak{g}$. 
Let $P \subset \mathfrak{h}^\ast$ be the weight lattice for $\mathfrak{g}$, $P_+ \subset P$ the set of dominant integral weights, $\{\alpha_i \mid i \in I\} \subset P$ the set of simple roots, $\{h_i \mid i \in I\} \subset \mathfrak{h}$ the set of simple coroots, $c_{i,j} \coloneqq \langle \alpha_j, h_i \rangle$ the Cartan integer for $i,j \in I$, and $e_i, f_i, h_i \in \mathfrak{g}$, $i \in I$, the Chevalley generators. Denote by $N_G(H)$ the normalizer of $H$ in $G$, and by $W \coloneqq N_G(H)/H$ the Weyl group of $\mathfrak{g}$. 
The simple reflection associated with $i\in I$ is denoted by $s_i$. For $w\in W$, let us denote by $\ell (w)$ the length of $w$ with respect to $\{s_i\mid i\in I\}$, and set 
\[
R(w)\coloneqq \{(i_1,\dots, i_{\ell(w)}) \in I^{\ell(w)} \mid s_{i_1}\dots s_{i_{\ell(w)}}=w\}. 
\]
An element of $R(w)$ is called a \emph{reduced word} for $w \in W$. 

For $i\in I$, there exist one-parameter subgroups $x_{i}, y_{i}\colon \mathbb{C}\to G$ of $G$ such that 
\begin{align*}
hx_{i}(t)h^{-1}&=x_i(h^{\alpha_i}t),& dx_{i}\colon \mathbb{C}&\xrightarrow{\sim}\mathbb{C}e_i,\\
hy_{i}(t)h^{-1}&=y_i((h^{\alpha_i})^{-1} t),& dy_{i}\colon \mathbb{C}&\xrightarrow{\sim}\mathbb{C}f_i
\end{align*}
for $h\in H$ and $t\in\mathbb{C}$. Here $h^{\alpha_i} \in \mathbb{C}$ is determined by $\mathrm{Ad}(h)e_i=h^{\alpha_i}e_i$ through the adjoint action $\mathrm{Ad}$ of $G$ on $\mathfrak{g}$, and $dx_{i}$ and $dy_{i}$ are the derivations of $x_{i}$ and $y_{i}$, respectively. In this paper, we normalize them so that $dx_{i}(1)=e_i$ and $dy_i(1)=f_i$. Then there exists a homomorphism $\varphi_i\colon SL_2(\mathbb{C})\to G$ such that 
\begin{align*}
\begin{pmatrix}
1&t\\
0&1
\end{pmatrix}&\mapsto x_i(t),&
\begin{pmatrix}
1&0\\
t&1
\end{pmatrix}&\mapsto y_i(t).
\end{align*}
For $i\in I$, set 
\[
\overline{s}_i\coloneqq\varphi_i\left(\begin{pmatrix}
0&-1\\
1&0
\end{pmatrix}\right).
\]
Moreover, for $(i_1,\dots, i_m)\in R(w)$, set 
\[
\overline{w} \coloneqq \overline{s}_{i_1}\cdots \overline{s}_{i_m} \in N_G(H).
\]
Then the element $\overline{w}$ does not depend on the choice of a reduced word $(i_1,\dots, i_m)$ for $w$, and $\overline{w}$ is a lift for $w\in W=N_G(H)/H$. 
 \begin{defi}[{see, for instance, \cite[Section I\hspace{-.1em}I.13.3]{Jan} and \cite[Definition 7.1.13]{Kum}}]
Denote by $X(w)$ for $w \in W$ the Zariski closure of $B \widetilde{w} B/B$ in $G/B$, where $\widetilde{w} \in N_G(H)$ denotes a lift for $w$. The closed subvariety $X(w)$ of $G/B$ is called the {\it Schubert variety} corresponding to $w$.
\end{defi} 

The Schubert variety $X(w)$ is a normal projective variety of complex dimension $\ell(w)$ (see, for instance, \cite[Sections I\hspace{-.1em}I.13.3, I\hspace{-.1em}I.14.15]{Jan}). If $w$ is the longest element $w_0$ in $W$, then the Schubert variety $X(w_0)$ coincides with the full flag variety $G/B$. 

For $\lambda \in P_+$, we define a line bundle $\mathcal{L}_\lambda$ on $G/B$ as follows:
\[
\mathcal{L}_\lambda \coloneqq (G \times \mathbb{C})/B,
\] 
where $B$ acts on $G \times \mathbb{C}$ from the right by $(g, c) \cdot b \coloneqq (g b, \lambda(b) c)$ for $g \in G$, $c \in \mathbb{C}$, and $b \in B$. By restricting this bundle, we obtain a line bundle on $X(w)$, which we denote by the same symbol $\mathcal{L}_\lambda$. 

\begin{prop}[{see, for instance, \cite[Sections I\hspace{-.1em}I.4.4, I\hspace{-.1em}I.8.5]{Jan}}]\label{p:line_bundle_positivity}
For $\lambda \in P_+$, the following hold.
\begin{enumerate}
\item[{\rm (1)}] The line bundle $\mathcal{L}_\lambda$ on $G/B$ is generated by global sections. 
\item[{\rm (2)}] The line bundle $\mathcal{L}_\lambda$ on $G/B$ is very ample if and only if $\lambda$ is a regular dominant integral weight, that is, $\langle \lambda, h_i\rangle \in \z_{>0}$ for all $i \in I$. 
\end{enumerate}
\end{prop}

By \cref{p:line_bundle_positivity}, the line bundle $\mathcal{L}_\lambda$ on $X(w)$ is generated by global sections for all $w \in W$ and $\lambda \in P_+$; in addition, it is very ample if $\lambda$ is regular. 

For $\lambda\in P_{+}$, let $V(\lambda)$ be the (rational) irreducible $G$-module over $\c$ with highest weight $\lambda$. We fix a highest weight vector $v_{\lambda}$ of $V(\lambda)$. Set 
\begin{align*}
v_{w\lambda} \coloneqq \overline{w} v_{\lambda}
\end{align*}
for $w\in W$. The {\it Demazure module} $V_w(\lambda)$ corresponding to $w \in W$ is defined to be the $B$-submodule of $V(\lambda)$ given by 
\[
V_w(\lambda) \coloneqq \sum_{b \in B} \mathbb{C} b v_{w\lambda}.
\] 
From the Borel--Weil type theorem (see, for instance, \cite[Corollary 8.1.26]{Kum}), we know that the space $H^0(G/B, \mathcal{L}_\lambda)$ (resp., $H^0(X(w), \mathcal{L}_\lambda)$) of global sections is a $G$-module (resp., a $B$-module) which is isomorphic to the dual module $V(\lambda)^\ast \coloneqq {\rm Hom}_\mathbb{C}(V(\lambda), \mathbb{C})$ (resp., $V_w (\lambda)^\ast \coloneqq {\rm Hom}_\mathbb{C}(V_w(\lambda), \mathbb{C})$). 
Here the natural right actions of a group on the dual spaces of modules are transformed into the left actions by the inverse map of the group. We fix a lowest weight vector $\tau_\lambda \in H^0(G/B, \mathcal{L}_\lambda)$. By restricting this section, we obtain a section in $H^0(X(w), \mathcal{L}_\lambda)$, which we denote by the same symbol $\tau_\lambda$. 

Let $B^- \subset G$ denote the opposite Borel subgroup, and $U^-$ the unipotent radical of $B^-$ with Lie algebra $\mathfrak{u}^-$. We regard $U^-$ as an affine open subvariety of $G/B$ by the following open embedding: 
\[U^- \hookrightarrow G/B,\ u \mapsto u \bmod B.\] 
Then the intersection $U^- \cap X(w)$ is thought of as an affine open subvariety of $X(w)$; in particular, we have $\mathbb{C}(X(w)) = \mathbb{C}(U^- \cap X(w))$. For ${\bm i} = (i_1, \ldots, i_m) \in R(w)$, we see by \cite[Ch.\ I\hspace{-.1em}I.13]{Jan} that the morphism \[\c^m \rightarrow U^- \cap X(w),\ (t_1, \ldots, t_m) \mapsto y_{i_1}(t_1) \cdots y_{i_m}(t_m) \bmod B\] is birational. Hence the function field $\mathbb{C}(X(w)) = \c (U^- \cap X(w))$ is identified with the field $\mathbb{C}(t_1, \ldots, t_m)$ of rational functions in $t_1, \ldots, t_m$. 

\begin{defi}\label{d:lowest_term_valuation_Schubert}
Consider two lexicographic orders $<$ and $\prec$ on $\mathbb{Z}^m$, defined as follows: 
\[
(a_1, \ldots, a_m) < (a_1 ^\prime, \ldots, a_m ^\prime)\quad (\text{resp., }(a_1, \ldots, a_m) \prec (a_1 ^\prime, \ldots, a_m ^\prime)) 
\]
if and only if there exists $1 \le k \le m$ such that 
\[
a_1 = a_1 ^\prime, \ldots, a_{k-1} = a_{k-1} ^\prime,\ a_k < a_k ^\prime\quad (\text{resp., }a_m = a_m ^\prime, \ldots, a_{k+1} = a_{k+1} ^\prime,\ a_k < a_k ^\prime).
\]
Then we define valuations $v_{\bm i} ^{\rm low}$ and $\tilde{v}_{\bm i} ^{\rm low}$ on $\mathbb{C}(X(w))$ to be $v^{\rm low} _{<}$ and $v^{\rm low} _{\prec}$ on $\mathbb{C}(t_1, \ldots, t_m)$, respectively (see \cref{ex:lowest_term_valuation}). 
\end{defi}

\begin{rem}
The definition of $\tilde{v}_{\bm i} ^{\rm low}$ is slightly different from the one in \cite[Section 2]{FO} because the order of coordinates of its values is reversed.
\end{rem}

\begin{rem}[{see \cite[Section 2]{FO}}]
The valuation $v_{\bm i} ^{\rm low}$ coincides with the one given by counting the orders of zeros or poles along a specific sequence of Schubert varieties. A similar result also holds for the valuation $\tilde{v}_{\bm i} ^{\rm low}$ after reversing the order of coordinates of its values.
\end{rem}

Let $\Delta_{\bm i} (\lambda)$ (resp., $\widetilde{\Delta}_{\bm i} (\lambda)$) denote the string polytope (resp., the Nakashima--Zelevinsky polytope) associated with ${\bm i} \in R(w)$ and $\lambda \in P_+$. 
We do not recall the original definitions of these polytopes since they are not needed in the present paper (see \cite[Section 1]{Lit}, \cite[Definition 2.15]{FN}, \cite[Definition 3.24]{FO}, and \cite[Definition 3.9]{Fuj} for their precise definitions).
However, we note that these polytopes are defined from a representation-theoretic structure $\mathcal{B}(\lambda)$, called the \emph{Kashiwara crystal basis}, for $V(\lambda)$, which is equipped with specific operators $\tilde{e}_i, \tilde{f}_i \colon \mathcal{B}(\lambda) \rightarrow \mathcal{B}(\lambda) \cup \{0\}$, called the \emph{Kashiwara operators}, for each $i \in I$. 
Roughly speaking, the string polytope $\Delta_{\bm i} (\lambda)$ is defined from a specific parametrization of $\mathcal{B}(\lambda)$ given by calculating lengths of certain strings, or more precisely by arranging numbers of the form $\max\{a \in \z_{\geq 0} \mid \tilde{e}_i^a b \neq 0\}$ for some $b \in \mathcal{B}(\lambda)$. 
The Nakashima--Zelevinsky polytope $\widetilde{\Delta}_{\bm i} (\lambda)$ is defined similarly by using the operator $\tilde{e}_i^\ast$ instead of $\tilde{e}_i$, where $\tilde{e}_i^\ast$ is the twist of $\tilde{e}_i$ by Kashiwara's involution $\ast$. 

\begin{rem}
The definition of Nakashima--Zelevinsky polytopes $\widetilde{\Delta}_{\bm i} (\lambda)$ in \cite[Definition 3.9]{Fuj} is slightly different from the one in \cite[Definition 3.24 (2)]{FO} because the order of coordinates is reversed. In the present paper, we use the definition in \cite[Definition 3.9]{Fuj}.
\end{rem}

The authors \cite{FO} proved that $\Delta(X(w), \mathcal{L}_\lambda, v_{\bm i} ^{\rm low}, \tau_\lambda)$ and $\Delta(X(w), \mathcal{L}_\lambda, \tilde{v}_{\bm i} ^{\rm low}, \tau_\lambda)$ coincide with the representation-theoretic polytopes mentioned above as follows.

\begin{thm}[{see \cite[Propositions 3.28, 3.29 and Corollaries 5.3, 5.4]{FO}}]\label{t:NO_body_crystal_basis}
Let $w \in W$, $\lambda \in P_+$, and ${\bm i} \in R(w)$.
\begin{enumerate}
\item[{\rm (1)}] The Newton--Okounkov bodies $\Delta(X(w), \mathcal{L}_\lambda, v_{\bm i} ^{\rm low}, \tau_\lambda)$ and $\Delta(X(w), \mathcal{L}_\lambda, \tilde{v}_{\bm i} ^{\rm low}, \tau_\lambda)$ are both rational convex polytopes.
\item[{\rm (2)}] The Newton--Okounkov polytope $\Delta(X(w), \mathcal{L}_\lambda, v_{\bm i} ^{\rm low}, \tau_\lambda)$ coincides with the Nakashima--Zelevinsky polytope $\widetilde{\Delta}_{\bm i} (\lambda)$.
\item[{\rm (3)}] The Newton--Okounkov polytope $\Delta(X(w), \mathcal{L}_\lambda, \tilde{v}_{\bm i} ^{\rm low}, \tau_\lambda)$ coincides with the string polytope $\Delta_{\bm i} (\lambda)$.
\end{enumerate}
\end{thm}

\begin{rem}
Such relation between Newton--Okounkov bodies and string polytopes (resp., Nakashima--Zelevinsky polytopes) was previously given by Kaveh \cite{Kav} (resp., by the first named author and Naito \cite{FN}) using a different kind of valuation.
\end{rem}

\begin{ex}\label{ex:basic_example}
Let $G = SL_3(\c)$ (of type $A_2$), $I = \{1, 2\}$, $\lambda = \alpha_1 + \alpha_2 \in P_+$, and ${\bm i} = (1, 2, 1) \in R(w_0)$. By standard monomial theory (see, for instance, \cite[Section 2]{Ses}), the $\c$-subspace $\{\sigma/\tau_\lambda \mid \sigma \in H^0(G/B, \mathcal{L}_\lambda)\}$ of $\c(G/B) = \c(t_1, t_2, t_3)$ is spanned by \[\{1, t_1 + t_3, t_2, t_1 t_2, t_2 t_3, t_1 t_2 (t_1 + t_3), t_2 ^2 t_3, t_1 t_2 ^2 t_3\}.\] Now we obtain the following list.\\

\begin{table}[h]
\begin{tabular}{|c||c|c|c|c|c|c|c|c|} \hline
valuation & $1$ & $t_1 + t_3$ & $t_2$ & $t_1 t_2$\\ \hline
$v_{\bm i} ^{\rm low}$ & $(0, 0, 0)$ & $\ \ (0, 0, 1)$ & $\ \ (0, 1, 0)$ & $\ \ (1, 1, 0)$\\ \hline
$\tilde{v}_{\bm i} ^{\rm low}$ & $(0, 0, 0)$ & $\ \ (1, 0, 0)$ & $\ \ (0, 1, 0)$ & $\ \ (1, 1, 0)$\\ \hline \hline
valuation & $t_2 t_3$ & $t_1 t_2 (t_1 + t_3)$ & $t_2 ^2 t_3$ & $t_1 t_2 ^2 t_3$\\ \hline
$v_{\bm i} ^{\rm low}$ & $\ \ (0, 1, 1)$ & $\ \ (1, 1, 1)$ & $\ \ (0, 2, 1)$ & $\ \ (1, 2, 1)$\\ \hline
$\tilde{v}_{\bm i} ^{\rm low}$ & $\ \ (0, 1, 1)$ & $\ \ (2, 1, 0)$ & $\ \ (0, 2, 1)$ & $\ \ (1, 2, 1)$\\ \hline
\end{tabular}
\end{table}

\noindent The Newton--Okounkov bodies $\Delta(G/B, \mathcal{L}_\lambda, v_{\bm i} ^{\rm low}, \tau_\lambda)$ and $\Delta(G/B, \mathcal{L}_\lambda, \tilde{v}_{\bm i} ^{\rm low}, \tau_\lambda)$ coincide with the convex hulls of the corresponding eight points in the list above, respectively; see Figures \ref{figure_ex_NZ_polytope}, \ref{figure_ex_string_polytope}.

\begin{figure}[!ht]
\begin{center}
   \includegraphics[width=8.0cm,bb=60mm 180mm 150mm 230mm,clip]{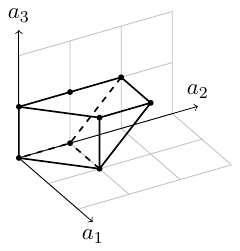}
	\caption{The Newton--Okounkov body $\Delta(G/B, \mathcal{L}_\lambda, v_{\bm i} ^{\rm low}, \tau_\lambda)$.}
	\label{figure_ex_NZ_polytope}
\end{center}
\end{figure}

\begin{figure}[!ht]
\begin{center}
   \includegraphics[width=8.0cm,bb=60mm 180mm 150mm 230mm,clip]{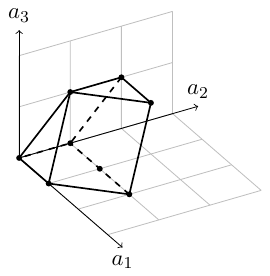}
	\caption{The Newton--Okounkov body $\Delta(G/B, \mathcal{L}_\lambda, \tilde{v}_{\bm i} ^{\rm low}, \tau_\lambda)$.}
	\label{figure_ex_string_polytope}
\end{center}
\end{figure}
\end{ex}

We know the following property of the semigroups and the closed convex cones; here we use Gordan's lemma (see, for instance, \cite[Proposition 1.2.17]{CLS}). 

\begin{thm}[{see the argument given before \cite[Corollary 5.6]{FO}}]\label{t:saturatedness_of_semigroups}
Let $w \in W$, $\lambda \in P_+$, and ${\bm i} \in R(w)$.
\begin{enumerate}
\item[{\rm (1)}] The real closed cones $C(X(w), \mathcal{L}_\lambda, v_{\bm i} ^{\rm low}, \tau_\lambda)$ and $C(X(w), \mathcal{L}_\lambda, \tilde{v}_{\bm i} ^{\rm low}, \tau_\lambda)$ are both rational convex polyhedral cones.
\item[{\rm (2)}] The equalities
\begin{align*}
&S(X(w), \mathcal{L}_\lambda, v_{\bm i} ^{\rm low}, \tau_\lambda) = C(X(w), \mathcal{L}_\lambda, v_{\bm i} ^{\rm low}, \tau_\lambda) \cap (\z_{>0} \times \z^m)\ {\it and}\\
&S(X(w), \mathcal{L}_\lambda, \tilde{v}_{\bm i} ^{\rm low}, \tau_\lambda) = C(X(w), \mathcal{L}_\lambda, \tilde{v}_{\bm i} ^{\rm low}, \tau_\lambda) \cap (\z_{>0} \times \z^m)
\end{align*}
hold. In particular, the semigroups $S(X(w), \mathcal{L}_\lambda, v_{\bm i} ^{\rm low}, \tau_\lambda)$ and $S(X(w), \mathcal{L}_\lambda, \tilde{v}_{\bm i} ^{\rm low}, \tau_\lambda)$ are both finitely generated and saturated.
\end{enumerate}
\end{thm}

For ${\bm i} \in R(w)$, we define $C_{\bm i} \subset \r^m$ and $\widetilde{C}_{\bm i} \subset \r^m$ by 
\begin{align*}
C_{\bm i} \coloneqq \bigcup_{\lambda \in P_+} \Delta(X(w), \mathcal{L}_\lambda, v_{\bm i} ^{\rm low}, \tau_\lambda), \\
\widetilde{C}_{\bm i} \coloneqq \bigcup_{\lambda \in P_+} \Delta(X(w), \mathcal{L}_\lambda, \tilde{v}_{\bm i} ^{\rm low}, \tau_\lambda),
\end{align*}
respectively. 

\begin{prop}[{see \cite[Corollaries 3.20 and 5.5]{FO}}]\label{p:string_cone_lattice_points}
Let $w \in W$, and ${\bm i} \in R(w)$. Then the following equalities hold:
\begin{align*}
&C_{\bm i} \cap \z^m = v_{\bm i} ^{\rm low} (\c[U^- \cap X(w)] \setminus \{0\}),\ {\it and}\\
&\widetilde{C}_{\bm i} \cap \z^m = \tilde{v}_{\bm i} ^{\rm low} (\c[U^- \cap X(w)] \setminus \{0\}).
\end{align*}
\end{prop}

Write ${\bm a}^{\rm op} \coloneqq (a_m, \ldots, a_1)$ for an element ${\bm a} = (a_1, \ldots, a_m) \in \mathbb{R}^m$, and $V^{\rm op} \coloneqq \{{\bm a}^{\rm op} \mid {\bm a} \in V\}$ for a subset $V \subset \mathbb{R}^m$. By \cref{t:NO_body_crystal_basis} and \cite[Section 2.4]{NZ}, we obtain the following relation with string cones (see \cite[Section 1]{Lit} and \cite[Section 3.2]{BZ2} for more details on string cones).  

\begin{cor}\label{c:relation_with_string_cones}
Let ${\bm i} = (i_1, \ldots, i_m) \in R(w)$, and define ${\bm i}^{\rm op} \in R(w^{-1})$ by ${\bm i}^{\rm op} \coloneqq (i_m, \ldots, i_1)$.
\begin{enumerate}
\item[{\rm (1)}] The set $\widetilde{C}_{\bm i}$ coincides with the string cone associated with ${\bm i}$.
\item[{\rm (2)}] The set $C_{\bm i} ^{\rm op}$ coincides with the string cone associated with ${\bm i}^{\rm op}$.
\item[{\rm (3)}] The equality $\widetilde{C}_{{\bm i}^{\rm op}} = C_{\bm i} ^{\rm op}$ holds.
\end{enumerate}
\end{cor}

\section{Cluster algebras and valuations}\label{s:Clusterval}

We first recall the definitions of $\mathcal{A}$-cluster varieties and (upper) cluster algebras of geometric type, following \cite{BFZ2, FG, FZ:ClusterIV, GHK}. For a $\z$-lattice $L$ of finite rank, we
set $L_{\r} \coloneqq L \otimes_{\z} \r$, and define an algebraic torus $T_L$ by 
\[T_L \coloneqq L \otimes_{\z} \c^\times = \Spec(\c[{\rm Hom}_{\z}(L, \z)]).\]
We fix the following data $\Gamma$, called \emph{fixed data}:
\begin{itemize}
\item a $\z$-lattice $N$ of finite rank with a skew-symmetric bilinear form 
\[\{\cdot, \cdot\} \colon N \times N \rightarrow \q,\]
\item a saturated sublattice $N_{\rm uf} \subset N$, called an \emph{unfrozen sublattice},
\item a finite set $J$ with $|J| = \rank N$ and a subset $J_{\rm uf} \subset J$ with $|J_{\rm uf}| = \rank N_{\rm uf}$,
\item positive integers $d_j$, $j \in J$, with greatest common divisor $1$,
\item a sublattice $N^\circ \subset N$ of finite index such that 
\[\{N_{\rm uf}, N^\circ\} \subset \z,\quad \{N, N_{\rm uf} \cap N^\circ\} \subset \z,\]
\item $M \coloneqq {\rm Hom}_\z (N, \z)$, $M^\circ \coloneqq {\rm Hom}_\z (N^\circ, \z)$.
\end{itemize}
Then a \emph{seed} $\widehat{\bf s} = (e_j)_{j \in J}$ for $\Gamma$ is a $\z$-basis of $N$ such that $(e_j)_{j \in J_{\rm uf}}$ is a $\z$-basis of $N_{\rm uf}$, and $(d_j e_j)_{j \in J}$ is a $\z$-basis of $N^\circ$. Given a seed $\widehat{\bf s} = (e_j)_{j \in J}$, we obtain the following:
\begin{itemize}
\item $\widehat{\varepsilon} = (\varepsilon_{i, j})_{i, j \in J} \in \Mat_{J \times J}(\mathbb{Q})$ given by $\varepsilon_{i, j} \coloneqq \{e_i, e_j\} d_j$ for $i, j \in J$,
\item the dual basis $(e_j ^\ast)_{j \in J}$ for $M$,
\item a $\z$-basis $(f_j)_{j \in J}$ of $M^\circ$ defined by $f_j \coloneqq d_j ^{-1} e_j ^\ast$ for $j \in J$.
\end{itemize}
Note that $\varepsilon_{i, j} \in \z$ unless $i, j \in J_{\rm fr} \coloneqq J \setminus J_{\rm uf}$. We set 
\[\mathcal{A}_{\widehat{\bf s}} \coloneqq T_{N^\circ} = \Spec(\c[M^\circ]),\]
and define $A_{j; \widehat{\bf s}}$, $j \in J$, to be the coordinates on $\mathcal{A}_{\widehat{\bf s}}$ corresponding to the basis elements $f_j \in M^\circ$, $j \in J$. Let us write $[a]_+ \coloneqq \max\{a, 0\}$ for $a \in \mathbb{R}$. We define the \emph{mutation} $\mu_k (\widehat{\bf s})$ in direction $k \in J_{\rm uf}$ to be another $\z$-basis $\widehat{\bf s}^\prime = (e_j ^\prime)_{j \in J}$ of $N$, where 
\[e_j ^\prime \coloneqq 
\begin{cases}
e_j + [\varepsilon_{j, k}]_+ e_k &(j \neq k),\\
-e_k &(j = k).
\end{cases}\]
Then $\mu_k (\widehat{\bf s}) = \widehat{\bf s}^\prime$ is also a seed for $\Gamma$. Note that $\mu_k\mu_k (\widehat{\bf s})\neq \widehat{\bf s}$ in general. By the mutation $\mu_k$, the matrix $\widehat{\varepsilon} = (\varepsilon_{i, j})_{i, j \in J}$ changes as follows: if we denote the matrix for $\widehat{\bf s}^\prime$ by $\widehat{\varepsilon}^\prime = (\varepsilon_{i, j} ^\prime)_{i, j \in J}$, then
\begin{equation}\label{eq_mutation_for_varepsilon}
\begin{aligned}
\varepsilon^\prime _{i, j} = \begin{cases}
-\varepsilon_{i, j}&\ \text{if}\ i=k\ \text{or}\ j=k,\\
\varepsilon_{i, j} + \sgn(\varepsilon_{k, j})[\varepsilon_{i, k} \varepsilon_{k, j}]_+&\ \text{otherwise}. 
\end{cases}
\end{aligned}
\end{equation}
We define a birational map $\mu_k \colon \mathcal{A}_{\widehat{\bf s}} \dashrightarrow \mathcal{A}_{\mu_k (\widehat{\bf s})}$, called the \emph{birational mutation}, as follows:
\begin{equation}\label{eq_mutation_for_cluster_variable}
\begin{aligned}
A_i ^\prime \coloneqq \begin{cases}
\displaystyle \frac{\prod_{j \in J} A_j ^{[\varepsilon_{k, j}]_+} + \prod_{j \in J} A_j ^{[-\varepsilon_{k, j}]_+}}{A_k} &\ \text{if}\ i = k,\\
A_i &\ \text{otherwise}
\end{cases}
\end{aligned}
\end{equation}
for $i \in J$, where we denote $A_{j; \widehat{\bf s}}$ (resp., $\mu_k ^\ast A_{j; \mu_k (\widehat{\bf s})}$) by $A_j$ (resp., by $A_j ^\prime$) for $j \in J$. 

Let $\widehat{\mathbb{T}}$ be the oriented rooted tree such that each vertex has $|J_{\rm uf}|$ outgoing edges labeled by $J_{\rm uf}$ so that the $|J_{\rm uf}|$ outgoing edges emanating from each vertex receive different labels. Denote by $t_0$ the root of $\widehat{\mathbb{T}}$. We fix a seed $\widehat{\bf s}_{t_0}$ for $\Gamma$. For each $t \in \widehat{\mathbb{T}}$, there exists a unique simple path starting at $t_0$ and ending at $t$, which is of the form $t_0\xrightarrow{k_1}t_1\xrightarrow{k_2}\cdots \xrightarrow{k_{\ell}}t_{\ell}=t$. Then we define a seed $\widehat{\bf s}_t$ for $\Gamma$ by 
\begin{align}
\widehat{\bf s}_t \coloneqq \mu_{t_0, t}(\widehat{\bf s}_{t_0}),\text{ where }\mu_{t_0, t}\coloneqq \mu_{k_{\ell}} \cdots \mu_{k_{2}} \mu_{k_{1}},\label{eq:assignment}
\end{align}
which gives an assignment $\{\widehat{\bf s}_t\}_{t \in \widehat{\mathbb{T}}}$ of a seed $\widehat{\bf s}_t$ for $\Gamma$ to each vertex $t \in \widehat{\mathbb{T}}$. For $t \in \widehat{\mathbb{T}}$, denote the torus $\mathcal{A}_{\widehat{\bf s}_t}$ by $\mathcal{A}_t$. Then we obtain a scheme 
\[
\mathcal{A} = \bigcup_{t \in \widehat{\mathbb{T}}} \mathcal{A}_t = \bigcup_{t \in \widehat{\mathbb{T}}} \Spec (\mathbb{C}[A_{j; \widehat{\bf s}_t}^{\pm 1} \mid j \in J])
\]
by gluing tori $\mathcal{A}_t = \Spec (\mathbb{C}[A_{j; \widehat{\bf s}_t}^{\pm 1} \mid j \in J])$ via the birational mutation $\mu_k$ defined above (see \cite[Proposition 2.4]{GHK} for careful treatment of gluing). This scheme $\mathcal{A}$ is called an \emph{$\mathcal{A}$-cluster variety}.

Let $\c(\mathcal{A}_{t_0})=\c(A_{j; \widehat{\bf s}_{t_0}} \mid j \in J)$ be the function field of $\mathcal{A}_{t_0}$. For $t \in \widehat{\mathbb{T}}$ and $j\in J$, write 
\[
A_{j; t}\coloneqq \mu_{t_0, t}^{\ast}(A_{j; \widehat{\bf s}_t})\in\c(\mathcal{A}_{t_0}),
\]
where we use the notation in \eqref{eq:assignment}. 

\begin{defi}[{see \cite[Definitions 1.6 and 1.11]{BFZ2}}]
We set 
\[
{\rm up}(\mathcal{A})\coloneqq \bigcap_{t \in \widehat{\mathbb{T}}} \c[A_{j; t} ^{\pm 1} \mid j \in J] \subset \c(\mathcal{A}_{t_0}).
\]
This is called an \emph{upper cluster algebra of geometric type}, which is isomorphic to the ring of regular functions on $\mathcal{A}$ through the restriction of the domain of functions. The (\emph{ordinary}) \emph{cluster algebra} ${\rm ord}(\mathcal{A})$ \emph{of geometric type} is defined to be the $\mathbb{C}$-subalgebra of $\c(\mathcal{A}_{t_0})$ generated by $\{A_{j; t} \mid t \in \widehat{\mathbb{T}},\ j \in J_{\rm uf}\} \cup \{A_{j; t}^{\pm 1} \mid t \in \widehat{\mathbb{T}},\ j \in J_{\rm fr}\}$. Note that $A_{j; t}=A_{j; t'}$ for all $t, t' \in \widehat{\mathbb{T}}$ if $j \in J_{\rm fr}$.
\end{defi}
\begin{thm}[{\cite[Theorem 3.14]{GHK}}]
The canonical map $\mathcal{A} \rightarrow \Spec({\rm up}(\mathcal{A}))$ is an open immersion. In particular, $\mathcal{A}$ is separated.
\end{thm}

In this paper, we mainly deal with the algebras ${\rm ord}(\mathcal{A})$ and ${\rm up}(\mathcal{A})$. Indeed, we can define these algebras without using the whole of the fixed data and seeds, following Fomin--Zelevinsky's approach \cite{FZ:ClusterI,FZ:ClusterIV}. Since this approach is more convenient in some cases, we also recall it briefly. 

Let $\mathcal{F} \coloneqq \mathbb{C}(z_j \mid j \in J)$ be the field of rational functions in $|J|$ variables. Then a \emph{Fomin--Zelevinsky seed} (\emph{FZ-seed} for short) ${\bf s} = (\mathbf{A}, \varepsilon)$ of $\mathcal{F}$ is a pair of 
\begin{itemize}
\item a $J$-tuple $\mathbf{A} = (A_j)_{j \in J}$ of elements of $\mathcal{F}$,\ {\rm and}
\item $\varepsilon = (\varepsilon_{i, j})_{i \in J_{\rm uf}, j \in J} \in \Mat_{J_{\rm uf} \times J}(\mathbb{Z})$
\end{itemize}
such that
\begin{itemize}
	\item[(i)] $\mathbf{A}$ forms a free generating set of $\mathcal{F}$, and
   \item[(ii)] the $J_{\rm uf} \times J_{\rm uf}$-submatrix $\varepsilon^{\circ}$ of $\varepsilon$ is skew-symmetrizable, that is, there exists $(d_i)_{i \in J_{\rm uf}} \in \mathbb{Z}_{>0} ^{J_{\rm uf}}$ such that $\varepsilon_{i, j} d_i = -\varepsilon_{j, i} d_j$ for all $i, j \in J_{\rm uf}$.  
\end{itemize}
The matrix $\varepsilon$ is called the \emph{exchange matrix} of ${\bf s}$, and the submatrix $\varepsilon^{\circ}$ is called the \emph{principal part} of $\varepsilon$. 

\begin{rem}
Our exchange matrix $\varepsilon$ is transposed to the one in \cite[Section 2]{FZ:ClusterIV}.
\end{rem}

Let ${\bf s} = (\mathbf{A}, \varepsilon) = ((A_j)_{j \in J}, (\varepsilon_{i, j})_{i \in J_{\rm uf}, j \in J})$ be an FZ-seed of $\mathcal{F}$. For $k \in J_{\rm uf}$, the \emph{mutation} $\mu_k ({\bf s}) = (\mu_k (\mathbf{A}), \mu_k (\varepsilon))$ in direction $k$ is defined by the formulas \eqref{eq_mutation_for_varepsilon} and \eqref{eq_mutation_for_cluster_variable}, where we write $\mu_k (\mathbf{A}) = (A^\prime _j)_{j \in J}$ and $\mu_k (\varepsilon) = (\varepsilon^\prime _{i, j})_{i \in J_{\rm uf}, j \in J}$. Then $\mu_k ({\bf s})$ is again an FZ-seed of $\mathcal{F}$. This time we have $\mu_k \mu_k ({\bf s}) = {\bf s}$. Two FZ-seeds ${\bf s}$ and ${\bf s}^\prime$ are said to be \emph{mutation equivalent} if there exists a sequence $(k_1, k_2, \ldots, k_{\ell})$ in $J_{\rm uf}$ such that 
\[
\mu_{k_{\ell}} \cdots \mu_{k_{2}} \mu_{k_{1}} ({\bf s}) = {\bf s}^\prime.
\]
In this case, we write ${\bf s} \sim {\bf s}^\prime$, which gives an equivalence relation. 

Let $\mathbb{T}$ be the $|J_{\rm uf}|$-regular (neither oriented nor rooted) tree whose edges are labeled by $J_{\rm uf}$ so that the $|J_{\rm uf}|$-edges emanating from each vertex receive different labels. If $t, t^\prime \in \mathbb{T}$ are joined by an edge labeled by $k \in J_{\rm uf}$, then we write $t \overset{k}{\text{---}} t^\prime$. \emph{A cluster pattern} $\mathcal{S} = \{{\bf s}_t\}_{t \in \mathbb{T}} = \{(\mathbf{A}_t, \varepsilon_t)\}_{t \in \mathbb{T}}$ is an assignment of an FZ-seed ${\bf s}_t = (\mathbf{A}_t, \varepsilon_t)$ of $\mathcal{F}$ to each vertex $t \in \mathbb{T}$ such that $\mu_k ({\bf s}_t) = {\bf s}_{t^\prime}$ whenever $t \overset{k}{\text{---}} t'$. For a cluster pattern $\mathcal{S} = \{{\bf s}_t = (\mathbf{A}_t, \varepsilon_t)\}_{t \in \mathbb{T}}$, write 
\[
\mathbf{A}_t = (A_{j; t})_{j \in J}, \quad \varepsilon_t = (\varepsilon_{i, j} ^{(t)})_{i \in J_{\rm uf}, j \in J}.
\]

We define the (\emph{ordinary}) \emph{cluster algebra} $\mathscr{A}(\mathcal{S})$ \emph{of geometric type} to be the $\mathbb{C}$-subalgebra of $\mathcal{F}$ generated by $\mathcal{V} \coloneqq \{A_{j; t} \mid t \in \mathbb{T},\ j \in J\}$ and $\{A_{j; t}^{-1} \mid t \in \mathbb{T},\ j \in J_{\rm fr}\}$. The elements of $\mathcal{V}$ are called \emph{cluster variables}. They are divided into two parts: $\mathcal{V}_{\rm uf} \coloneqq \{A_{j; t} \mid t \in \mathbb{T},\ j \in J_{\rm uf}\}$ and $\mathcal{V}_{\rm fr} \coloneqq \{A_{j; t} \mid t \in \mathbb{T},\ j \in J_{\rm fr}\}$. The elements of $\mathcal{V}_{\rm uf}$ (resp., $\mathcal{V}_{\rm fr}$) are called \emph{unfrozen variables} (resp., \emph{frozen variables}). For $j \in J_{\rm fr}$, the frozen variable $A_{j; t}$ is independent of the choice of $t \in \mathbb{T}$, which is also denoted by $A_j$. For each $t \in \mathbb{T}$, a set $\{A_{j; t}\mid j \in J\}$ of cluster variables is called the \emph{cluster} at $t$ (the frozen variables are included in the clusters in this paper). A \emph{cluster monomial} is a monomial in cluster variables all of which belong to the same cluster $\{A_{j; t}\mid j \in J\}$. Moreover, we set 
\[
\mathscr{U}(\mathcal{S})\coloneqq \bigcap_{t \in \mathbb{T}} \c[A_{j; t} ^{\pm 1} \mid j \in J], 
\]
which is called the \emph{upper cluster algebra of geometric type}. We sometimes write $\mathscr{A}(\mathcal{S})$ (resp., $\mathscr{U}(\mathcal{S})$) as $\mathscr{A}({\bf s}_{t_0})$ or $\mathscr{A}(\mathbf{A}_{t_0}, \varepsilon_{t_0})$ (resp., $\mathscr{U}({\bf s}_{t_0})$ or $\mathscr{U}(\mathbf{A}_{t_0}, \varepsilon_{t_0})$) for some $t_0 \in \mathbb{T}$ since a cluster pattern $\mathcal{S} = \{{\bf s}_t\}_{t \in \mathbb{T}}$ is determined by the FZ-seed ${\bf s}_{t_0} = (\mathbf{A}_{t_0}, \varepsilon_{t_0})$. We usually fix $t_0 \in \mathbb{T}$, and construct a cluster pattern $\mathcal{S} = \{{\bf s}_t\}_{t \in \mathbb{T}}$ from an FZ-seed ${\bf s}_{t_0}$. In this case, ${\bf s}_{t_0}$ is called the \emph{initial FZ-seed}.

\begin{thm}[{\cite[Theorem 3.1]{FZ:ClusterI}}]\label{t:laurentpheno}
Let $\mathcal{S}$ be a cluster pattern. Then it follows that
\[
\mathscr{A}(\mathcal{S}) \subset \mathbb{C}[A_{j; t}^{\pm 1} \mid j \in J] 
\]
for all $t \in \mathbb{T}$; this property is called the \emph{Laurent phenomenon}. In particular, $\mathscr{A}(\mathcal{S})$ is included in the upper cluster algebra $\mathscr{U}(\mathcal{S})$.
\end{thm}

Let $\{\widehat{\bf s}_t\}_{t \in \widehat{\mathbb{T}}}$ be the family of seeds for $\Gamma$ defined above. Then it induces a cluster pattern $\mathcal{S} = \{{\bf s}_{t'}\}_{t' \in \mathbb{T}}$ in the following way. Fix $t'_0 \in \mathbb{T}$. There uniquely exists a surjective map $\Pi\colon \{\text{vertices of }\widehat{\mathbb{T}}\}\twoheadrightarrow \{\text{vertices of }\mathbb{T}\}$ such that, for $t\in \widehat{\mathbb{T}}$ with  $t_0\xrightarrow{k_1}t_1\xrightarrow{k_2}\cdots \xrightarrow{k_{\ell}}t_{\ell}=t$, 
\[
\Pi(t)=t',\text{ where }t'_0 \overset{k_1}{\text{---}} t'_1  \overset{k_2}{\text{---}}\cdots \overset{k_{\ell}}{\text{---}}t'_{\ell}=t'. 
\]
Set $\mathcal{F}=\c(\mathcal{A}_{t_0})$. For $t' \in \mathbb{T}$, we set 
\[
{\bf s}_{t'}=((\mu_{t_0, t}^{\ast}(A_{j; \widehat{\bf s}_t}))_{j\in J}, \varepsilon_t),
\]
where $t\in \Pi^{-1}(t')$ and $\varepsilon_t$ is a $J_{\rm uf} \times J$-submatrix of the skew-symmetrizable matrix $\widehat{\varepsilon}_t$ associated with $\widehat{\bf s}_t$. Then ${\bf s}_{t'}$ is a well-defined FZ-seed of $\mathcal{F}$, and $\mathcal{S} = \{{\bf s}_{t'}\}_{t' \in \mathbb{T}}$ is a cluster pattern. We have $\mathscr{A}(\mathcal{S})={\rm ord}(\mathcal{A})$ and $\mathscr{U}(\mathcal{S})={\rm up}(\mathcal{A})$. 

\begin{defi}[{\cite[Definition 3.1.1]{Qin}}]\label{d:order}
	Let ${\bf s} = (\mathbf{A}, \varepsilon)$ be an FZ-seed of $\mathcal{F}$, and assume that the exchange matrix $\varepsilon$ is of full rank. For $\bm{a}, \bm{a}'\in \mathbb{Z}^{J}$, we write 
	\[
	\bm{a} \preceq_{\varepsilon} \bm{a}'\ \text{if and only if }
	\bm{a}=\bm{a}'+\bm{v} \varepsilon\ \text{for some }\bm{v}\in \mathbb{Z}_{\geq 0}^{J_{\rm uf}},
	\]
	where we regard elements of $\mathbb{Z}^{J}$ (resp., $\mathbb{Z}_{\geq 0}^{J_{\rm uf}}$) as $1 \times J$-matrices (resp., $1 \times J_{\rm uf}$-matrices). This $\preceq_{\varepsilon}$ defines a partial order on $\mathbb{Z}^{J}$, called the \emph{dominance order} associated with $\varepsilon$.
\end{defi}

For $i \in J_{\rm uf}$, we set 
\[
\widehat{X}_{i; t} \coloneqq \prod_{j \in J} A_{j; t} ^{\varepsilon_{i, j} ^{(t)}};
\]
see \cite[Section 3]{FZ:ClusterIV}. For $t \overset{k}{\text{---}} t^\prime$ and $j \in J_{\rm uf}$, we have 
\begin{align}
\mu_k(\widehat{X}_{j; t}) \coloneqq \widehat{X}_{j; t^\prime} = \begin{cases}
\widehat{X}_{j; t}\widehat{X}_{k; t}^{[\varepsilon_{j, k}^{(t)}]_+} (1 + \widehat{X}_{k; t})^{-\varepsilon_{j, k} ^{(t)}}&(j\neq k),\\
\widehat{X}_{j; t}^{-1}&(j = k);
\end{cases}\label{eq:X-mutation}
\end{align}
see \cite[Proposition 3.9]{FZ:ClusterIV}. Note that only the principal part of the exchange matrix appears in the relation \eqref{eq:X-mutation}. In the rest of this section, we assume that 
\begin{enumerate}
\item[($\dagger$)] the exchange matrix $\varepsilon_{t}$ is of full rank for all $t \in \mathbb{T}$.
\end{enumerate}
Indeed, this assumption is automatically satisfied if $\varepsilon_t$ is of full rank for some $t \in \mathbb{T}$; see \cite[Lemma 3.2]{BFZ2}. We introduce the notion of weakly pointed elements, which is a slight modification of pointedness introduced by Qin \cite{Qin}.

\begin{defi}[{see \cite{FZ:ClusterIV, Qin, Tra}}]\label{d:weakly_pointed}
Let $\mathcal{S} = \{{\bf s}_t = (\mathbf{A}_t, \varepsilon_t)\}_{t \in \mathbb{T}}$ be a cluster pattern, and fix $t \in \mathbb{T}$. An element $f \in \c[A_{j;t} ^{\pm 1} \mid j \in J]$ is said to be \emph{weakly pointed} at $(g_j)_{j \in J} \in \z^J$ for $t$ if we can write
\[
f = \left(\prod_{j \in J} A_{j; t} ^{g_j}\right) \left(\sum_{{\bm a}=(a_j)_{j} \in \mathbb{Z}_{\geq 0}^{J_{\rm uf}}} c_{\bm a} \prod_{j \in J_{\rm uf}} \widehat{X}_{j; t} ^{a_j}\right)
\]
for some $\{c_{\bm a} \in \c \mid {\bm a} \in \mathbb{Z}_{\geq 0}^{J_{\rm uf}}\}$ such that $c_0 \neq 0$. In this case, we set $g_t (f) \coloneqq (g_j)_{j \in J} \in \z^J$, which is called the \emph{extended $g$-vector} of $f$ (see \cite[Section 6]{FZ:ClusterIV} and \cite[Definition 3.7]{Tra}). If $c_0 = 1$ in addition, then $f$ is said to be \emph{pointed}; see \cite[Definition 3.1.4]{Qin}.
\end{defi} 

One significant class of pointed elements is given by cluster monomials.

\begin{thm}[{see \cite[Corollary 6.3]{FZ:ClusterIV}, \cite[Theorem 1.7]{DWZ}, and \cite{GHKK}}]\label{t:pointedness_cluster_monomial}
All cluster variables and hence monomials are pointed for all $t \in \mathbb{T}$.
\end{thm}

Now we define our main valuations in this paper.

\begin{defi}\label{d:main_valuation}
Let $\mathcal{S} = \{{\bf s}_t = (\mathbf{A}_t, \varepsilon_t)\}_{t \in \mathbb{T}}$ be a cluster pattern, and $\preceq_{\varepsilon_t} ^{\rm op}$ the opposite order of $\preceq_{\varepsilon_t}$. We fix a total order $<_t$ on $\z^J$ which refines $\preceq_{\varepsilon_t} ^{\rm op}$. It induces a total order (denoted by the same symbol $<_t$) on the set of Laurent monomials in $A_{j;t}$, $j \in J$, by identifying $(a_j)_{j \in J} \in \z^J$ with $\prod_{j \in J} A_{j;t}
 ^{a_j}$. Let $v_{{\bf s}_t}$ (or simply $v_t$) denote the corresponding lowest term valuation $v^{\rm low} _{<_t}$ on $\mathcal{F} = \c(A_{j; t} \mid j \in J)$. 
\end{defi} 

The following is an immediate consequence of the definitions.

\begin{prop}\label{p:valuation_generalizing_g}
Let $\mathcal{S} = \{(\mathbf{A}_t, \varepsilon_t)\}_{t \in \mathbb{T}}$ be a cluster pattern. If $f \in \c[A_{j; t} ^{\pm 1} \mid j \in J]$ is weakly pointed for $t \in \mathbb{T}$, then the equality $v_t (f) = g_t (f)$ holds for an arbitrary refinement of $\preceq_{\varepsilon_t} ^{\rm op}$.
\end{prop}

Hence we obtain the following by \cref{t:pointedness_cluster_monomial}.

\begin{cor}\label{c:cluster_monomial_valuation}
If $f$ is a cluster monomial, then the equality $v_t (f) = g_t (f)$ holds for all $t \in \mathbb{T}$ and all refinements of $\preceq_{\varepsilon_t} ^{\rm op}$.
\end{cor}
\begin{ex}
Let $J=\{1,2,3,4\}, J_{\rm uf}=\{1,2\}$, and 
\[
\varepsilon=\begin{pmatrix}
    0&1&1&0\\
    -1&0&0&1
\end{pmatrix}.
\]
Consider the FZ-seed ${\bf s}_{t_0}= ((A_j)_{j\in J}, \varepsilon)$ of $\mathcal{F} \coloneqq \mathbb{C}(z_j \mid j \in J)$, and write $\mu_1\mu_2({\bf s}_{t_0})=((A'_j)_{j\in J}, \varepsilon')$. 
Then 
\[
\widehat{X}_{1} \coloneqq A_2A_3,\qquad \widehat{X}_{2} \coloneqq A_1^{-1}A_4,
\]
and 
\[
A'_1=\frac{A_1+A_4+A_2A_3A_4}{A_1A_2}=\frac{1}{A_2}(1+\widehat{X}_{2}+\widehat{X}_{1}\widehat{X}_{2}).
\]
Hence $A_1'$ is pointed for $t_0$, and its extended $g$-vector is $g_{t_0} (A_1')=(0, -1, 0, 0)$. 

On the other hand, since $1 \preceq_{\varepsilon} ^{\rm op}\widehat{X}_{j}$ ($j=1, 2$)  by definition, any refinement $<_t$ of $\preceq_{\varepsilon} ^{\rm op}$ satisfies $1 <_t\widehat{X}_{j}$ ($j=1, 2$). Therefore, 
\[
v_{t_0} (A_1') = (0, -1, 0, 0)=g_{t_0} (A_1'). 
\]
\end{ex}

\section{Relation with GHKK's toric degenerations}\label{s:GHKK}

Gross--Hacking--Keel--Kontsevich \cite{GHKK} gave a systematic method of constructing toric degenerations of compactified $\mathcal{A}$-cluster varieties. In this section, we relate this construction with Newton--Okounkov bodies by using the valuations defined in \cref{d:main_valuation}. This section is independent of the rest of this paper. In particular, readers who are mainly interested in the results concerning Newton--Okounkov bodies of Schubert varieties may skip it. The main theorem in this section is \cref{t:GHKK_NO}.

\subsection{Cluster varieties and tropicalizations}\label{ss:tropicalization}

In this subsection, we review some basic facts on cluster varieties and their tropicalizations, following \cite{FG, GHK, GHKK}. We take a family $\{\widehat{\bf s}_t\}_{t \in \widehat{\mathbb{T}}}$ of seeds for $\Gamma$ as in Section \ref{s:Clusterval}, and let 
\[\mathcal{A} = \bigcup_{t \in \widehat{\mathbb{T}}} \mathcal{A}_t = \bigcup_{t \in \widehat{\mathbb{T}}} \Spec (\mathbb{C}[A_{j; \widehat{\bf s}_t}^{\pm 1} \mid j \in J])\]
be the corresponding $\mathcal{A}$-cluster variety.

Following \cite{FG}, the Fock--Goncharov dual $\mathcal{A}^\vee$ of $\mathcal{A}$ is defined to be the Langlands dual of the $\mathcal{X}$-cluster variety. More precisely,  
\[\mathcal{A}^\vee = \bigcup_{t \in \widehat{\mathbb{T}}} \mathcal{A}_t ^\vee = \bigcup_{t \in \widehat{\mathbb{T}}} \Spec (\mathbb{C}[X_{j; t}^{\pm 1} \mid j \in J])\]
is given by gluing tori $\mathcal{A}_t ^\vee = \Spec (\mathbb{C}[X_{j; t}^{\pm 1} \mid j \in J])$ via the following birational mutations: for $t \xrightarrow{k} t'$,
\[\mu_k \colon \mathcal{A}^{\vee} _t \dashrightarrow \mathcal{A}^{\vee} _{t'},\ (X_{j; t})_{j \in J} \mapsto (X_{j; t} ^\prime)_{j \in J},\] 
where
\[X_{j; t} ^\prime \coloneqq
\begin{cases}
X_{j; t} X_{k; t} ^{[-\varepsilon_{k, j} ^{(t)}]_+} (1 + X_{k; t})^{\varepsilon_{k, j} ^{(t)}} &(j \neq k),\\
X_{j; t} ^{-1} &(j = k)
\end{cases}\] 
for $j \in J$ if we denote the matrix $\widehat{\varepsilon}$ for $\widehat{\bf s}_t$ by $\widehat{\varepsilon} = (\varepsilon_{i, j} ^{(t)})_{i, j \in J}$ (cf.\ \eqref{eq:X-mutation}). For each $t \in \widehat{\mathbb{T}}$, the torus $\mathcal{A}_t ^\vee$ is regarded as the dual torus $T_{M^\circ} = \Spec (\c[N^\circ])$ of $\mathcal{A}_t$, where the coordinate $X_{j; t}$ of $\mathcal{A}_t ^\vee$ corresponds to $d_j e_j \in N^\circ$ for the seed $\widehat{\bf s}_t$. Since $\mathcal{A}^\vee$ is a positive space, that is, obtained by gluing algebraic tori via subtraction-free birational maps, we obtain the set $\mathcal{A}^{\vee} (\r^T)$ of $\r^T$-valued points, where $\r^T$ is a semifield $(\r; \max, +)$. More precisely, the set $\mathcal{A}^{\vee} (\r^T)$ is defined by gluing 
\[\mathcal{A}^{\vee} _t (\r^T) = T_{M^\circ} (\r^T) = (M^\circ)_{\r} = \r^J\] 
via the following tropicalized cluster mutations: for $t \xrightarrow{k} t'$,
\[\mu_k ^T \colon \mathcal{A}^{\vee} _t (\r^T) \rightarrow \mathcal{A}^{\vee} _{t'} (\r^T),\ (g_j)_{j \in J} \mapsto (g' _j)_{j \in J},\] 
where 
\begin{equation}\label{eq:tropical_mutation}
\begin{aligned}
g' _j = 
\begin{cases}
g_j + [-\varepsilon_{k, j} ^{(t)}]_+ g_k + \varepsilon_{k, j} ^{(t)} [g_k]_+ &(j \neq k),\\
-g_j &(j = k)
\end{cases}
\end{aligned}
\end{equation}
for $j \in J$. Since $\mu_k ^T \colon \mathcal{A}^{\vee} _t (\r^T) \rightarrow \mathcal{A}^{\vee} _{t'} (\r^T)$ is bijective, the set $\mathcal{A}^{\vee} (\r^T)$ is identified with $\mathcal{A}^{\vee} _t (\r^T) = \r^J$ for each $t \in \widehat{\mathbb{T}}$. If we change $t$, then we obtain a different identification of $\mathcal{A}^{\vee} (\r^T)$ with the Euclidean space $\r^J$. For $q \in \mathcal{A}^\vee (\r^T)$ and $t \in \widehat{\mathbb{T}}$, denote by $q_t$ the corresponding element in $\mathcal{A}^{\vee} _t (\r^T)$, and set 
\[\Xi_t \coloneqq \{q_t \mid q \in \Xi\} \subset \mathcal{A}^{\vee} _t (\r^T)\]
for a subset $\Xi \subset \mathcal{A}^\vee (\r^T)$. Similarly, the set $\mathcal{A}^{\vee} (\z^T)$ of $\z^T$-valued points of $\mathcal{A}^{\vee}$ is defined, which is naturally regarded as a subset of $\mathcal{A}^{\vee} (\r^T)$, where $\z^T$ is a semifield $(\z; \max, +)$. Let 
\[\mathcal{A}_{\rm prin} = \bigcup_{t \in \widehat{\mathbb{T}}} \mathcal{A}_{{\rm prin}, t}\]
be the $\mathcal{A}$-cluster variety with principal coefficients; see \cite[Construction 2.11]{GHK} for the definition. In particular, the torus $\mathcal{A}_{{\rm prin}, t}$ is given by $\mathcal{A}_{{\rm prin}, t} = T_{N^\circ \oplus M} = \Spec(\c[M^\circ \oplus N])$. Since the canonical projections
\[\pi \colon \mathcal{A}_{{\rm prin}, t} = T_{N^\circ \oplus M} \twoheadrightarrow T_M\] 
for $t \in \widehat{\mathbb{T}}$ are compatible with birational mutations, we obtain a canonical map 
\[\pi \colon \mathcal{A}_{\rm prin} \rightarrow T_M,\]
which induces a $\c[N]$-algebra structure on ${\rm up}(\mathcal{A}_{\rm prin})$. For $z \in T_M$, we set $\mathcal{A}_z \coloneqq \pi^{-1} (z)$. Then we have $\mathcal{A}_e = \mathcal{A}$ for the identity element $e \in T_M$; see an argument given before \cite[Definition 2.12]{GHK}. By \cite[Proposition B.2 (4)]{GHKK}, the canonical projection $M^\circ \oplus N \twoheadrightarrow M^\circ$ induces a map $\rho \colon \mathcal{A}_{\rm prin} ^\vee \rightarrow \mathcal{A}^\vee$. Then the tropicalization 
\[\rho^T \colon \mathcal{A}_{\rm prin} ^\vee (\r^T) \rightarrow \mathcal{A}^\vee (\r^T)\]
is given by the canonical projection $(M^\circ \oplus N)_{\r} \twoheadrightarrow (M^\circ)_{\r}$ under the identifications $\mathcal{A}_{\rm prin} ^\vee (\r^T) \simeq \mathcal{A}_{{\rm prin}, t} ^\vee (\r^T)$ and $\mathcal{A}^\vee (\r^T) \simeq \mathcal{A}_t ^\vee (\r^T)$ for each $t \in \widehat{\mathbb{T}}$.

\subsection{GHKK's toric degenerations and Newton--Okounkov bodies} 

In this subsection, we assume the following conditions:
\begin{enumerate}
\item[{(i)}] $\mathcal{A}$ has large cluster complex in the sense of \cite[Definition 8.23]{GHKK};
\item[{(ii)}] the equivalent conditions of \cite[Lemma B.7]{GHKK} hold.
\end{enumerate}
Since $\mathcal{A}$ has large cluster complex, \cite[Proposition 8.25]{GHKK} implies the full Fock--Goncharov conjecture for $\mathcal{A}_{\rm prin}$ (see \cite[Definition 0.6]{GHKK}). In particular, it follows that
\[{\rm up}(\mathcal{A}_{\rm prin}) = \sum_{q \in \mathcal{A}_{\rm prin} ^\vee (\z^T)} \c \vartheta_q,\]
where $\{\vartheta_q \mid q \in \mathcal{A}_{\rm prin} ^\vee (\z^T)\}$ is the theta function basis. Since the translation action of $N$ on $\mathcal{A}_{{\rm prin}, t} ^\vee (\z^T) = M^\circ \oplus N$ for $t \in \widehat{\mathbb{T}}$ is compatible with tropicalized cluster mutations, we obtain a canonical action of $N$ on $\mathcal{A}_{\rm prin} ^\vee (\z^T)$. This is compatible with the $\c[N]$-algebra structure on ${\rm up}(\mathcal{A}_{\rm prin})$ through the theta function basis. Similarly, we obtain a canonical action of $N_{\r}$ on $\mathcal{A}_{\rm prin} ^\vee (\r^T)$. Since the equivalent conditions of \cite[Lemma B.7]{GHKK} hold, the submatrix $\varepsilon_t = (\varepsilon_{i, j})_{i \in J_{\rm uf}, j \in J}$ of the matrix $\widehat{\varepsilon}$ for $\widehat{\bf s}_t$ is of full rank for some $t \in \widehat{\mathbb{T}}$, which implies that it is of full rank for all $t \in \widehat{\mathbb{T}}$. In addition, the map $\pi \colon \mathcal{A}_{\rm prin} \rightarrow T_M$ is isomorphic to the trivial bundle $\mathcal{A} \times T_M \rightarrow T_M$ by \cite[Lemma B.7]{GHKK}. Thus, we obtain $\mathcal{A}_z \simeq \mathcal{A}$ for all $z \in T_M$. By \cite[Proposition 8.25]{GHKK}, this implies the full Fock--Goncharov conjecture for $\mathcal{A}$. In particular, we have 
\[{\rm up}(\mathcal{A}) = \sum_{q \in \mathcal{A}^\vee (\z^T)} \c \vartheta_q,\]
where $\{\vartheta_q \mid q \in \mathcal{A}^\vee (\z^T)\}$ is the theta function basis. Note that we have 
\[{\rm up}(\mathcal{A}) = {\rm up}(\mathcal{A}_{\rm prin}) \otimes_{\c[N]} \c,\]
where $\c[N] \rightarrow \c$ is given by $e \in T_M$. Then the equality $\vartheta_{\rho^T (q)} = \vartheta_q \otimes 1$ holds for all $q \in \mathcal{A}_{\rm prin} ^\vee (\z^T)$; see \cite[Definition--Lemma 7.14]{GHKK}. As we have seen after \cref{t:laurentpheno}, the family $\{\widehat{\bf s}_t\}_{t \in \widehat{\mathbb{T}}}$ of seeds for $\Gamma$ induces a cluster pattern. Hence we can define the pointedness, the extended $g$-vector $g_t$, and the valuation $v_t$ for $t \in \widehat{\mathbb{T}}$ by applying Definitions \ref{d:weakly_pointed} and \ref{d:main_valuation} to ${\bf s}_{\Pi(t)} = ((A_{j; t}
 \coloneqq \mu_{t_0, t}^{\ast}(A_{j; \widehat{\bf s}_t}))_{j\in J}, \varepsilon_t)$. By the construction of the theta function basis in \cite{GHKK}, each basis element $\vartheta_q \in {\rm up}(\mathcal{A}) \subset \c[A_{j; t}^{\pm 1} \mid j \in J]$ is pointed for all $t \in \widehat{\mathbb{T}}$, and we have 
\[g_t (\vartheta_q) = q_t\]
for all $q \in \mathcal{A}^\vee (\z^T)$ and $t \in \widehat{\mathbb{T}}$ (see \cite{GHKK} and \cite[Appendix A.1]{Qin2}). For a closed subset $\Xi \subset \mathcal{A}_{\rm prin} ^\vee (\r^T)$ and $d \in \z_{\ge 0}$, set 
\begin{align*}
&{\bf C}(\Xi) \coloneqq \overline{\{(r, p) \mid r \in \r_{\ge 0},\ p \in r \Xi\}} \subset \r_{\ge 0} \times \mathcal{A}_{\rm prin} ^\vee (\r^T),\\
&d\Xi (\z) \coloneqq \{q \in \mathcal{A}_{\rm prin} ^\vee (\z^T) \mid (d, q) \in {\bf C}(\Xi)\},
\end{align*}
and define $\widetilde{S}_{\Xi} \subset {\rm up}(\mathcal{A}_{\rm prin})[x]$ by
\[\widetilde{S}_{\Xi} \coloneqq \bigoplus_{d \in \z_{\ge 0}} \bigoplus_{q \in d \Xi(\z)} \c \vartheta_q x^d,\]
where $x$ is an indeterminate; see \cite[Section 8.1 and Theorem 8.19]{GHKK}. 

We now explain GHKK's construction of toric degenerations, following \cite[Section 8]{GHKK}. Let $\Xi \subset \mathcal{A}_{\rm prin} ^\vee (\r^T)$ be a rationally-defined positive convex polytope in the sense of \cite[Section 8]{GHKK}, and assume that $\Xi$ is full-dimensional and bounded. If we set 
\[\overline{\Xi} \coloneqq \rho^T (\Xi) \subset \mathcal{A}^\vee (\r^T),\]
then $\overline{\Xi}_t$ is a $|J|$-dimensional rational convex polytope for all $t \in \widehat{\mathbb{T}}$. We fix $t \in \widehat{\mathbb{T}}$, and write $\widehat{\bf s}_t = (e_j)_{j \in J}$. Set 
\[
N_t ^\oplus \coloneqq \sum_{j \in J} \z_{\ge 0} e_j,\quad N_{t, \r} ^\oplus \coloneqq \sum_{j \in J} \r_{\ge 0} e_j,
\]
and let
\[\pi_{N} ^{(t)} \colon \mathcal{A}_{\rm prin} ^\vee (\r^T) \simeq \mathcal{A}_{{\rm prin}, t} ^\vee (\r^T) = (M^\circ \oplus N)_{\r} \twoheadrightarrow N_{\r}\] 
denote the second projection. If we write 
\[\widetilde{\Xi} \coloneqq \Xi + N_{\r},\quad {\Xi}^+ \coloneqq \widetilde{\Xi} \cap (\pi_{N} ^{(t)})^{-1} (N_{t, \r} ^\oplus),\]
then the sets $\widetilde{S}_{\widetilde{\Xi}}$ and $\widetilde{S}_{{\Xi}^+}$ are both finitely generated $\c$-subalgebras of ${\rm up}(\mathcal{A}_{\rm prin})[x]$; see an argument given after \cite[Lemma 8.29]{GHKK}. Now the inclusion of $\c[N_t ^\oplus] = \c[N_t ^\oplus] \vartheta_0$ in the degree $0$ part of $\widetilde{S}_{{\Xi}^+}$ induces a flat morphism $\pi_t \colon \mathfrak{X} \coloneqq {\rm Proj}(\widetilde{S}_{{\Xi}^+}) \rightarrow \Spec(\c[N_t ^\oplus]) = \c ^J$. Similarly, the inclusion of $\c[N] = \c[N] \vartheta_0$ in the degree $0$ part of $\widetilde{S}_{\widetilde{\Xi}}$ induces a flat morphism $\mathfrak{X}^\prime \coloneqq {\rm Proj}(\widetilde{S}_{\widetilde{\Xi}}) \rightarrow \Spec(\c[N]) = T_M$. This is the restriction of $\pi_t \colon \mathfrak{X} \rightarrow \c ^J$ as follows:
\begin{align*}
\xymatrix{
\mathfrak{X}^\prime \ar[d] \ar@{^{(}->}[r] & \mathfrak{X} \ar[d]^-{\pi_t} \\
T_M \ar@{^{(}->}[r] & \c^J.
}
\end{align*}
For $z \in \c^J$, let $\mathfrak{X}_z \coloneqq \pi_t ^{-1} (z)$. The generic fiber of $\pi_t$ gives a compactification of $\mathcal{A}$ as follows.

\begin{thm}[{see \cite[Theorem 8.32]{GHKK}}]
For $z \in T_M$, the fiber $\mathfrak{X}_z$ is a normal projective variety containing $\mathcal{A}_z \simeq \mathcal{A}$ as an open subscheme.
\end{thm}
\begin{thm}[{\cite[Theorem 8.30]{GHKK}}]
The central fiber $\mathfrak{X}_0$ is the normal toric variety corresponding to the rational convex polytope $\overline{\Xi}_t$.
\end{thm}

We set 
\[\widetilde{S}_e \coloneqq \widetilde{S}_{\widetilde{\Xi}} \otimes_{\c[N]} \c,\]
where $\c[N] \rightarrow \c$ is given by $e \in T_M$. Then it follows that $\mathfrak{X}_e = {\rm Proj}(\widetilde{S}_e)$. 
GHKK's framework gives a systematic method of constructing toric degenerations of $\mathfrak{X}_e$. In the following theorem, we interpret the rational convex polytope $\overline{\Xi}_t$ as a Newton--Okounkov body of $\mathfrak{X}_e$.

\begin{thm}\label{t:GHKK_NO}
For each $t \in \widehat{\mathbb{T}}$, the Newton--Okounkov body $\Delta(\mathfrak{X}_e, \mathcal{L}, v_t, x)$ coincides with the rational convex polytope $\overline{\Xi}_t$, where $\mathcal{L}$ is the restriction of $\mathcal{O}_{\mathfrak{X}}(1)$ on $\mathfrak{X} = {\rm Proj}(\widetilde{S}_{{\Xi}^+})$ to $\mathfrak{X}_e$. In particular, $\Delta(\mathfrak{X}_e, \mathcal{L}, v_t, x)$ is independent of the choice of a refinement of the opposite dominance order $\preceq_{\varepsilon_t} ^{\rm op}$.
\end{thm}

\begin{proof}
By considering $\mathcal{L}^{\otimes L}$ for sufficiently divisible $L$, we will reduce the proof to the case that the theta function basis gives a basis of the homogeneous coordinate ring of $(\mathfrak{X}_e,\mathcal{L}^{\otimes L})$. 
For $d \in \z_{\ge 0}$, denote the degree $d$ part of $\widetilde{S}_e$ by $(\widetilde{S}_e)_d$. Then the definition of $\widetilde{S}_{\widetilde{\Xi}}$ implies that $(\widetilde{S}_e)_0 = \c$. In addition, since $\vartheta_{\rho^T (q)} = \vartheta_q \otimes 1$ for all $q \in \mathcal{A}_{\rm prin} ^\vee (\z^T)$, we see that 
\[\widetilde{S}_e = \bigoplus_{d \in \z_{\ge 0}} \bigoplus_{q \in d \overline{\Xi} \cap \mathcal{A}^\vee (\z^T)} \c \vartheta_q x^d.\]
Since $\widetilde{S}_{\widetilde{\Xi}}$ is finitely generated, the $\c$-algebra $\widetilde{S}_e$ is also finitely generated. Hence, for sufficiently divisible $L \in \z_{>0}$, the $\c$-subalgebra 
\[\widetilde{S}_e ^{(L)} \coloneqq \bigoplus_{k \in \z_{\ge 0}} (\widetilde{S}_e)_{k L}\]
is generated by $(\widetilde{S}_e)_{L}$, and we have 
\[(\mathfrak{X}_e, \mathcal{L}^{\otimes L}) \simeq ({\rm Proj} (\widetilde{S}_e ^{(L)}), \mathcal{O}(1)).\]
By the definition of Newton--Okounkov bodies, it follows that
\[\Delta(\mathfrak{X}_e, \mathcal{L}^{\otimes L}, v_t, x^L) = L \Delta(\mathfrak{X}_e, \mathcal{L}, v_t, x).\]
Hence it suffices to prove that 
\begin{equation}\label{eq:positive_polytope_NO_body}
\begin{aligned}
L \overline{\Xi}_t = \Delta(\mathfrak{X}_e, \mathcal{L}^{\otimes L}, v_t, x^L).
\end{aligned}
\end{equation}
Since $\widetilde{S}_e ^{(L)}$ is the homogeneous coordinate ring of $(\mathfrak{X}_e, \mathcal{L}^{\otimes L}) \simeq ({\rm Proj} (\widetilde{S}_e ^{(L)}), \mathcal{O}(1))$, it follows by Remark \ref{r:homogeneous_coordinate} that $C(\mathfrak{X}_e, \mathcal{L}^{\otimes L}, v_t, x^L)$ coincides with the smallest real closed cone containing 
\begin{equation}\label{eq:generating_lattice_points}
\begin{aligned}
\bigcup_{k \in \z_{>0}} \{(k, v_t(\sigma / x^{k L})) \mid \sigma \in (\widetilde{S}_e)_{k L} \setminus \{0\}\}.
\end{aligned}
\end{equation}
Since $v_t (\vartheta_q) = g_t (\vartheta_q) = q_t$ for all $k \in \z_{> 0}$ and $q \in k L \overline{\Xi} \cap \mathcal{A}^\vee (\z^T)$ by \cref{p:valuation_generalizing_g}, we see that 
\[\{v_t(\sigma / x^{k L}) \mid \sigma \in (\widetilde{S}_e)_{k L} \setminus \{0\}\} = (k L \overline{\Xi}_t) \cap \z^J.\]
Hence the set in \eqref{eq:generating_lattice_points} equals 
\[\bigcup_{k \in \z_{>0}} \{k\} \times ((k L \overline{\Xi}_t) \cap \z^J).\]
This implies \eqref{eq:positive_polytope_NO_body}, which proves the theorem.
\end{proof}

\begin{ex}[{see \cite[Example 8.31]{GHKK}}]
Set $N = N_{\rm uf} = N^\circ = \z^2$ and $d_1 = d_2 = 1$. 
We fix $t \in \widehat{\mathbb{T}}$, and consider a seed $\widehat{\bf s} = (e_1, e_2)$ with
\[\widehat{\varepsilon} = (\{e_i, e_j\})_{i,j = 1,2} = \begin{pmatrix}
0&1\\
-1&0
\end{pmatrix}.\]
Then there exists a rationally-defined positive convex polytope $\Xi \subset \mathcal{A}_{\rm prin} ^\vee (\r^T)$, which is full-dimensional and bounded, such that the rational convex polytope $\overline{\Xi}_t \subset \r^2$ coincides with the pentagon with vertices
\[(1,0), (0,1), (-1,0), (0, -1), (1, -1).\]
In addition, we see by \cite[Example 8.31]{GHKK} that $\mathfrak{X}_e$ is a smooth del Pezzo surface of degree $5$. 
Hence it follows by \cref{t:GHKK_NO} that the pentagon $\overline{\Xi}_t$ is a Newton--Okounkov body of the del Pezzo surface $\mathfrak{X}_e$ associated with the valuation $v_t$ arising from its cluster structure.
\end{ex}

\section{Cluster structures on unipotent cells}\label{s:unipcell}

In the rest of this paper, we study the Newton--Okounkov bodies of Schubert varieties associated with the valuation defined in \cref{d:main_valuation}. 
We use an (upper) cluster algebra structure on the coordinate ring of a unipotent cell $U^-_w$, which is an affine open subvariety of a Schubert variety $X(w)$ (in particular, $\mathbb{C}(U^-_w)\simeq \mathbb{C}(X(w))$). 
In this section, we first review the cluster structure on the coordinate ring of $U^-_w$, verified in \cite{BFZ2,GLS:Kac-Moody,Dem}; see also Appendix \ref{a:doubleBruhat} for its relation with that of the double Bruhat cell $G^{w, e}$.
We then study some properties of the associated Newton--Okounkov bodies of $X(w)$ by using the existence of a $\c$-basis of $\c[U_w ^-]$ with desirable properties.

\subsection{Unipotent cells}\label{ss:unipotent}

For $w\in W$, we set 
\begin{align*}
U^-_w&\coloneqq U^-\cap B\widetilde{w}B,
\end{align*}
where $\widetilde{w}\in N_G(H)$ is a lift for $w\in W=N_G(H)/H$. The space $U^-_w$ is called the \emph{unipotent cell} associated with $w$. Note that it does not depend on the choice of a lift $\widetilde{w}$ for $w$, and the open embedding $U^- \hookrightarrow G/B$ induces an open embedding $U_w^- \hookrightarrow X(w)$. In the following, we define specific regular functions on $U^-_w$. 

\begin{prop}
	For $\lambda\in P_+$, there exists a unique non-degenerate symmetric $\mathbb{C}$-bilinear form $(\cdot ,\cdot)_{\lambda}$ on $V(\lambda)$ such that
	\begin{align*}
	(v_{\lambda}, v_{\lambda})_{\lambda}&=1,&(gv, v')_{\lambda}&=(v, g^T v^\prime)_{\lambda}
	\end{align*}
	for $g\in G$ and $v, v'\in V(\lambda)$; here the correspondence $g\mapsto g^T$ is an anti-involution of the algebraic group $G$ given by $x_i(t)^T=y_i(t)$ and $h^T=h$ for $i\in I, t\in\mathbb{C}$, and $h\in H$.
\end{prop}

For $v\in V(\lambda)$, define $v^{\vee} \in V(\lambda)^{\ast}$ by $v^\vee (v^\prime) \coloneqq (v, v')_{\lambda}$ for $v' \in V(\lambda)$. For $w \in W$, we set $f_{w\lambda} \coloneqq v_{w\lambda}^{\vee}\in V(\lambda)^{\ast}$; note that $f_{w\lambda}(v_{w\lambda}) = (v_{w\lambda}, v_{w\lambda})_{\lambda}=1$. For $f\in V(\lambda)^{\ast}$ and $v\in V(\lambda)$, define a function $C_{f, v} \in \mathbb{C}[G]$ by  
\[
C_{f, v} (g) \coloneqq \langle f, gv\rangle
\]
for $g\in G$. A function on $G$ of this form is called a \emph{matrix coefficient}. For $u,u'\in W$, we write 
\[
\Delta_{u\lambda,u'\lambda} \coloneqq C_{f_{u\lambda},v_{u'\lambda}},
\]
which is called a \emph{generalized minor}. For $f\in V(\lambda)^{\ast}$, $v\in V(\lambda)$, and $u, u'\in W$, we define $D_{f, v}, D_{u\lambda,u'\lambda} \in \c[U_w^-]$ by
\[
D_{f, v} \coloneqq C_{f, v}|_{U_w^-},\quad 
D_{u\lambda,u'\lambda}\coloneqq \Delta_{u\lambda,u'\lambda}|_{U_w^-},
\]
respectively. The domains of these functions depend on $w\in W$, but we omit it from the notations because there is no fear of confusion. The function $D_{u\lambda,u'\lambda}$ is called a \emph{unipotent minor}.

\subsection{Cluster structures on unipotent cells}\label{ss:cluster_structure_unipotent_cell}

Berenstein--Fomin--Zelevinsky \cite{BFZ2} proved that the coordinate ring of the double Bruhat cell $G^{w, e}\coloneqq B^-\cap B\widetilde{w}B$ has a structure of an upper cluster algebra. This double Bruhat cell $G^{w, e}$ is closely related to the unipotent cell $U^-_w$, and their result induces an upper cluster algebra structure on $\mathbb{C}[U^-_w]$. The precise demonstration of this implication is given in Appendix \ref{a:doubleBruhat}. 

\begin{ntn}\label{n:indexplus}
	Fix $w\in W$ and $\bm{i}=(i_{1},\dots,i_{m})\in R(w)$. For $1 \leq k \leq m$ and $i\in I$, we write 
	\begin{align*}
		w_{\leq k}&\coloneqq s_{i_1}\cdots s_{i_k},\ w_{\leq 0} \coloneqq e,\\
		k^{+} &\coloneqq \min(\{m+1\}\cup\{k+1\leq j\leq m \mid i_{j}=i_{k}\}),\\
		k^{-} &\coloneqq \max(\{0\}\cup\{1\leq j\leq k-1\mid i_{j}=i_{k}\}),\\
		k^-(i)&\coloneqq \max(\{0\}\cup \{1\leq j\leq k-1\mid i_{j}=i\}).
	\end{align*}
\end{ntn}

For $w \in W$, fix $\bm{i}=(i_1,\dots, i_m)\in R(w)$. We set 
	\[
	J \coloneqq \{1,\dots,m\},\ J_{\rm fr} = \{j\in J\mid j^+=m+1\},\ \text{and}\ J_{\rm uf} \coloneqq J\setminus J_{\rm fr}. 
	\]
	Define a $J_{\rm uf} \times J$-integer matrix $\varepsilon^{\bm i} = (\varepsilon_{s, t})_{s\in J_{\rm uf}, t\in J}$ by
	\[
	\varepsilon_{s, t} \coloneqq
	\begin{cases}
	-1&\text{if}\ s=t^+, \\
	-c_{i_t, i_s}&\text{if}\ t<s<t^+<s^+,\\
	1&\text{if}\ s^+=t, \\
	c_{i_t, i_s}&\text{if}\ s<t<s^+<t^+, \\
	0&\text{otherwise}.
	\end{cases}
	\]
	Recall that $c_{i, j}=\langle \alpha_j, h_i\rangle$. For $s\in J$, we set  
\[
D(s, \bm{i})\coloneqq D_{w_{\leq s}\varpi_{i_s}, \varpi_{i_s}},
\]
where $\varpi_i \in  P_+$ denotes the fundamental weight corresponding to $i\in I$. Let us consider the upper cluster algebra $\mathscr{U}(\mathbf{s}_{t_0})$ whose initial FZ-seed is given as $\mathbf{s}_{t_0} = ((A_{s; t_0})_{s \in J}, \varepsilon^{\bm{i}})$. 

\begin{thm}[{\cite[Theorem 2.10]{BFZ2} and \cite[Theorem 4.16]{Wil} (see also \cref{t:Bruhat-unip})}]\label{t:upperBruhat}
	There exists a $\mathbb{C}$-algebra isomorphism 
\[
\mathscr{U}({\bf s}_{t_0})\xrightarrow{\sim} \mathbb{C}[U^-_w]\ \text{given by}\ A_{s; t_0}\mapsto D(s, \bm{i})\ \text{for}\ s\in J.
\]
\end{thm}

Through the isomorphism in \cref{t:upperBruhat}, we obtain an FZ-seed of $\c (U^-_w)$ given by 
\[
\mathbf{s}_{\bm{i}}\coloneqq (\mathbf{D}_{\bm{i}}\coloneqq (D(s, \bm{i}))_{ s\in J}, \varepsilon^{\bm{i}}).
\]

\begin{thm}[{\cite[Theorem 3.5]{FG:amal}}]\label{t:independence_of_rex}
For $w\in W$, the cluster pattern associated with $\mathbf{s}_{\bm{i}}$ does not depend on the choice of $\bm{i}\in R(w)$. Namely, all $\mathbf{s}_{\bm{i}}$, $\bm{i}\in R(w)$, are mutually mutation equivalent. 
\end{thm}

When $\mathfrak{g}$ is of simply-laced, the entries of $\varepsilon^{\bm{i}}$ are $0$ or $\pm 1$. Hence $\varepsilon^{\bm{i}}$ is described as the quiver whose vertex set is $J$ and whose arrow set is given by 
\[
\{s \to t\mid \varepsilon_{s, t}=-1\ \text{or}\ \varepsilon_{t, s}=1 \}. 
\]
Note that there are no arrows between the vertices in $J_{\rm fr}$.

\begin{ex}\label{e:initial}
We denote by $w_0$ the longest element of the Weyl group $W$. 

Consider the case when $G = SL_4(\mathbb{C})$ and $\bm{i} = (2, 1, 2, 3, 2, 1) \in R(w_0)$ (see \cref{s:cluster_cone} for the labeling of the vertices of the Dynkin diagram). The initial FZ-seed $\mathbf{s}_{\bm{i}} = (\mathbf{D}_{\bm{i}}, \varepsilon^{\bm{i}})$ is given as follows:
	
	\hfill
	\scalebox{0.7}[0.7]{
		\begin{xy} 0;<1pt,0pt>:<0pt,-1pt>::
			(240,30) *+{D_{w_0 \varpi_3, \varpi_3}} ="4",
			(60,60) *+{D_{s_2 \varpi_2, \varpi_2}} ="1",
			(180,60) *+{D_{s_2 s_1 s_2 \varpi_2, \varpi_2}} ="3",
			(300,60) *+{D_{w_0 \varpi_2, \varpi_2}} ="5",
			(120,90) *+{D_{s_2 s_1 \varpi_1, \varpi_1}} ="2",
			(360,90) *+{D_{w_0 \varpi_1, \varpi_1}} ="6",
			"1", {\ar"2"},
			"3", {\ar"1"},
			"3", {\ar"4"},
			"5", {\ar"3"},
			"2", {\ar"5"},
			"6", {\ar"2"},
		\end{xy}
	}
	\hfill
	\hfill
	
When $G = SL_5(\mathbb{C})$ and $\bm{i} = (1, 2, 1, 3, 2, 1, 4, 3, 2, 1)\in R(w_0)$, the initial FZ-seed $\mathbf{s}_{\bm{i}} = (\mathbf{D}_{\bm{i}}, \varepsilon^{\bm{i}})$ is described as follows:
	
	\hfill
	\scalebox{0.7}[0.7]{
		\begin{xy} 0;<1pt,0pt>:<0pt,-1pt>::
			(180,0) *+{D_{w_0 \varpi_4, \varpi_4}} ="0",
			(120,30) *+{D_{s_1 s_2 s_1 s_3 \varpi_3, \varpi_3}} ="1",
			(240,30) *+{D_{w_0 \varpi_3, \varpi_3}} ="2",
			(60,60) *+{D_{s_1 s_2 \varpi_2, \varpi_2}} ="3",
			(180,60) *+{D_{s_1 s_2 s_1 s_3 s_2 \varpi_2, \varpi_2}} ="4",
			(300,60) *+{D_{w_0 \varpi_2, \varpi_2}} ="5",
			(0,90) *+{D_{s_1 \varpi_1, \varpi_1}} ="6",
			(120,90) *+{D_{s_1 s_2 s_1 \varpi_1, \varpi_1}} ="7",
			(240,90) *+{D_{s_1 s_2 s_1 s_3 s_2 s_1 \varpi_1, \varpi_1}} ="8",
			(360,90) *+{D_{w_0 \varpi_1, \varpi_1}} ="9",
			"1", {\ar"0"},
			"3", {\ar"1"},
			"1", {\ar"4"},
			"4", {\ar"2"},
			"6", {\ar"3"},
			"3", {\ar"7"},
			"7", {\ar"4"},
			"4", {\ar"8"},
			"8", {\ar"5"},
			"2", {\ar"1"},
			"4", {\ar"3"},
			"5", {\ar"4"},
			"7", {\ar"6"},
			"8", {\ar"7"},
			"9", {\ar"8"},
		\end{xy}
	}
	\hfill
	\hfill
	
\end{ex}

\subsection{Bases parametrized by tropical points}\label{ss:bases_tropical}

Consider the upper cluster algebra structure on $\c[U_w ^-]$ explained in Section \ref{ss:cluster_structure_unipotent_cell}, and let $\mathcal{S}=\{{\bf s}_t = (\mathbf{A}_t, \varepsilon_t)\}_{t \in \mathbb{T}}$ be the corresponding cluster pattern. 
Then each $t \in \mathbb{T}$ gives a valuation $v_t = v_{{\bf s}_t}$ on $\c(X(w)) = \c(U_w ^-)$ as in \cref{d:main_valuation}. 
To study the Newton--Okounkov body $\Delta(X(w), \mathcal{L}_\lambda, v_t, \tau_\lambda)$, we use a $\c$-basis ${\bf B}_w$ of $\c[U_w ^-]$ that has the following properties ${\rm (T)}_1$--${\rm (T)}_4$. 
\begin{enumerate}
\item[${\rm (T)}_1$] Every element in ${\bf B}_w$ is weakly pointed for all $t \in \mathbb{T}$. In particular, the extended $g$-vectors $g_t (b)$, $b \in {\bf B}_w$, are defined.
\item[${\rm (T)}_2$] For each $t \in \mathbb{T}$, the map ${\bf B}_w \rightarrow \z^J$ given by $b \mapsto g_t (b)$ is injective.
\item[${\rm (T)}_3$] For $t \overset{k}{\text{---}} t^\prime$ in $\mathbb{T}$ and $b \in {\bf B}_w$, the extended $g$-vectors $g_t (b) = (g_j)_{j \in J}$ and $g_{t^\prime} (b) = (g_j ^\prime)_{j \in J}$ satisfy the following relation;
\[g_j ^\prime \coloneqq
\begin{cases}
g_j + [-\varepsilon_{k, j} ^{(t)}]_+ g_k + \varepsilon_{k, j} ^{(t)} [g_k]_+ &(j \neq k),\\
-g_j &(j = k)
\end{cases}\] 
for $j \in J$. In other words, the equality $g_{t^\prime} (b) = \mu_k ^T (g_t (b))$ holds for all $b \in {\bf B}_w$, where $\mu_k^T$ is the tropicalized cluster mutation for $\mathcal{A}^\vee$ given in \eqref{eq:tropical_mutation}.
\item[${\rm (T)}_4$] For each $\lambda \in P_+$, there exists a subset ${\bf B}_w [\lambda] \subset {\bf B}_w$ such that 
\[\{\sigma / \tau_{\lambda} \mid \sigma \in H^0(X(w), \mathcal{L}_{\lambda})\} = \sum_{b \in {\bf B}_w [\lambda]} \c b.\]
\end{enumerate}

For each $b \in {\bf B}_w$, it follows from the property ${\rm (T)}_3$ that the extended $g$-vectors $g_t (b)$, $t \in \mathbb{T}$, form a well-defined element of $\mathcal{A}^\vee(\mathbb{R}^T)$; hence ${\bf B}_w$ is parametrized by tropical points in $\mathcal{A}^\vee(\mathbb{R}^T)$. 
Note that $\mathcal{A}^\vee(\mathbb{R}^T)$ is defined from an FZ-seed as in Section \ref{ss:tropicalization}.
The existence of such basis ${\bf B}_w$ was proved by Kashiwara--Kim \cite{KasKim} and Qin \cite{Qin3}.

\begin{thm}[{see \cite{KasKim, Qin3}}]\label{t:existence_of_bases}
For each $w \in W$, there exists a $\c$-basis ${\bf B}_w$ of $\c[U_w ^-]$ having the properties ${\rm (T)}_1$--${\rm (T)}_4$. 
\end{thm}

Indeed, in \cite{KasKim,Qin3}, they showed that the \emph{dual canonical basis}/\emph{upper global basis} of $\c[U_w ^-]$ in the sense of Lusztig \cite{Lus_can, Lus_quivers, Lus1} and Kashiwara \cite{Kas1,Kas2,Kas3} has the properties ${\rm (T)}_1$--${\rm (T)}_4$; see Appendix\ \ref{a:upper_global_bases} for more details.  
In the rest of this subsection, we give some corollaries of \cref{t:existence_of_bases}. 

\begin{cor}\label{c:relation_of_NO_by_tropicalized_mutations}
The following hold for all $w \in W$, $\lambda \in P_+$, and $t \in \mathbb{T}$.
\begin{enumerate}
\item[{\rm (1)}] The Newton--Okounkov body $\Delta(X(w), \mathcal{L}_\lambda, v_t, \tau_\lambda)$ is independent of the choice of a refinement of the opposite dominance order $\preceq_{\varepsilon_t} ^{\rm op}$.
\item[{\rm (2)}] If $t \overset{k}{\text{---}} t^\prime$, then 
\[\Delta(X(w), \mathcal{L}_\lambda, v_{t^\prime}, \tau_\lambda) = \mu_k ^T (\Delta(X(w), \mathcal{L}_\lambda, v_t, \tau_\lambda)).\]
\end{enumerate}
\end{cor}

\begin{proof}
Take a $\c$-basis ${\bf B}_w$ of $\c[U_w ^-]$ as in \cref{t:existence_of_bases}.
By Proposition \ref{p:valuation_generalizing_g} and the property ${\rm (T)}_1$, we have $v_t (b) = g_t (b)$ for all $t \in \mathbb{T}$ and $b \in {\bf B}_w$. Hence we deduce by \cref{prop1_val} and the property ${\rm (T)}_2$ that 
\begin{equation}
\begin{aligned}\label{eq:values_are_g_vectors}
v_t \left(\left(\sum_{b \in Y} \c b\right) \setminus \{0\}\right) = \{g_t (b) \mid b \in Y\}
\end{aligned}
\end{equation}
for an arbitrary subset $Y \subset {\bf B}_w$. Since $\mathcal{L}_{\lambda} ^{\otimes k} = \mathcal{L}_{k\lambda}$ and $\tau_\lambda ^k \in \c^\times \tau_{k\lambda}$ in $H^0(X(w), \mathcal{L}_{k\lambda})$ for all $k \in \z_{> 0}$, it follows that 
\[S(X(w), \mathcal{L}_\lambda, v_t, \tau_\lambda) = \bigcup_{k \in \z_{>0}} \{(k, v_t (\sigma / \tau_{k\lambda})) \mid \sigma \in H^0(X(w), \mathcal{L}_{k\lambda}) \setminus \{0\}\}.\] 
In addition, by the property ${\rm (T)}_4$, we have 
\[
\{\sigma / \tau_{k\lambda} \mid \sigma \in H^0(X(w), \mathcal{L}_{k\lambda})\} = \sum_{b \in {\bf B}_w[k\lambda]} \c b.
\]
Hence it follows by \eqref{eq:values_are_g_vectors} that 
\begin{equation}
\begin{aligned}\label{eq:semigroups_from_g_vectors}
S(X(w), \mathcal{L}_\lambda, v_t, \tau_\lambda) = \bigcup_{k \in \z_{>0}} \{(k, g_t (b)) \mid b \in {\bf B}_w [k\lambda]\},
\end{aligned}
\end{equation}
which implies part (1) of the corollary. 

By \eqref{eq:semigroups_from_g_vectors} and the property ${\rm (T)}_3$, we deduce that 
\begin{equation}
\begin{aligned}\label{eq:mutations_for_semigroups}
S(X(w), \mathcal{L}_\lambda, v_{t^\prime}, \tau_\lambda) = \tilde{\mu}_k ^T (S(X(w), \mathcal{L}_\lambda, v_t, \tau_\lambda))
\end{aligned}
\end{equation}
for $t \overset{k}{\text{---}} t^\prime$, where $\tilde{\mu}_k ^T \colon \r \times \r^J \rightarrow \r \times \r^J$ is defined by 
\begin{equation}
\begin{aligned}\label{eq:definition_tilde_mu}
\tilde{\mu}_k ^T (k, {\bm a}) \coloneqq (k, \mu_k ^T({\bm a}))
\end{aligned}
\end{equation}
for $(k, {\bm a}) \in \r \times \r^J$. Since $\mu_k ^T$ and hence $\tilde{\mu}_k ^T$ are piecewise-linear, we see by \eqref{eq:mutations_for_semigroups} that 
\begin{align*}
C(X(w), \mathcal{L}_\lambda, v_{t^\prime}, \tau_\lambda) = \tilde{\mu}_k ^T (C(X(w), \mathcal{L}_\lambda, v_t, \tau_\lambda)).
\end{align*}
This proves part (2) of the corollary. 
\end{proof}

By \cref{c:relation_of_NO_by_tropicalized_mutations} (2), the family $\{\Delta(X(w), \mathcal{L}_\lambda, v_{t}, \tau_\lambda)\}_{t\in \mathbb{T}}$ of Newton--Okounkov bodies can be viewed as a well-defined subset of $\mathcal{A}^\vee(\mathbb{R}^T)$. 

The following is also proved in the proof of \cref{c:relation_of_NO_by_tropicalized_mutations}.

\begin{cor}\label{c:mutations_for_semigroups_and_cones}
The following hold for all $w \in W$, $\lambda \in P_+$, and $t \in \mathbb{T}$.
\begin{enumerate}
\item[{\rm (1)}] The semigroup $S(X(w), \mathcal{L}_\lambda, v_t, \tau_\lambda)$ and the real closed cone $C(X(w), \mathcal{L}_\lambda, v_t, \tau_\lambda)$ are both independent of the choice of a refinement of the opposite dominance order $\preceq_{\varepsilon_t} ^{\rm op}$.
\item[{\rm (2)}] If $t \overset{k}{\text{---}} t^\prime$, then the equalities
\begin{align*}
&S(X(w), \mathcal{L}_\lambda, v_{t^\prime}, \tau_\lambda) = \tilde{\mu}_k ^T (S(X(w), \mathcal{L}_\lambda, v_t, \tau_\lambda))\ {\it and}\\
&C(X(w), \mathcal{L}_\lambda, v_{t^\prime}, \tau_\lambda) = \tilde{\mu}_k ^T (C(X(w), \mathcal{L}_\lambda, v_t, \tau_\lambda))
\end{align*}
hold, where $\tilde{\mu}_k ^T$ is the map defined in \eqref{eq:definition_tilde_mu}.
\end{enumerate}
\end{cor}

For $t \in \mathbb{T}$, the \emph{cluster cone} $C_{{\bf s}_t} \subset \r^J$ of $X(w)$ associated with ${\bf s}_t$ is defined as the smallest real closed cone containing $v_t (\c[U^- \cap X(w)] \setminus \{0\})$. 
Since $v_t (\c[U^- \cap X(w)] \setminus \{0\})$ forms a semigroup by the definition of valuations, it follows that the cone $C_{{\bf s}_t}$ is convex.
Note that
\[
\c[U^- \cap X(w)] = \bigcup_{\lambda \in P_+} \{\sigma/\tau_\lambda \mid \sigma \in H^0 (X(w), \mathcal{L}_\lambda)\}; 
\]
see, for instance, \cite[Corollary 3.20 (1)]{FO}.
Hence we obtain a $\c$-basis ${\bf B}[w]$ of $\c[U^- \cap X(w)]$ defined by
\[{\bf B}[w] \coloneqq \bigcup_{\lambda \in P_+} {\bf B}_w[\lambda].\]
Then it follows by \eqref{eq:values_are_g_vectors} that 
\[v_t (\c[U^- \cap X(w)] \setminus \{0\}) = \{g_t (b) \mid b \in {\bf B}[w]\}.\]
Hence the following holds. 

\begin{cor}
For all $w \in W$ and $t \in \mathbb{T}$, the cluster cone $C_{{\bf s}_t}$ of $X(w)$ is independent of the choice of a refinement of the opposite dominance order $\preceq_{\varepsilon_t} ^{\rm op}$.
\end{cor}

\section{Relation with representation-theoretic polytopes}\label{s:rel_with_stringNZ}

Recall from Section \ref{ss:NOSchu} the birational morphism 
\[
\c^m \to U^- \cap X(w),\ (t_1, \ldots, t_m) \mapsto y_{i_1}(t_1) \cdots y_{i_m}(t_m) \bmod B
\]
for $w \in W$ and ${\bm i} = (i_1, \ldots, i_m) \in R(w)$. Then the image of $(\mathbb{C}^{\times})^{m}$ under this morphism is contained in the unipotent cell $U^-_w$. Hence we obtain the following birational morphism: 
\[
y_{\bm{i}}\colon (\mathbb{C}^{\times})^{m} \to U^-_w, \ (t_1, \ldots, t_m) \mapsto y_{i_1}(t_1) \cdots y_{i_m}(t_m) \bmod B,
\]
which is indeed injective. It gives an embedding $y_{\bm{i}}^{\ast}\colon \c [U^-_w]\hookrightarrow \mathbb{C}[t_1^{\pm 1},\dots, t_m^{\pm 1}]$, which induces an isomorphism $y_{\bm{i}}^{\ast}\colon\c(U^-_w)\xrightarrow{\sim}\mathbb{C}(t_1,\dots, t_m)$. 
Hence we can regard the valuations $v_{\bm i} ^{\rm low}$ and $\tilde{v}_{\bm i} ^{\rm low}$ on $\c(X(w))$ defined in \cref{d:lowest_term_valuation_Schubert} as the ones on $\c(U^-_w)$. Note that we have reviewed an upper cluster algebra structure on $\c [U^-_w]$ 
in Section \ref{s:unipcell}, which induces a valuation on $\c(U^-_w)$ as in \cref{d:main_valuation}.

In this section, we clarify relations between these two kinds of valuations, and deduce some properties of Newton--Okounkov bodies associated with the valuations arising from the cluster structure. 
As an application, we construct a family of toric degenerations of $X(w)$ parametrized by the set of FZ-seeds for $\c [U^-_w]$. 

\begin{ntn}\label{n:fix_w}
    Throughout this section, we fix $w\in W$ and its reduced word $\bm{i}=(i_1, \ldots, i_m)\in R(w)$ (unless otherwise specified as in examples). We consider the cluster structure on $\mathbb{C}[U^-_w]$ given in \cref{t:upperBruhat}. In particular, we fix the index set for the FZ-seed $\mathbf{s}_{\bm{i}}$ as $J=\{1,\dots, m\}$, and sometimes identify $\mathbb{Z}^J$ with $\mathbb{Z}^m$ in an obvious way.
\end{ntn}

\subsection{Relation with string polytopes} 

In this subsection, we relate the valuation $\tilde{v}_{\bm i}^{\rm low}$ with $v_{\mathbf{s}_{\bm{i}}}$ associated with some specific refinement of $\preceq_{\varepsilon^{\bm{i}}}^{\rm op}$ (see \cref{d:order,d:main_valuation}). 
\begin{lem}[{\cite[Lemma 6.4]{BZ1}}]\label{l:monomial}
Let $u\in W$, $\bm{j}=(j_1,\dots, j_l)\in R(u)$, and $\lambda\in P_+$. Then the following equality holds:
	\begin{align*}
	y_{\bm{j}}^{\ast}(D_{u\lambda, \lambda})=t_1^{b_1}t_2^{b_2}\cdots t_{l}^{b_l},
	\end{align*}
	where $b_k \coloneqq \langle h_{j_k}, s_{j_{k+1}}\cdots s_{j_l}\lambda\rangle$ for $1 \leq k \leq l$.
\end{lem}

\begin{prop}\label{p:transform}
	For $s\in J$, the following equality holds:
	\[
	\tilde{v}_{\bm i}^{\rm low}(D(s, \bm{i}))=(d_1^{(s)}, d_2^{(s)},\dots, d_m^{(s)}),\text{ where }d_k^{(s)}\coloneqq \begin{cases}
	\langle h_{i_k}, s_{i_{k+1}}\cdots s_{i_s}\varpi_{i_s}\rangle&\text{if }k\leq s,\\
	0&\text{if }k>s.
	\end{cases}
	\]
\end{prop}

\begin{proof}
    Let $\iota \colon \mathbb{C}[t_1,\dots, t_s]\hookrightarrow \mathbb{C}[t_1,\dots, t_m]$ denote the inclusion homomorphism, and consider the evaluation homomorphism $\mathrm{ev}_0\colon \mathbb{C}[t_1,\dots, t_m]\to \mathbb{C}[t_1,\dots, t_s]$ given by 
	\[
	t_k\mapsto \begin{cases}
	t_k&\text{if }k\leq s,\\
	0&\text{if }k>s.
		\end{cases}
	\]
	Since we have $y_{\bm{i}}^{\ast}(D(s, \bm{i}))\in \mathbb{C}[t_1,\dots, t_m]$, we can consider $\mathrm{ev}_0(y_{\bm{i}}^{\ast}(D(s, \bm{i})))$. 
	By the definition of $\tilde{v}_{\bm i}^{\rm low}$, we have  
	\[
	\tilde{v}_{\bm i}^{\rm low}(\iota(\mathrm{ev}_0(y_{\bm{i}}^{\ast}(D(s, \bm{i})))))=\tilde{v}_{\bm i}^{\rm low}(y_{\bm{i}}^{\ast}(D(s, \bm{i})))
	\]
	if $\mathrm{ev}_0(y_{\bm{i}}^{\ast}(D(s, \bm{i})))\neq 0$. 
	It follows from \cref{l:monomial} and from the definition of $y_{\bm{i}}$ that  
	\begin{align*}
	t_1^{d_1}t_2^{d_2}\cdots t_{s}^{d_s}
	=\iota(y_{(i_1,\dots, i_s)}^{\ast}(D(s, \bm{i})))
	=\iota(\mathrm{ev}_0(y_{\bm{i}}^{\ast}(D(s, \bm{i})))), 
	\end{align*}
	where $d_k \coloneqq \langle h_{i_k}, s_{i_{k+1}}\cdots s_{i_s}\varpi_{i_s}\rangle$ for $1 \leq k \leq s$. Hence we have 
	\[
	\tilde{v}_{\bm i}^{\rm low}(y_{\bm{i}}^{\ast}(D(s, \bm{i})))
	=(d_1, \ldots, d_s, 0, \ldots, 0),
	\]
	which completes the proof of the proposition.
\end{proof}
\begin{prop}\label{p:refine}
	The total order $\prec$ on $\mathbb{Z}^m=\mathbb{Z}^J$ defined in \cref{d:lowest_term_valuation_Schubert} refines the partial order $\preceq_{\varepsilon^{\bm{i}}}^{\rm op}$. 
\end{prop}
\begin{proof}
	It suffices to show that 
	\begin{equation}\label{eq:refinement_desired_claim}
	\begin{aligned}
	(0, 0,\dots, 0)\prec (\varepsilon_{s, 1}, \varepsilon_{s, 2},\dots, \varepsilon_{s, m})
	\end{aligned}
	\end{equation}
	for all $s\in J_{\rm uf}$. Fix an arbitrary $s\in J_{\rm uf}$. By the definition of $\varepsilon^{\bm i}$, we see that
	\[
	\{t\in J \mid \varepsilon_{s,t}\neq 0 \}=\{s^+\}\cup \{t\mid s<t<s^{+}<t^{+}\text{ or }t<s<t^+<s^+\text{ or }t^+=s \}. 
	\]
	Hence it follows that $\max \{ t\in J \mid \varepsilon_{s, t}\neq 0 \}=s^+$. Moreover, we have $\varepsilon_{s,s^+}=1$. From these, we conclude \eqref{eq:refinement_desired_claim}. This implies the proposition.
\end{proof}

The upper cluster algebra structure on $\mathbb{C}[U^-_w]$ gives an identification 
\[
\mathbb{C}(X(w))\simeq \mathbb{C}(U^-_w)\simeq \mathbb{C}(D(1, \bm{i}),\dots, D(m, \bm{i})).
\]
The lexicographic order $\prec$ on $\mathbb{Z}^m$ defines a total order $\preceq_{\bm{i}}$ on the set of Laurent monomials in $D(1, \bm{i}),\dots, D(m, \bm{i})$ by identifying $(a_1,\dots, a_m)\in \mathbb{Z}^m$ with $D(1, \bm{i})^{a_1}\cdots D(m, \bm{i})^{a_m}$. By \cref{p:refine}, we can use it to define a valuation $v_{\mathbf{s}_{\bm{i}}}$ in \cref{d:main_valuation}. Using the notation in \cref{p:transform}, we 
define a $J \times J$-matrix $M_{\bm i}$ by $M_{\bm{i}}\coloneqq (d_{t}^{(s)})_{s, t\in J}$, where $s$ is a row index, and $t$ is a column index. 

\begin{thm}\label{t:relstring}
	Consider the valuation $v_{\mathbf{s}_{\bm{i}}}$ on $\mathbb{C}(X(w))$ defined above. Then the equality
	\[
	\tilde{v}_{\bm i}^{\rm low}(f)=v_{\mathbf{s}_{\bm{i}}}(f)M_{\bm{i}}
	\] 	
	holds in $\z^J$ for $f\in \mathbb{C}(X(w)) \setminus \{0\}$. Here elements of $\mathbb{Z}^J$ are regarded as $1\times J$-matrices. 
\end{thm}

\begin{proof}
	It suffices to show that 
	\[
	\tilde{v}_{\bm i}^{\rm low}(f)=v_{\mathbf{s}_{\bm{i}}}(f)M_{\bm{i}}
	\]
	for a nonzero polynomial $f=\sum_{\bm{a}=(a_1,\dots, a_m)\in \mathbb{Z}_{\geq 0}^m}c_{\bm{a}}D(1, \bm{i})^{a_1}\cdots D(m, \bm{i})^{a_m} \in \c[D(1, \bm{i}), \dots, D(m, \bm{i})]$, where $c_{\bm{a}}\in \mathbb{C}$. By \cref{p:transform}, we have 
	\begin{align*}
	\tilde{v}_{\bm i}^{\rm low}(D(1, \bm{i})^{a_1}\cdots D(m, \bm{i})^{a_m})&=
	\sum_{s=1}^ma_s\tilde{v}_{\bm i}^{\rm low}(D(s, \bm{i}))\\
	&=\sum_{s=1}^ma_s(0,\dots,\overset{\overset{s}{\vee}}{1},\dots, 0)M_{\bm{i}}\\
	&=\sum_{s=1}^ma_sv_{\mathbf{s}_{\bm{i}}}(D(s, \bm{i}))M_{\bm{i}}\\
	&=v_{\mathbf{s}_{\bm{i}}}(D(1, \bm{i})^{a_1}\cdots D(m, \bm{i})^{a_m})M_{\bm{i}}
	\end{align*}
	for all $\bm{a}=(a_1,\dots, a_m)\in \mathbb{Z}_{\geq 0}^m$. Since $M_{\bm{i}}$ is lower unitriangular as an $m \times m$-matrix, this implies that  
\[
\tilde{v}_{\bm i}^{\rm low}(f)=v_{\mathbf{s}_{\bm{i}}}(f)M_{\bm{i}},
\]
which proves the theorem. 
\end{proof}

Combining Theorems \ref{t:NO_body_crystal_basis}, \ref{t:saturatedness_of_semigroups}, and \ref{t:relstring} with the lower-unitriangularity of $M_{\bm i}$, we obtain the following.

\begin{cor}\label{c:relstring}
Let $\lambda \in P_+$. 
\begin{enumerate}
\item[{\rm (1)}] The equalities 
\begin{align*}
&S(X(w), \mathcal{L}_\lambda, \tilde{v}_{\bm i} ^{\rm low}, \tau_\lambda) = \{(k, {\bm a} M_{\bm i}) \mid (k, {\bm a}) \in S(X(w), \mathcal{L}_\lambda, v_{\mathbf{s}_{\bm{i}}}, \tau_\lambda)\},\\
&C(X(w), \mathcal{L}_\lambda, \tilde{v}_{\bm i} ^{\rm low}, \tau_\lambda) = \{(k, {\bm a} M_{\bm i}) \mid (k, {\bm a}) \in C(X(w), \mathcal{L}_\lambda, v_{\mathbf{s}_{\bm{i}}}, \tau_\lambda)\},\ {\it and}\\
&\Delta(X(w), \mathcal{L}_\lambda, \tilde{v}_{\bm i} ^{\rm low}, \tau_\lambda) = \Delta(X(w), \mathcal{L}_\lambda, v_{\mathbf{s}_{\bm{i}}}, \tau_\lambda) M_{\bm i}
\end{align*}
hold. In particular, the Newton--Okounkov body $\Delta(X(w), \mathcal{L}_\lambda, v_{\mathbf{s}_{\bm{i}}}, \tau_\lambda)$ is a rational convex polytope which is unimodularly equivalent to the string polytope $\Delta_{\bm i} (\lambda)$. 
\item[{\rm (2)}] The real closed cone $C(X(w), \mathcal{L}_\lambda, v_{\mathbf{s}_{\bm i}}, \tau_\lambda)$ is a rational convex polyhedral cone.
\item[{\rm (3)}] The following equality holds: 
\begin{align*}
&S(X(w), \mathcal{L}_\lambda, v_{{\bf s}_{\bm i}}, \tau_\lambda) = C(X(w), \mathcal{L}_\lambda, v_{{\bf s}_{\bm i}}, \tau_\lambda) \cap (\z_{>0} \times \z^m).
\end{align*}
\end{enumerate}
\end{cor}

Recall from Section \ref{ss:bases_tropical} the cluster cone $C_{{\bf s}_{\bm i}}$ of $X(w)$ associated with the seed ${\bf s}_{\bm i}$. 
By \cref{c:relation_with_string_cones} and \cref{t:relstring}, we obtain the following.

\begin{cor}\label{c:relations_string_cone}
The equality 
\[
\widetilde{C}_{\bm i} = C_{{\bf s}_{\bm i}} M_{\bm i}
\]
holds. 
In particular, the cluster cone $C_{{\bf s}_{\bm i}}$ is unimodularly equivalent to the string cone associated with ${\bm i}$, and it holds that
\[
C_{{\bf s}_{\bm i}} = \bigcup_{\lambda \in P_+} \Delta(X(w), \mathcal{L}_\lambda, v_{{\bf s}_{\bm i}}, \tau_\lambda).
\]
\end{cor} 

\begin{thm}\label{t:main_result_toric_degenerations}
Let $w \in W$, $\lambda \in P_+$, ${\bm i} \in R(w)$, and ${\bf s}$ an FZ-seed of $\c(U_w^-) = \c(X(w))$ which is mutation equivalent to ${\bf s}_{\bm i}$.
\begin{enumerate}
\item[{\rm (1)}] The Newton--Okounkov body $\Delta(X(w), \mathcal{L}_\lambda, v_{\bf s}, \tau_\lambda)$ is a rational convex polytope.
\item[{\rm (2)}] The following equality holds:
\begin{align*}
&\Delta(X(w), \mathcal{L}_\lambda, v_{\bf s}, \tau_\lambda) \cap \z^J = \{g_{\bf s}(b) \mid b \in {\bf B}_w [\lambda]\}.
\end{align*}
\item[{\rm (3)}] If $\mathcal{L}_\lambda$ is very ample on $X(w)$, then there exists a flat degeneration of $X(w)$ to the normal toric variety corresponding to the rational convex polytope $\Delta(X(w), \mathcal{L}_\lambda, v_{\bf s}, \tau_\lambda)$.
\end{enumerate}
\end{thm}

\begin{proof}
\cref{c:relstring} (2) implies that $C(X(w), \mathcal{L}_\lambda, v_{{\bf s}_{\bm i}}, \tau_\lambda)$ is a rational convex polyhedral cone. 
Hence we see by \cref{c:mutations_for_semigroups_and_cones} that $C(X(w), \mathcal{L}_\lambda, v_{\bf s}, \tau_\lambda)$ is given by a finite number of piecewise-linear inequalities, which implies that $C(X(w), \mathcal{L}_\lambda, v_{\bf s}, \tau_\lambda)$ is a finite union of rational convex polyhedral cones. 
However, since $C(X(w), \mathcal{L}_\lambda, v_{\bf s}, \tau_\lambda)$ is convex, this is a rational convex polyhedral cone, which implies part (1) of the theorem. 

Since we have 
\begin{align*}
&S(X(w), \mathcal{L}_\lambda, v_{{\bf s}_{\bm i}}, \tau_\lambda) = C(X(w), \mathcal{L}_\lambda, v_{{\bf s}_{\bm i}}, \tau_\lambda) \cap (\z_{>0} \times \z^J)
\end{align*}
by \cref{c:relstring} (3), we deduce by \cref{c:mutations_for_semigroups_and_cones} that 
\begin{align*}
&S(X(w), \mathcal{L}_\lambda, v_{\bf s}, \tau_\lambda) = C(X(w), \mathcal{L}_\lambda, v_{\bf s}, \tau_\lambda) \cap (\z_{>0} \times \z^J).
\end{align*}
This implies part (2) of the theorem by \eqref{eq:semigroups_from_g_vectors}.

Now part (3) of the theorem follows from Theorem \ref{t:toric deg} as follows. 
Since $C(X(w), \mathcal{L}_\lambda, v_{\bf s}, \tau_\lambda)$ is a rational convex polyhedral cone as we have proved above, the semigroup $S(X(w), \mathcal{L}_\lambda, v_{\bf s}, \tau_\lambda)$ is finitely generated and saturated by Gordan's lemma (see, for instance, \cite[Proposition 1.2.17]{CLS}). 
This implies that ${\rm Proj} (\c[S(X(w), \mathcal{L}_\lambda, v_{\bf s}, \tau_\lambda)])$ is normal by \cite[Theorem 1.3.5]{CLS}; hence if $\mathcal{L}_\lambda$ is very ample on $X(w)$, then ${\rm Proj} (\c[S(X(w), \mathcal{L}_\lambda, v_{\bf s}, \tau_\lambda)])$ is identical to the normal toric variety corresponding to the $|J|$-dimensional rational convex polytope $\Delta(X(w), \mathcal{L}_\lambda, v_{\bf s}, \tau_\lambda)$. 
Thus, we obtain part (3) by Theorem \ref{t:toric deg}.
\end{proof}

The following is also proved in the proof of \cref{t:main_result_toric_degenerations}.

\begin{cor}
Let $\lambda \in P_+$, and ${\bf s}$ an FZ-seed of $\c(U_w^-)$ which is mutation equivalent to ${\bf s}_{\bm i}$.
\begin{enumerate}
\item[{\rm (1)}] The real closed cone $C(X(w), \mathcal{L}_\lambda, v_{\bf s}, \tau_\lambda)$ is a rational convex polyhedral cone.
\item[{\rm (2)}] The following equality holds:
\begin{align*}
&S(X(w), \mathcal{L}_\lambda, v_{\bf s}, \tau_\lambda) = C(X(w), \mathcal{L}_\lambda, v_{\bf s}, \tau_\lambda) \cap (\z_{>0} \times \z^J).
\end{align*}
In particular, the semigroup $S(X(w), \mathcal{L}_\lambda, v_{\bf s}, \tau_\lambda)$ is finitely generated and saturated.
\end{enumerate}
\end{cor}

In a way similar to the proof of Theorem \ref{t:main_result_toric_degenerations}, we obtain the following by \cref{c:relations_string_cone}.

\begin{cor}
Let ${\bf s}$ be an FZ-seed of $\c(U_w^-)$ which is mutation equivalent to ${\bf s}_{\bm i}$. 
Then the cluster cone $C_{\bf s}$ is a rational convex polyhedral cone, and it holds that
\begin{align*}
&C_{\bf s} = \bigcup_{\lambda \in P_+} \Delta(X(w), \mathcal{L}_\lambda, v_{\bf s}, \tau_\lambda),\\
&C_{\bf s} \cap \z^J = v_{\bf s} (\c[U^- \cap X(w)] \setminus \{0\}).
\end{align*}
\end{cor}

\subsection{Relation with Nakashima--Zelevinsky polytopes}\label{s:mutation_sequence}

In this subsection, we relate the valuation $v_{\bm i}^{\rm low}$ with the one arising from the cluster structure. To do that, we need to find a nice FZ-seed $\mathbf{s}_{\bm{i}}^{\rm mut}$ by applying iterated mutations to $\mathbf{s}_{\bm{i}}$. For this purpose, we first summarize results on specific mutation sequences in Section \ref{ss:mutation_sequence} without proofs. Since the proofs of these statements are technical and might be known to experts, we give them in Appendix \ref{a:mutation_sequence}. Indeed, our first mutation sequence $\overleftarrow{\boldsymbol{\mu}_{\bm{i}}}$ was essentially introduced in \cite[Section 13.1]{GLS:Kac-Moody}. However, in \cite{GLS:Kac-Moody}, the Cartan matrix of $\mathfrak{g}$ is assumed to be \emph{symmetric}, and the explicit formula of the exchange matrix (in particular, the arrows among the frozen and unfrozen variables) after the mutation sequence is not given. Hence we provide the corresponding statements in \emph{symmetrizable} case; their proofs are given in Appendix \ref{a:mutation_sequence}.

\subsubsection{Specific mutation sequences}\label{ss:mutation_sequence}

We define an injective map $\xi_{\bm{i}}\colon J\to I\times \mathbb{Z}_{>0}$ by 
\[
\xi_{\bm{i}} (s) \coloneqq (i_s, k[s]),\text{ where } k[s] \coloneqq |\{1 \leq t \leq s\mid i_t = i_s\}|. 
\]
Let $\mathrm{pr}_1\colon I\times \mathbb{Z}_{>0}\to I$ denote the first projection. For $i\in \mathrm{pr}_1(\Image (\xi_{\bm{i}}))$, we have
\[
\Image (\xi_{\bm{i}})\cap \mathrm{pr}_1^{-1}(i)= \{(i, 1), (i, 2),\dots,  (i, m_i)\},
\]
where $m_i \coloneqq |\{1 \leq s \leq m \mid i_s = i\}|$.
Note that 
\[
\xi_{\bm{i}}(J_{\rm fr})=\{(i, m_i)\mid i\in \mathrm{pr}_1(\Image \xi_{\bm{i}}) \}. 
\]
We consider a mutation sequence $\overleftarrow{\boldsymbol{\mu}_{\bm{i}}}$ defined as follows:
\[
\overleftarrow{\boldsymbol{\mu}_{\bm{i}}}\coloneqq \overleftarrow{\boldsymbol{\mu}_{\bm{i}}}[m]\circ \cdots \circ \overleftarrow{\boldsymbol{\mu}_{\bm{i}}}[2]\circ \overleftarrow{\boldsymbol{\mu}_{\bm{i}}}[1],
\]
where 
\[
\overleftarrow{\boldsymbol{\mu}_{\bm{i}}}[s]\coloneqq 
\begin{cases}
\mu_{(i_s, m_{i_s}-k[s])}\circ \cdots \circ \mu_{(i_s, 2)}\circ \mu_{(i_s, 1)}&\text{if }s\not\in J_{\rm fr},\\
\mathrm{id}&\text{if }s\in J_{\rm fr}. 
\end{cases}
\]
We investigate this mutation sequence by using quivers. For a $J_{\rm uf}\times J$-matrix $(\varepsilon_{s, t})_{s \in J_{\rm uf}, t \in J}$ whose $J_{\rm uf}\times J_{\rm uf}$-submatrix is skew-symmetrizable, we define a quiver $\Gamma$ by the following rules:
\begin{itemize}
	\item[(I)] the vertex set is $J$;
	\item[(II)] write an arrow $s\to t$ if $\varepsilon_{s, t}<0$ or $\varepsilon_{t, s}>0$ for $s, t\in J$ (note that there exists no arrow between any two elements of $J_{\rm fr}$).
\end{itemize}
We denote by $\Gamma_{\bm{i}}$ the quiver associated with $\varepsilon^{\bm{i}}$. 

\begin{ex}\label{e:quiverexample}	
Let $w_0$ denote the longest element of the Weyl group $W$. See \cref{s:cluster_cone} for the labeling of the vertices of the Dynkin diagram. 

When $G = SL_4(\mathbb{C})$ (of type $A_3$) and $\bm{i}=(2, 1, 2, 3, 2, 1)\in R(w_0)$, the corresponding quiver $\Gamma_{\bm{i}}$ is described as the following:

\hfill
\begin{xy} 0;<1pt,0pt>:<0pt,-1pt>::
	(-18,0) *+{3},
(-10,0) *+{\scalebox{0.65}{$>$}},	
(-18,30) *+{2},
(-10,30) *+{\scalebox{0.65}{$>$}},		
(-18,60) *+{1},
(-10,60) *+{\scalebox{0.65}{$>$}},			
	(120,0) *+{4} ="1",
	(40,60) *+{2} ="3",
	(160,30) *+{5} ="4",
	(0,30) *+{1} ="6",
	(80,30) *+{3} ="7",
	(200,60) *+{6} ="8",
	"7", {\ar"6"},
	"4", {\ar"7"},
	"8", {\ar"3"},
	"6", {\ar"3"},
	"3", {\ar"4"},
	"7", {\ar"1"},
\end{xy}
\hfill
\hfill

	When $G = Sp_4(\c)$ (of type $C_2$) and $\bm{i}=(1, 2, 1, 2)\in R(w_0)$, the corresponding quiver $\Gamma_{\bm{i}}$ is given by the following:

\hfill
	\begin{xy} 0;<1pt,0pt>:<0pt,-1pt>::
			(-18,0) *+{2},
		(-10,0) *+{\scalebox{0.65}{$>$}},	
		(-18,30) *+{1},
		(-10,30) *+{\scalebox{0.65}{$>$}},		
		(0,30) *+{1} ="1",
		(40,0) *+{2} ="2",
		(80,30) *+{3} ="3",
		(120,0) *+{4} ="4",
		"1", {\ar"2"},
		"2", {\ar"3"},
		"3", {\ar"1"},
		"4", {\ar"2"},
	\end{xy}
\hfill
\hfill

\end{ex}
An arrow $s\to t$ in the quiver $\Gamma_{\bm{i}}$ is said to be \emph{horizontal} if the first components of $\xi_{\bm{i}}(s)$ and $\xi_{\bm{i}}(t)$ are the same. Otherwise, it is said to be \emph{inclined}. From $\Gamma_{\bm{i}}$, we can recover $\varepsilon^{\bm{i}}=(\varepsilon_{s, t})_{s\in J_{\rm uf}, t\in J}$ by the following rule: 
\begin{align}
	\varepsilon_{s, t}=
\begin{cases}
-1&\text{if there exists a horizontal arrow } s\to t, \\
c_{i_t,i_s}&\text{if there exists an inclined arrow } s\to t, \\
1&\text{if there exists a horizontal arrow } t\to s, \\
-c_{i_t,i_s}&\text{if there exists an inclined arrow } t\to s,\\
0&\text{otherwise}.
\end{cases}
\label{eq:recoverexchange}
\end{align}

We set $\bm{i}^{\mathrm{op}}\coloneqq (i_m,\dots, i_1)\in R(w^{-1})$. Then we obtain a bijection $\mathsf{R}_{\bm{i}}\colon J \to J$ given by $\mathsf{R}_{\bm{i}} \coloneqq (\xi_{\bm{i}^{\mathrm{op}}})^{-1}\circ \xi_{\bm{i}}$. Note that $\mathsf{R}_{\bm{i}}(J_{\rm uf}) = J_{\rm uf}$ and $\mathsf{R}_{\bm{i}}(J_{\rm fr}) = J_{\rm fr}$. The following theorem describes the exchange matrix of the FZ-seed $\overleftarrow{\boldsymbol{\mu}_{\bm{i}}}(\mathbf{s}_{\bm{i}})$. Its proof is given in Appendix \ref{a:mutation_sequence}. 

\begin{thm}\label{t:mutation_max}
	The $J_{\rm uf}\times J$-matrix $\overleftarrow{\boldsymbol{\mu}_{\bm{i}}}(\varepsilon^{\bm{i}})=(\overline{\varepsilon}_{s, t})_{s\in J_{\rm uf}, t\in J}$ is given by 
	\begin{align*}
	\overline{\varepsilon}_{s,t}&=
		\begin{cases}
	1&\text{if}\ s=t^+, \\
	-1&\text{if}\ s^+=t, \\
	c_{i_t, i_s}&\text{if}\ \mathsf{R}_{\bm{i}}(t)<\mathsf{R}_{\bm{i}}(s)<\mathsf{R}_{\bm{i}}(t^+)<\mathsf{R}_{\bm{i}}(s^+),\\
	-c_{i_t, i_s}&\text{if}\ \mathsf{R}_{\bm{i}}(s)<\mathsf{R}_{\bm{i}}(t)<\mathsf{R}_{\bm{i}}(s^+)<\mathsf{R}_{\bm{i}}(t^+), \\
    0&\text{otherwise}.
	\end{cases} 
	\end{align*}
\end{thm}

\begin{rem}\label{r:opquiver}
	The quiver associated with $\overleftarrow{\boldsymbol{\mu}_{\bm{i}}}(\varepsilon^{\bm{i}})$ is the same as the quiver obtained from $\Gamma_{\bm{i}^{\mathrm{op}}}$ by reversing directions of all arrows. 
\end{rem}
\begin{ex}\label{e:mutseq}
	For the examples in \cref{e:quiverexample}, we obtain the following.
	\begin{itemize}
	
\item The case of $G = SL_4(\mathbb{C})$ and $\bm{i}=(2, 1, 2, 3, 2, 1)\in R(w_0)$: we have 
\[
 \overleftarrow{\boldsymbol{\mu}_{\bm{i}}}[1]=\mu_3 \mu_1,\ 
 \overleftarrow{\boldsymbol{\mu}_{\bm{i}}}[2]=\mu_2,\ 
 \overleftarrow{\boldsymbol{\mu}_{\bm{i}}}[3]=\mu_1,\ 
 \overleftarrow{\boldsymbol{\mu}_{\bm{i}}}[4]=\overleftarrow{\boldsymbol{\mu}_{\bm{i}}}[5]=\overleftarrow{\boldsymbol{\mu}_{\bm{i}}}[6]=\mathrm{id}. 
\]
Hence $\overleftarrow{\boldsymbol{\mu}_{\bm{i}}} = \mu_1 \mu_2 \mu_3 \mu_1$, and the quiver associated with $\overleftarrow{\boldsymbol{\mu}_{\bm{i}}}(\varepsilon^{\bm{i}})$ is given by 
	
	\hfill
	\begin{xy} 0;<1pt,0pt>:<0pt,-1pt>::
		(-18,0) *+{3},
		(-10,0) *+{\scalebox{0.65}{$>$}},	
		(-18,30) *+{2},
		(-10,30) *+{\scalebox{0.65}{$>$}},		
		(-18,60) *+{1},
		(-10,60) *+{\scalebox{0.65}{$>$}},			
		(0,30) *+{1} ="1",
		(40,60) *+{2} ="2",
		(80,30) *+{3} ="3",
		(120,0) *+{4} ="4",		
		(160,30) *+{5} ="5",
		(200,60) *+{6} ="6",
		"1", {\ar"3"},
		"3", {\ar"5"},
		"3", {\ar"2"},
		"6", {\ar"3"},
		"2", {\ar"6"},
		"4", {\ar"1"},
	\end{xy}
	\hfill
	\hfill
	
\item The case of $G = Sp_4(\c)$ and $\bm{i}=(1, 2, 1, 2)\in R(w_0)$: we have 
\[
 \overleftarrow{\boldsymbol{\mu}_{\bm{i}}}[1]=\mu_1,\ 
 \overleftarrow{\boldsymbol{\mu}_{\bm{i}}}[2]=\mu_2,\ 
 \overleftarrow{\boldsymbol{\mu}_{\bm{i}}}[3]=\overleftarrow{\boldsymbol{\mu}_{\bm{i}}}[4]=\mathrm{id}. 
\]
Hence $\overleftarrow{\boldsymbol{\mu}_{\bm{i}}} = \mu_2 \mu_1$, and the quiver associated with $\overleftarrow{\boldsymbol{\mu}_{\bm{i}}}(\varepsilon^{\bm{i}})$ is given by 

	\hfill
	\begin{xy} 0;<1pt,0pt>:<0pt,-1pt>::
		(-18,0) *+{2},
		(-10,0) *+{\scalebox{0.65}{$>$}},	
		(-18,30) *+{1},
		(-10,30) *+{\scalebox{0.65}{$>$}},		
		(0,30) *+{1} ="1",
		(40,0) *+{2} ="2",
		(80,30) *+{3} ="3",
		(120,0) *+{4} ="4",
		"1", {\ar"2"},
		"1", {\ar"3"},
		"2", {\ar"4"},
		"4", {\ar"1"},
	\end{xy}
	\hfill
	\hfill
	
	\end{itemize}	
\end{ex}

Next we describe the explicit form of $\overleftarrow{\boldsymbol{\mu}_{\bm{i}}}(\mathbf{D}_{\bm{i}})$. Its proof is again postponed to Appendix \ref{a:mutation_sequence}. 
\begin{thm}\label{t:mutation_max_variable}
The equality
$\overleftarrow{\boldsymbol{\mu}_{\bm{i}}}(\mathbf{D}_{\bm{i}})=(D^{\vee}(s, \bm{i}) )_{s\in J}$ holds, where 
\[
D^{\vee}(s, \bm{i})\coloneqq\begin{cases}
D_{w\varpi_{i_s}, w_{\leq s^{\vee}}\varpi_{i_s}}\text{ with }s^{\vee}\coloneqq \xi_{\bm{i}}^{-1}(i_s, m_{i_s}-k[s])&\text{if }k[s]<m_{i_s},\\
D_{w\varpi_{i_s}, \varpi_{i_s}}&\text{if }k[s]=m_{i_s}.
\end{cases}
\]
\end{thm}

\begin{rem}\label{r:FZ-twist}
By \cref{t:mutation_max,t:mutation_max_variable}, there exists an isomorphism of $\mathbb{C}$-algebras $\mathbb{C}[U_{w^{-1}}^{-}]\to \mathbb{C}[U_w^{-}]$ compatible with the cluster structures given by 
\[
D(s, \bm{i}^{\rm op})\mapsto D^{\vee}(s, \bm{i})
\]
for $s\in \{1, \ldots, \ell(w^{-1})(=m)\}$. Indeed, it coincides with the map induced from the map \cite[(2.56) and Lemma 2.25]{FZ:Double} (see \cite[Proposition 8.5]{GLS:Kac-Moody} and \cite[Corollary 2.22]{KimOya} for the relation between unipotent cells and unipotent subgroups). Moreover, the quantum lift of this isomorphism becomes an anti-isomorphism which is studied in \cite{LY,KimOya:qtwist}. See \cite[Theorem 3.22]{KimOya:qtwist}. By \cite[Theorem 3.8]{KimOya:qtwist}, it preserves the upper global basis.
\end{rem}

\subsubsection{Seeds adapted to Nakashima--Zelevinsky polytopes}

As a prerequisite, we recall an automorphism $\eta_w^{\ast}$ on $\mathbb{C}[U^-_w]$, called the \emph{Berenstein--Fomin--Zelevinsky twist automorphism}. 
\begin{thm}[{\cite[Lemma 1.3]{BFZ1} and \cite[Theorem 1.2]{BZ1} (cf.~\cite[Proposition 2.23]{KimOya})}]\label{t:twistautom}
	There exists a $\mathbb{C}$-algebra automorphism $\eta_w^{\ast}\colon \mathbb{C}[U^-_w]\to \mathbb{C}[U^-_w]$ given by 
	\[
	D_{v^{\vee}, v_{\lambda}}\mapsto \frac{D_{f_{w\lambda}, v}}{D_{w\lambda, \lambda}}
	\]
	for $\lambda\in P_+$ and $v\in V(\lambda)$. 
\end{thm}
\begin{rem}
	The automorphism $\eta_w^{\ast}$ is induced from a regular automorphism $\eta_w\colon U^-_w\to U^-_w$ given by 
	\[
	u\mapsto [u^T\overline{w}]_-,
	\]
	where $u^T$ is the transpose of $u$ in $G$, and $[u^T\overline{w}]_-$ is the $U^-$-component of $u^T\overline{w}\in U^- B$. 
\end{rem}
Recall the isomorphism $y_{\bm{i}}^{\ast}\colon \mathbb{C}(U^-_w) \xrightarrow{\sim} \mathbb{C}(t_1,\dots, t_m)$ explained in the beginning of Section \ref{s:rel_with_stringNZ}. 
The Berenstein--Fomin--Zelevinsky twist automorphism $\eta_w^{\ast}$ is used for describing the variables $t_s$, $s\in J$, in terms of unipotent minors. The formula for $t_s$ in the following theorem is called the \emph{Chamber Ansatz formula}. 

\begin{thm}[{\cite[Theorems 1.4, 2.7.1]{BFZ1} and \cite[Theorems 1.4, 4.3]{BZ1} (cf.\ \cite[Section 1.3]{Oya})}]\label{t:Chamber_Ansatz}
For $s\in J$, the equality
\[
t_s=y_{\bm{i}}^{\ast}\left(\frac{\prod_{t<s<t^+}(\eta_w^{\ast})^{-1}(D(t, \bm{i}))^{-c_{i_t, i_s}}}{(\eta_w^{\ast})^{-1}(D(s^-, \bm{i}))(\eta_w^{\ast})^{-1}(D(s, \bm{i}))}\right)
\]
holds, where $D(0, \bm{i})\coloneqq 1$. Moreover, these formulas are equivalent to the equalities
\[
(y_{\bm{i}}^{\ast}\circ (\eta_w^{\ast})^{-1})(D(s, \bm{i}))=t_1^{-d_1^{(s)}}\cdots t_s^{-d_s^{(s)}}
\]
for $s\in J$, where $d_k^{(s)}\coloneqq \langle h_{i_k}, s_{i_{k+1}}\dots s_{i_s}\varpi_{i_s}\rangle$ for $1 \le k \le s$.
\end{thm}  

Since $\eta_w^{\ast}\colon \mathbb{C}[U^-_w]\to \mathbb{C}[U^-_w]$ is a $\mathbb{C}$-algebra automorphism, it induces a $\mathbb{C}$-algebra automorphism $\eta_w^{\ast}\colon \mathbb{C}(U^-_w)\to \mathbb{C}(U^-_w)$. Hence, for each FZ-seed $\mathbf{s}=(\mathbf{A}=(A_j)_{j\in J}, \varepsilon)$ of $\mathbb{C}(U^-_w)$, we obtain a new FZ-seed 
\begin{align*}
(\eta_w^{\ast})^{-1}(\mathbf{s})\coloneqq (\{(\eta_w^{\ast})^{-1}(A_j)\}_{j\in J}, \varepsilon)
\end{align*}
of $\mathbb{C}(U^-_w)$. Moreover, if $\mathbf{s}$ satisfies $\mathscr{U}(\mathbf{s})=\mathbb{C}[U^-_w]$, then $\mathscr{U}((\eta_w^{\ast})^{-1}(\mathbf{s}))=\mathbb{C}[U^-_w]$. Note that we have
\begin{align}
    (\eta_w^{\ast})^{-1}(\mu_{s}(\mathbf{s}))=\mu_{s}((\eta_w^{\ast})^{-1}(\mathbf{s}))\text{ for all }s\in J_{\rm uf}.\label{eq:twistmutation}
\end{align}
For $\bm{i}\in R(w)$, we set 
\[
\mathbf{s}_{\bm{i}}^{\rm mod}\coloneqq(\eta_w^{\ast})^{-1}(\overleftarrow{\boldsymbol{\mu}_{\bm{i}}}(\mathbf{s}_{\bm{i}})).
\]

\begin{prop}\label{p:modified_seed}
Write 
\[
\mathbf{s}_{\bm{i}}^{\rm mod}=((D^{\rm mod}(s,\bm{i}))_{s\in J}, \varepsilon^{\bm{i},{\rm mod}}=(\varepsilon'_{s, t})_{s\in J_{\rm uf}, t\in J}).
\]
Then the following equalities hold (here $s^\vee \coloneqq 0$ if $k[s]=m_{i_s}$):
\[
D^{\rm mod}(s,\bm{i})=\frac{D_{w_{\leq s^{\vee}}\varpi_{i_s}, \varpi_{i_s}}}{D_{w\varpi_{i_s}, \varpi_{i_s}}},
\]
and 
\[
\varepsilon'_{s, t}=
		\begin{cases}
	1&\text{if}\ s=t^+, \\
	-1&\text{if}\ s^+=t, \\
	c_{i_t, i_s}&\text{if}\ \mathsf{R}_{\bm{i}}(t)<\mathsf{R}_{\bm{i}}(s)<\mathsf{R}_{\bm{i}}(t^+)<\mathsf{R}_{\bm{i}}(s^+),\\
	-c_{i_t, i_s}&\text{if}\ \mathsf{R}_{\bm{i}}(s)<\mathsf{R}_{\bm{i}}(t)<\mathsf{R}_{\bm{i}}(s^+)<\mathsf{R}_{\bm{i}}(t^+), \\
    0&\text{otherwise};
	\end{cases} 
\]
recall the notation in \cref{t:mutation_max,t:mutation_max_variable}.
\end{prop}

\begin{proof}
By \cref{t:mutation_max_variable}, the cluster variable of  $\overleftarrow{\boldsymbol{\mu}_{\bm{i}}}(\mathbf{s}_{\bm{i}})$ at $s\in J$ is $D^{\vee}(s, \bm{i})=D_{w\varpi_{i_s}, w_{\leq s^{\vee}}\varpi_{i_s}}$. Hence we have 
\[
D^{\rm mod}(s,\bm{i})=(\eta_w^{\ast})^{-1}(D_{w\varpi_{i_s}, w_{\leq s^{\vee}}\varpi_{i_s}})=(\eta_w^{\ast})^{-1}\left(\frac{D_{w\varpi_{i_s}, w_{\leq s^{\vee}}\varpi_{i_s}}}{D_{w\varpi_{i_s}, \varpi_{i_s}}}D_{w\varpi_{i_s}, \varpi_{i_s}}\right)=\frac{D_{w_{\leq s^{\vee}}\varpi_{i_s}, \varpi_{i_s}}}{D_{w\varpi_{i_s}, \varpi_{i_s}}}. 
\]
The explicit form of $\varepsilon'_{s, t}$ immediately follows from \cref{t:mutation_max}. 
\end{proof}

Write $\varepsilon^{\bm{i}}=(\varepsilon_{s, t})_{s\in J_{\rm uf}, t\in J}$ and $\varepsilon^{\bm{i},{\rm mod}}=(\varepsilon'_{s, t})_{s\in J_{\rm uf}, t\in J}$. Set 
\begin{align*}
\widehat{X}_{s; \bm{i}} \coloneqq \prod_{t \in J} D(t, \bm{i})^{\varepsilon_{s, t}}, \qquad
\widehat{X}_{s; \bm{i}}^{\rm mod} \coloneqq \prod_{t \in J} D^{\rm mod}(t,\bm{i})^{\varepsilon'_{s, t}} 
\end{align*}
for $s \in J_{\rm uf}$. 

\begin{thm}\label{t:Xhat_modified}
$\widehat{X}_{s^{\vee}; \bm{i}}^{\rm mod}=\widehat{X}_{s; \bm{i}}$ for all $s \in J_{\rm uf}$.
\end{thm}

\begin{proof}
If we regard the index set $J$ for ${\bf s}_{\bm{i}}$ as $\xi_{\bm{i}}(J)$ via $\xi_{\bm{i}}$, then the FZ-seed 
\[
{\bf s}_{\bm{i}}=((D(s, \bm{i}))_{s\in \xi_{\bm{i}}(J)}, (\varepsilon_{s, t})_{s\in \xi_{\bm{i}}(J_{\rm uf}), t\in \xi_{\bm{i}}(J)})
\]
is described as follows (here $\xi_{\bm{i}}^{-1}((i,m_i+1))\coloneqq m+1$ for all $i\in I$):
	\[
	\varepsilon_{(i, k), (j,\ell)}=
	\begin{cases}
	-1&\text{if}\ i=j\text{ and }k=\ell+1, \\
	1&\text{if}\  i=j\text{ and }k+1=\ell, \\
	-c_{j, i}&\text{if}\ \xi_{\bm{i}}^{-1}(j,\ell)<\xi_{\bm{i}}^{-1}(i,k)<\xi_{\bm{i}}^{-1}(j,\ell+1)<\xi_{\bm{i}}^{-1}(i,k+1),\\
	c_{j, i}&\text{if}\ \xi_{\bm{i}}^{-1}(i,k)<\xi_{\bm{i}}^{-1}(j,\ell)<\xi_{\bm{i}}^{-1}(i,k+1)<\xi_{\bm{i}}^{-1}(j,\ell +1), \\
	0&\text{otherwise},
	\end{cases}
	\]
	and 
\[
	D((i, k), \bm{i})=D_{w_{\leq s}\varpi_i, \varpi_i},\text{where }s=\xi_{\bm{i}}^{-1}(i,k). 
\]	
If we regard the index set $J$ for ${\bf s}_{\bm{i}}^{\rm mod}$ as $\xi_{\bm{i}}(J)$ via $\xi_{\bm{i}}$, then the FZ-seed 
\[
\mathbf{s}_{\bm{i}}^{\rm mod}=((D^{\rm mod}(s, \bm{i}))_{s\in \xi_{\bm{i}}(J)}, (\varepsilon'_{s, t})_{s\in \xi_{\bm{i}}(J_{\rm uf}), t\in \xi_{\bm{i}}(J)})
\]
is described as follows (here $\xi_{\bm{i}^{\rm op}}^{-1}((i,m_i+1))\coloneqq m+1$, and $\xi_{\bm{i}}^{-1}((i,0))\coloneqq 0$ for all $i\in I$):
	\[
	\varepsilon'_{(i, k), (j,\ell)}=
	\begin{cases}
	1&\text{if}\ i=j\text{ and }k=\ell+1, \\
	-1&\text{if}\  i=j\text{ and }k+1=\ell, \\
	c_{j, i}&\text{if}\ \xi_{\bm{i}^{\rm op}}^{-1}(j,\ell)<\xi_{\bm{i}^{\rm op}}^{-1}(i,k)<\xi_{\bm{i}^{\rm op}}^{-1}(j,\ell+1)<\xi_{\bm{i}^{\rm op}}^{-1}(i,k+1),\\
	-c_{j, i}&\text{if}\ \xi_{\bm{i}^{\rm op}}^{-1}(i,k)<\xi_{\bm{i}^{\rm op}}^{-1}(j,\ell)<\xi_{\bm{i}^{\rm op}}^{-1}(i,k+1)<\xi_{\bm{i}^{\rm op}}^{-1}(j,\ell +1), \\
	0&\text{otherwise},
	\end{cases}
	\]
	and 
\[
	D^{\rm mod}((i, k), \bm{i})=\frac{D_{w_{\leq s}\varpi_i, \varpi_i}}{D_{w\varpi_i, \varpi_i}}\quad \text{with }s=\xi_{\bm{i}}^{-1}(i,m_i-k), 
\]	
where we note that $\xi_{\bm{i}}(J)=\xi_{\bm{i}^{\rm op}}(J)$. For $(i, k), (j,\ell)\in \xi_{\bm{i}}(J_{\rm uf})=\xi_{\bm{i}^{\rm op}}(J_{\rm uf})$, the condition 
\[
\xi_{\bm{i}^{\rm op}}^{-1}(j,\ell)<\xi_{\bm{i}^{\rm op}}^{-1}(i,k)<\xi_{\bm{i}^{\rm op}}^{-1}(j,\ell+1)<\xi_{\bm{i}^{\rm op}}^{-1}(i,k+1)
\]
is equivalent to 
\[
\xi_{\bm{i}}^{-1}(i,m_i-k)<\xi_{\bm{i}}^{-1}(j,m_j-\ell)<\xi_{\bm{i}}^{-1}(i,m_i-k+1)<\xi_{\bm{i}}^{-1}(j,m_j-\ell +1).
\]
Moreover, for $(i, k)\in  \xi_{\bm{i}}(J_{\rm uf})=\xi_{\bm{i}^{\rm op}}(J_{\rm uf})$ and $(j, m_j)\in  \xi_{\bm{i}}(J_{\rm fr})=\xi_{\bm{i}^{\rm op}}(J_{\rm fr})$, the condition
\[
\xi_{\bm{i}^{\rm op}}^{-1}(i,k)<\xi_{\bm{i}^{\rm op}}^{-1}(j,m_j)<\xi_{\bm{i}^{\rm op}}^{-1}(i,k+1)
\]
is equivalent to 
\[
\xi_{\bm{i}}^{-1}(i,m_i-k)<\xi_{\bm{i}}^{-1}(j,1)<\xi_{\bm{i}}^{-1}(i,m_i-k+1).
\]
Hence if we consider a bijection $\xi^{\vee}\colon \xi_{\bm{i}}(J)\to \xi_{\bm{i}}(J)$ given by 
\[
(i, k)\mapsto \begin{cases}
(i, m_i-k)&\text{if }k<m_i,\\
(i, m_{i})&\text{if }k=m_i,  
\end{cases}
\]
then we can describe $\varepsilon'_{s, t}$ as follows:
\begin{align*}
    \varepsilon'_{\xi^{\vee}(i, k), \xi^{\vee}(j,\ell)}&=
	\begin{cases}
	1&\text{if}\ i=j\text{ and }k+1=\ell<m_j, \\
	-1&\text{if}\  i=j\text{ and }k=\ell+1, \\
	-1&\text{if}\ i=j,\ k=1,\ \text{ and }\ell= m_j, \\
	-c_{j, i}&\text{if}\ \xi_{\bm{i}}^{-1}(j,\ell)<\xi_{\bm{i}}^{-1}(i,k)<\xi_{\bm{i}}^{-1}(j,\ell +1)<\xi_{\bm{i}}^{-1}(i,k+1), \\	
	-c_{j, i}&\text{if}\ \xi_{\bm{i}}^{-1}(i,k)<\xi_{\bm{i}}^{-1}(j,1)<\xi_{\bm{i}}^{-1}(i,k+1)\text{ and }\ell=m_j, \\	
	c_{j, i}&\text{if}\ \xi_{\bm{i}}^{-1}(i,k)<\xi_{\bm{i}}^{-1}(j,\ell)<\xi_{\bm{i}}^{-1}(i,k+1)<\xi_{\bm{i}}^{-1}(j,\ell +1)\text{ and }\ell<m_j,\\
	0&\text{otherwise}
	\end{cases}\\
	&=\begin{cases}
	\varepsilon_{(i, k), (j,\ell)}&\text{if}\ (j,\ell)\in \xi_{\bm{i}}(J_{\rm uf}), \\
	-1&\text{if}\ i=j,\ k=1,\ \text{ and }\ell= m_j, \\
    -c_{j, i}&\text{if}\ \xi_{\bm{i}}^{-1}(i,k)<\xi_{\bm{i}}^{-1}(j,1)<\xi_{\bm{i}}^{-1}(i,k+1)\text{ and }\ell=m_j, \\	
	0&\text{otherwise}.
	\end{cases}
\end{align*}
Moreover, it follows that
\[
	D^{\rm mod}(\xi^{\vee}(i, k), \bm{i})=
	\begin{cases}
	\displaystyle\frac{D_{w_{\leq s}\varpi_i, \varpi_i}}{D_{w\varpi_i, \varpi_i}}\quad \text{with }s=\xi_{\bm{i}}^{-1}(i,k)&\text{if }k<m_i,\\
	\displaystyle\frac{1}{D_{w\varpi_{i}, \varpi_{i}}}&\text{if }k=m_i.
	\end{cases}
\]	
Hence, for $(i,k)\in \xi_{\bm{i}}(J_{\rm uf})$, we have
\begin{align*}
\widehat{X}_{\xi^{\vee}(i, k); \bm{i}}^{\rm mod}
    &=\prod_{(j,\ell) \in \xi_{\bm{i}}(J)} D^{\rm mod}(\xi^{\vee}(j,\ell),\bm{i})^{\varepsilon'_{\xi^{\vee}(i,k), \xi^{\vee}(j,\ell)}}\\
    &=\left(\prod_{(j,\ell) \in \xi_{\bm{i}}(J_{\rm uf})} D^{\rm mod}(\xi^{\vee}(j,\ell),\bm{i})^{\varepsilon_{(i,k), (j,\ell)}}\right)\\
    &\phantom{=}\cdot\left(\prod_{\substack{j\in I;\\ \xi_{\bm{i}}^{-1}(i,k)<\xi_{\bm{i}}^{-1}(j,1)<\xi_{\bm{i}}^{-1}(i,k+1)}} D^{\rm mod}(\xi^{\vee}(j,m_j),\bm{i})^{-c_{j, i}}\right)D^{\rm mod}(\xi^{\vee}(i,m_i),\bm{i})^{-\delta_{k,1}}\\
    &=\left(\prod_{(j,\ell) \in \xi_{\bm{i}}(J_{\rm uf})} \left(\frac{D((j,\ell),\bm{i})}{D_{w\varpi_j, \varpi_j}}\right)^{\varepsilon_{(i,k), (j,\ell)}}\right)\left(\prod_{\substack{j\in I;\\\xi_{\bm{i}}^{-1}(i,k)<\xi_{\bm{i}}^{-1}(j,1)<\xi_{\bm{i}}^{-1}(i,k+1)}} D_{w\varpi_j, \varpi_j}^{c_{j, i}}\right)D_{w\varpi_i, \varpi_i}^{\delta_{k,1}}.
\end{align*}
Here the explicit formula of $\varepsilon_{(i, k), (j,\ell)}$ implies that  
\begin{align*}
\prod_{(j,\ell) \in \xi_{\bm{i}}(J_{\rm uf})} D_{w\varpi_j, \varpi_j}^{-\varepsilon_{(i,k), (j,\ell)}}&=D_{w\varpi_i, \varpi_i}^{-\delta_{k,1}+\delta_{k,m_i-1}} \left(\prod_{\substack{j\in I;\\ \xi_{\bm{i}}^{-1}(i,k)<\xi_{\bm{i}}^{-1}(j,1)<\xi_{\bm{i}}^{-1}(i,k+1)<
\xi_{\bm{i}}^{-1}(j, m_j)}} D_{w\varpi_j, \varpi_j}^{-c_{j, i}}\right)\\
&\phantom{=}\cdot \left(\prod_{\substack{j\in I;\\ \xi_{\bm{i}}^{-1}(j,1)<\xi_{\bm{i}}^{-1}(i,k)<\xi_{\bm{i}}^{-1}(j,m_j)<
\xi_{\bm{i}}^{-1}(i, k+1)}} D_{w\varpi_j, \varpi_j}^{c_{j, i}}\right).
\end{align*}
Hence, for $(i,k)\in \xi_{\bm{i}}(J_{\rm uf})$, we deduce that 
\begin{align*}
&\widehat{X}_{\xi^{\vee}(i, k); \bm{i}}^{\rm mod}\\
    &=\left(\prod_{(j,\ell) \in \xi_{\bm{i}}(J_{\rm uf})} D((j,\ell),\bm{i})^{\varepsilon_{(i,k), (j,\ell)}}\right)\left(\prod_{\substack{j\in I;\\ \xi_{\bm{i}}^{-1}(i,k)<\xi_{\bm{i}}^{-1}(j,1)<\xi_{\bm{i}}^{-1}(j, m_j)<\xi_{\bm{i}}^{-1}(i,k+1)}} D_{w\varpi_j, \varpi_j}^{c_{j, i}}\right)\\
    &\phantom{=}\cdot \left(\prod_{\substack{j\in I;\\ \xi_{\bm{i}}^{-1}(j,1)<\xi_{\bm{i}}^{-1}(i,k)<\xi_{\bm{i}}^{-1}(j,m_j)<
\xi_{\bm{i}}^{-1}(i, k+1)}} D_{w\varpi_j, \varpi_j}^{c_{j, i}}\right)D_{w\varpi_i, \varpi_i}^{\delta_{k,m_i-1}}\\
&=\left(\prod_{(j,\ell) \in \xi_{\bm{i}}(J_{\rm uf})} D((j,\ell),\bm{i})^{\varepsilon_{(i,k), (j,\ell)}}\right)\left(\prod_{\substack{j\in I;\\ \xi_{\bm{i}}^{-1}(i,k)<\xi_{\bm{i}}^{-1}(j, m_j)<\xi_{\bm{i}}^{-1}(i,k+1)}} D_{w\varpi_j, \varpi_j}^{c_{j, i}}\right)D_{w\varpi_i, \varpi_i}^{\delta_{k,m_i-1}}\\
&=\widehat{X}_{(i, k); \bm{i}}.
\end{align*}
Since $s^{\vee}=(\xi_{\bm{i}}^{-1}\circ \xi^{\vee}\circ \xi_{\bm{i}})(s)$ for $s\in J_{\rm uf}$, we obtain the desired result. 
\end{proof}

\begin{cor}\label{c:modified_Xvariable}
Let $(s_1,\dots, s_k)$ be a sequence of elements of $J_{\rm uf}$. Then the following equality holds:
\begin{align*}
    \mu_{s_k^{\vee}}\cdots \mu_{s_1^{\vee}}(\widehat{X}_{s^{\vee}; \bm{i}}^{\rm mod})=\mu_{s_k}\cdots \mu_{s_1}(\widehat{X}_{s; \bm{i}}).
\end{align*}
\end{cor}

\begin{proof}
In the proof of \cref{t:Xhat_modified}, we have proved that 
\[
\varepsilon'_{s^{\vee}, t^{\vee}}=\varepsilon_{s, t}
\]
for all $s, t\in J_{\rm uf}$. Therefore, we obtain the desired assertion from \cref{t:Xhat_modified} by computing $\mu_{s_k^{\vee}}\cdots \mu_{s_1^{\vee}}(\widehat{X}_{s^{\vee}; \bm{i}}^{\rm mod})$ and $\mu_{s_k}\cdots \mu_{s_1}(\widehat{X}_{s; \bm{i}})$ inductively, using the relation \eqref{eq:X-mutation}. 
\end{proof}
\begin{cor}\label{c:modified_Avariable}
Let $\bm{\sigma}=(s_1,\dots, s_k)$ be a sequence of elements of  $J_{\rm uf}$, and write 
\[
\mu_{s_k}\cdots \mu_{s_1}(\mathbf{D}_{\bm{i}})=(D_{s; \bm{\sigma}})_{s\in J},\quad
\mu_{s_k^{\vee}}\cdots \mu_{s_1^{\vee}}(\mathbf{D}_{\bm{i}}^{\rm mod})=(D_{s; \bm{\sigma}}^{\rm mod})_{s\in J},
\]
where $\mathbf{D}_{\bm{i}}$ and $\mathbf{D}_{\bm{i}}^{\rm mod}$ are the clusters of $\mathbf{s}_{\bm{i}}$ and $\mathbf{s}_{\bm{i}}^{\rm mod}$, respectively. Then 
$D_{s^{\vee}; \bm{\sigma}}^{\rm mod}/D_{s; \bm{\sigma}}$ is equal to a Laurent monomial in $D_{w\varpi_i, \varpi_i}, i\in I$, for all $s\in J_{\rm uf}$. 
\end{cor}

\begin{rem}\label{r:modified_frozenvariable}
Note that we have
\[
D_{s; \bm{\sigma}}=D_{w\varpi_{i_s}, \varpi_{i_s}}\quad \text{and} \quad 
D_{s; \bm{\sigma}}^{\rm mod}=\frac{1}{D_{w\varpi_{i_s}, \varpi_{i_s}}}
\]
for all $s\in J_{\rm fr}$.
\end{rem}

\begin{proof}[{Proof of \cref{c:modified_Avariable}}]
We prove our statement by induction on $k$. When $k=0$, it follows from \cref{p:modified_seed}. Next suppose that the proposition holds for $\bm{\sigma}'=(s_1,\dots, s_{k-1})$. By \eqref{eq_mutation_for_cluster_variable}, we have 
\begin{align*}
D_{s; \bm{\sigma}}&=\begin{cases}
\displaystyle\frac{\prod_{t\in J}D_{t; \bm{\sigma}'}^{[-\varepsilon_{s_k, t}^{(k-1)}]_+}}{D_{s_k; \bm{\sigma}'}}(1+\widehat{X}_{s_k; \bm{\sigma}'})&\text{if }s= s_k,\\
D_{s; \bm{\sigma}'}&\text{if }s\neq s_k,
\end{cases} \\
D_{s; \bm{\sigma}}^{\rm mod}&=\begin{cases}
\displaystyle\frac{\prod_{t\in J}(D_{t; \bm{\sigma}'}^{\rm mod})^{[-\varepsilon_{s_k^{\vee}, t}^{(k-1), {\rm mod}}]_+}}{D_{s_k^{\vee}; \bm{\sigma}'}^{\rm mod}}(1+\widehat{X}_{s_k^{\vee}; \bm{\sigma}'}^{\rm mod})&\text{if }s= s_k^{\vee},\\
D_{s; \bm{\sigma}'}^{\rm mod}&\text{if }s\neq s_k^{\vee},
\end{cases}
\end{align*}
where we set $\mu_{s_{k-1}}\cdots \mu_{s_1}(\varepsilon^{\bm{i}})=(\varepsilon_{s, t}^{(k-1)})_{s\in J_{\rm uf}, t\in J}$, $\mu_{s_{k-1}^{\vee}}\cdots \mu_{s_1^{\vee}}(\varepsilon^{\bm{i},{\rm mod}})=(\varepsilon_{s, t}^{(k-1), {\rm mod}})_{s\in J_{\rm uf}, t\in J}$, $\mu_{s_{k-1}}\cdots \mu_{s_1}(\widehat{X}_{s_k; \bm{i}})=\widehat{X}_{s_k; \bm{\sigma}'}$, and $\mu_{s_{k-1}^{\vee}}\cdots \mu_{s_1^{\vee}}(\widehat{X}_{s_k^{\vee}; \bm{i}}^{\rm mod})=\widehat{X}_{s_k^{\vee}; \bm{\sigma}'}^{\rm mod}$. 

Hence, by \cref{c:modified_Xvariable}, \cref{r:modified_frozenvariable}, the equalities $\varepsilon'_{s^{\vee}, t^{\vee}}=\varepsilon_{s, t}$ for $s, t\in J_{\rm uf}$, and our induction hypothesis, we deduce the corollary. 
\end{proof}

By \eqref{eq:twistmutation} and the definition of $\mathbf{s}_{\bm{i}}^{\rm mod}$, we have 
\[
\overrightarrow{\boldsymbol{\mu}_{\bm{i}}}(\mathbf{s}_{\bm{i}}^{\rm mod})=(\eta_w^{\ast})^{-1}(\mathbf{s}_{\bm{i}}),
\]
where $\overrightarrow{\boldsymbol{\mu}_{\bm{i}}} \coloneqq (\overleftarrow{\boldsymbol{\mu}_{\bm{i}}})^{-1}$ (the mutation sequence obtained from $\overleftarrow{\boldsymbol{\mu}_{\bm{i}}}$ by reversing the order of the composition). Set 
\[
\widehat{X}_{s; \bm{i}}^{\rm mod, mut}\coloneqq \overrightarrow{\boldsymbol{\mu}_{\bm{i}}}(\widehat{X}_{s; \bm{i}}^{\rm mod})
\]
for $s\in J_{\rm uf}$. 

\begin{thm}\label{t:XT-rel}
Recall the isomorphism $y_{\bm{i}}^{\ast}\colon \mathbb{C}(U^-_w)\xrightarrow{\sim} \mathbb{C}(t_1,\dots, t_m)$ explained in the beginning of Section \ref{s:rel_with_stringNZ}. Then it holds that
\[
y_{\bm{i}}^{\ast}(\widehat{X}_{s; \bm{i}}^{\rm mod, mut})=t_st_{s^{+}}^{-1}
\]
for all $s\in J_{\rm uf}$. 
\end{thm}

\begin{proof}
By \cref{t:Chamber_Ansatz,t:upperBruhat}, we have 
\begin{align*}
    &t_st_{s^{+}}^{-1}\\
    &=(y_{\bm{i}}^{\ast}\circ (\eta_w^{\ast})^{-1})\left(\frac{\prod_{t<s<t^+}D(t, \bm{i})^{-c_{i_t, i_s}}}{D(s^-, \bm{i})D(s, \bm{i})}\cdot 
    \frac{D(s, \bm{i})D(s^+, \bm{i})}{\prod_{t<s^+<t^+}D(t, \bm{i})^{-c_{i_t, i_s}}}\right)\\
    &=(y_{\bm{i}}^{\ast}\circ (\eta_w^{\ast})^{-1})\\
    &\left(\frac{\prod\limits_{t<s<s^+<t^+}D(t, \bm{i})^{-c_{i_t, i_s}}\prod\limits_{t<s<t^+<s^+}D(t, \bm{i})^{-c_{i_t, i_s}}}{D(s^-, \bm{i})D(s, \bm{i})}\cdot
    \frac{D(s, \bm{i})D(s^+, \bm{i})}{\prod\limits_{s<t<s^+<t^+}D(t, \bm{i})^{-c_{i_t, i_s}} \prod\limits_{t<s<s^+<t^+}D(t, \bm{i})^{-c_{i_t, i_s}}}\right)\\
       &=(y_{\bm{i}}^{\ast}\circ (\eta_w^{\ast})^{-1})\left(\frac{\prod_{t<s<t^+<s^+}D(t, \bm{i})^{-c_{i_t, i_s}}}{D(s^-, \bm{i})}\cdot
    \frac{D(s^+, \bm{i})}{\prod_{s<t<s^+<t^+}D(t, \bm{i})^{-c_{i_t, i_s}}}\right)=y_{\bm{i}}^{\ast}(\widehat{X}_{s; \bm{i}}^{\rm mod, mut}). 
\end{align*}
\end{proof}

Let $\overrightarrow{\boldsymbol{\mu}_{\bm{i}}}^{\vee}$ be a mutation sequence obtained from $\overrightarrow{\boldsymbol{\mu}_{\bm{i}}}$ by replacing each $\mu_s$ with $\mu_{s^{\vee}}$. Set 
\[
\mathbf{s}_{\bm{i}}^{\rm mut}\coloneqq\overrightarrow{\boldsymbol{\mu}_{\bm{i}}}^{\vee}(\mathbf{s}_{\bm{i}}),
\]
and write 
\begin{align*}
&\mathbf{D}_{\bm{i}}^{\rm mut}=(D^{\rm mut}(s, \bm{i}))_{s\in J}\coloneqq \overrightarrow{\boldsymbol{\mu}_{\bm{i}}}^{\vee}(\mathbf{D}_{\bm{i}}),\\ 
&\varepsilon^{\bm{i}, {\rm mut}} = (\varepsilon_{s, t}^{\rm mut})_{s\in J_{\rm uf}, t\in J}\coloneqq \overrightarrow{\boldsymbol{\mu}_{\bm{i}}}^{\vee}(\varepsilon^{\bm{i}}),\\
&\widehat{X}_{s; \bm{i}}^{\rm mut}\coloneqq \prod_{t \in J} D^{\rm mut}(t, \bm{i})^{\varepsilon_{s, t}^{\rm mut}}\text{ for }s\in J_{\rm uf}.  
\end{align*}
Then, by \cref{c:modified_Xvariable}, we have 
\[
\widehat{X}_{s; \bm{i}}^{\rm mut}=\widehat{X}_{s^{\vee}; \bm{i}}^{\rm mod, mut}
\]
for $s\in J_{\rm uf}$. 

\begin{ex}\label{e:smut}
	For the examples in \cref{e:quiverexample}, we obtain the following.
	\begin{itemize}
	
\item The case of $G = SL_4(\mathbb{C})$ and $\bm{i} = (2, 1, 2, 3, 2, 1)\in R(w_0)$: we have $\overrightarrow{\boldsymbol{\mu}_{\bm{i}}}^{\vee} = \mu_3 \mu_1 \mu_2 \mu_3$, and the quiver associated with $\varepsilon^{\bm{i}, {\rm mut}}$ is given by 
	
	\hfill
	\begin{xy} 0;<1pt,0pt>:<0pt,-1pt>::
		(-18,0) *+{3},
		(-10,0) *+{\scalebox{0.65}{$>$}},	
		(-18,30) *+{2},
		(-10,30) *+{\scalebox{0.65}{$>$}},		
		(-18,60) *+{1},
		(-10,60) *+{\scalebox{0.65}{$>$}},			
		(0,30) *+{1} ="1",
		(40,60) *+{2} ="2",
		(80,30) *+{3} ="3",
		(120,0) *+{4} ="4",		
		(160,30) *+{5} ="5",
		(200,60) *+{6} ="6",
		"1", {\ar"3"},
		"4", {\ar"1"},
		"3", {\ar"4"},
		"3", {\ar"2"},
		"5", {\ar"3"},
		"2", {\ar"6"},
	\end{xy}
	\hfill
	\hfill
	
\item The case of $G = Sp_4(\mathbb{C})$ and $\bm{i}=(1, 2, 1, 2)\in R(w_0)$: we have $\overrightarrow{\boldsymbol{\mu}_{\bm{i}}}^{\vee} = \mu_1 \mu_2$, and the quiver associated with $\varepsilon^{\bm{i}, {\rm mut}}$ is given by 
	
	\hfill
	\begin{xy} 0;<1pt,0pt>:<0pt,-1pt>::
		(-18,0) *+{2},
		(-10,0) *+{\scalebox{0.65}{$>$}},	
		(-18,30) *+{1},
		(-10,30) *+{\scalebox{0.65}{$>$}},		
		(0,30) *+{1} ="1",
		(40,0) *+{2} ="2",
		(80,30) *+{3} ="3",
		(120,0) *+{4} ="4",
		"1", {\ar"2"},
		"2", {\ar"4"},
		"3", {\ar"2"},
	\end{xy}
	\hfill
	\hfill

	\end{itemize}	
\end{ex}

The following is the main result of this subsection.

\begin{thm}\label{t:relNZ}
Let $w\in W$, and $\bm{i}\in R(w)$. 
Then there exist a refinement of the partial order $\preceq_{\varepsilon^{\bm{i}, {\rm mut}}}^{\rm op}$ (\cref{d:order,d:main_valuation}) and a unimodular $J\times J$-matrix $N_{\bm{i}}$ such that the corresponding valuation $v_{\mathbf{s}_{\bm{i}}^{\rm mut}}$ on $\mathbb{C}(X(w))$ associated with $\mathbf{s}_{\bm{i}}^{{\rm mut}}$ (\cref{d:main_valuation}) satisfies 
	\[
	v_{\bm i}^{\rm low}(f)=v_{\mathbf{s}_{\bm{i}}^{\rm mut}}(f)N_{\bm{i}} 
	\] 	
	in $\mathbb{Z}^J$ for all $f\in \mathbb{C}(X(w)) \setminus \{0\}$. Here we consider elements of $\mathbb{Z}^J$ as $1\times J$-matrices.
\end{thm}

\begin{proof}
By \cref{c:modified_Avariable} and \cref{r:modified_frozenvariable}, we have 
\[
D^{\rm mut}(s, \bm{i})=\begin{cases}
(\eta_w^{\ast})^{-1}(D(s^{\vee}, \bm{i}))\prod_{i\in I}D_{w\varpi_i, \varpi_i}^{f_i^{(s)}}&\text{ if } s\in J_{\rm uf},\\
D_{w\varpi_{i_s}, \varpi_{i_s}}&\text{ if } s\in J_{\rm fr}
\end{cases}
\]
for some $f_i^{(s)}\in \mathbb{Z}$ ($i\in I$, $s\in J$). Hence, by \cref{t:Chamber_Ansatz}, if we set 
\[
n_k^{(s)}\coloneqq \begin{cases}
-d_k^{(s^{\vee})}+\sum_{i\in I}f_i^{(s)}\langle h_{i_k}, s_{i_{k+1}}\cdots s_{i_{m}}\varpi_i\rangle&\text{ if }s\in J_{\rm uf},\\
d_k^{(s)}&\text{ if }s\in J_{\rm fr}
\end{cases}
\]
for $k \in J$, then we have  
\[
y_{\bm{i}}^{\ast}(D^{\rm mut}(s, \bm{i}))=t_1^{n_1^{(s)}}\cdots t_m^{n_m^{(s)}}
\]
for all $s\in J$. Moreover, by \cref{t:Chamber_Ansatz} again, there exist $\tilde{n}_{t}^{(s)}\in\mathbb{Z}$, $s,  t\in J$, such that 
\[
t_s=y_{\bm{i}}^{\ast}(D^{\rm mut}(1, \bm{i}))^{\tilde{n}_{1}^{(s)}}\cdots y_{\bm{i}}^{\ast}(D^{\rm mut}(m, \bm{i}))^{\tilde{n}_{m}^{(s)}}
\]
for all $s\in J$. Thus, if we set $N_{\bm{i}}\coloneqq (n^{(s)}_{t})_{s, t\in J}$ (we consider $s$ as a row index and $t$ as a column index), then $N_{\bm{i}}$ is a unimodular $J\times J$-matrix. Indeed, for $\widetilde{N}_{\bm{i}}\coloneqq (\tilde{n}^{(s)}_{t})_{s, t\in J}$ (we consider $s$ as a row index and $t$ as a column index), we have  $N_{\bm{i}}^{-1}=\widetilde{N}_{\bm{i}}$. We consider the \emph{total} order $<_{\widetilde{N}_{\bm{i}}}$ on $\mathbb{Z}^{J}$ defined by
	\[
	\bm{a} <_{\widetilde{N}_{\bm{i}}} \bm{a}'\ \text{if and only if }
	\bm{a}'-\bm{a}=\bm{v}\widetilde{N}_{\bm{i}} \ \text{for some }\bm{v}\in \mathbb{Z}^{J}=\mathbb{Z}^m\text{ such that }\bm{v}>(0,\dots, 0),
	\]
	where we recall the lexicographic order $<$ from \cref{d:lowest_term_valuation_Schubert}. Then we claim that the total order $<_{\widetilde{N}_{\bm{i}}}$ on $\mathbb{Z}^m$ refines the partial order $\preceq_{\varepsilon^{\bm{i}, {\rm mut}}}^{\rm op}$. Indeed, it suffices to show that  
	\[
	(0,\dots, 0)<_{\widetilde{N}_{\bm{i}}}(\varepsilon_{s, 1}^{\rm mut},\dots, \varepsilon_{s, m}^{\rm mut})
	\]
	for all $s\in J_{\rm uf}$. This is equivalent to the condition that
	\[
	(0,\dots, 0)<(\varepsilon_{s, 1}^{\rm mut},\dots, \varepsilon_{s, m}^{\rm mut})N_{\bm{i}}=v_{\bm i}^{\rm low}(\widehat{X}_{s; \bm{i}}^{\rm mut})
	\]
    for all $s\in J_{\rm uf}$. By \cref{t:XT-rel}, we have 
    \[
    v_{\bm i}^{\rm low}(\widehat{X}_{s; \bm{i}}^{\rm mut})=(0, \ldots,0, \overset{\overset{s^{\vee}}{\vee}}{1},0, \ldots, 0,\overset{\overset{(s^{\vee})^+}{\vee}}{-1},0, \ldots, 0),
    \]
    which proves our claim. Thus, by using $<_{\widetilde{N}_{\bm{i}}}$, we define the valuation $v_{\mathbf{s}_{\bm{i}}^{\rm mut}}$ on $\mathbb{C}(X(w))$. Then, by definition, 
    \[
    v_{\mathbf{s}_{\bm{i}}^{{\rm mut}}}(D^{\rm mut}(1, \bm{i})^{a_1}\cdots D^{\rm mut}(m, \bm{i})^{a_m})\leq_{\widetilde{N}_{\bm{i}}}v_{\mathbf{s}_{\bm{i}}^{{\rm mut}}}(D^{\rm mut}(1, \bm{i})^{a_1 ^\prime}\cdots D^{\rm mut}(m, \bm{i})^{a_m ^\prime})
    \]
    if and only if 
    \[
    v_{\bm i}^{\rm low}(D^{\rm mut}(1, \bm{i})^{a_1}\cdots D^{\rm mut}(m, \bm{i})^{a_m})\leq v_{\bm i}^{\rm low}(D^{\rm mut}(1, \bm{i})^{a_1 ^\prime}\cdots D^{\rm mut}(m, \bm{i})^{a_m ^\prime}). 
    \]
	Moreover, it follows that
\[
v_{\bm i}^{\rm low}(D^{\rm mut}(1, \bm{i})^{a_1}\cdots D^{\rm mut}(m, \bm{i})^{a_m})=v_{\mathbf{s}_{\bm{i}}^{\rm mut}}(D^{\rm mut}(1, \bm{i})^{a_1}\cdots D^{\rm mut}(m, \bm{i})^{a_m})N_{\bm{i}} 
\]
for all $(a_1, \ldots, a_m) \in \z^m$. Hence we have 
\[
	v_{\bm i}^{\rm low}(f)=v_{\mathbf{s}_{\bm{i}}^{\rm mut}}(f)N_{\bm{i}}
	\] 	
	for all $f\in \mathbb{C}(X(w)) \setminus \{0\}$. 
\end{proof}

Combining Theorems \ref{t:NO_body_crystal_basis} and \ref{t:relNZ} with the unimodularity of $N_{\bm i}$, we obtain the following.

\begin{cor}\label{c:relNZ}
Let $w \in W$, $\lambda \in P_+$, and ${\bm i} \in R(w)$. 
Define $v_{\mathbf{s}_{\bm{i}}^{\rm mut}}$ and $N_{\bm{i}}$ as in \cref{t:relNZ}. 
Then the equalities 
\begin{align*}
&S(X(w), \mathcal{L}_\lambda, v_{\bm i} ^{\rm low}, \tau_\lambda) = \{(k, {\bm a} N_{\bm i}) \mid (k, {\bm a}) \in S(X(w), \mathcal{L}_\lambda, v_{\mathbf{s}_{\bm{i}}^{\rm mut}}, \tau_\lambda)\},\\
&C(X(w), \mathcal{L}_\lambda, v_{\bm i} ^{\rm low}, \tau_\lambda) = \{(k, {\bm a} N_{\bm i}) \mid (k, {\bm a}) \in C(X(w), \mathcal{L}_\lambda, v_{\mathbf{s}_{\bm{i}}^{\rm mut}}, \tau_\lambda)\},\ {\it and}\\
&\Delta(X(w), \mathcal{L}_\lambda, v_{\bm i} ^{\rm low}, \tau_\lambda) = \Delta(X(w), \mathcal{L}_\lambda, v_{\mathbf{s}_{\bm{i}}^{\rm mut}}, \tau_\lambda) N_{\bm i}
\end{align*}
hold. 
In particular, the Newton--Okounkov body $\Delta(X(w), \mathcal{L}_\lambda, v_{\mathbf{s}_{\bm{i}}^{\rm mut}}, \tau_\lambda)$ is unimodularly equivalent to the Nakashima--Zelevinsky polytope $\widetilde{\Delta}_{\bm i} (\lambda)$. 
\end{cor}

Moreover, by \cref{c:relation_with_string_cones} and \cref{t:relNZ}, we obtain the following.

\begin{cor}
Define $v_{\mathbf{s}_{\bm{i}}^{\rm mut}}$ and $N_{\bm{i}}$ as in \cref{t:relNZ}. 
Then the equality
\[C_{\bm i} = C_{{\bf s}_{\bm i}^{\rm mut}} N_{\bm i}\]
holds. 
In particular, the cluster cone $C_{{\bf s}_{\bm i}^{\rm mut}}$ is unimodularly equivalent to the string cone associated with ${\bm i}^{\rm op}$.
\end{cor} 

By combining Corollaries \ref{c:relation_of_NO_by_tropicalized_mutations}, \ref{c:relstring}, and \ref{c:relNZ}, we obtain the following corollaries.

\begin{cor}\label{c:relation_string_NZ}
For $w \in W$ and $\lambda \in P_+$, the string polytopes $\Delta_{\bm i} (\lambda)$ and the Nakashima--Zelevinsky polytopes $\widetilde{\Delta}_{\bm i} (\lambda)$ associated with ${\bm i} \in R(w)$ are all related by tropicalized cluster mutations up to unimodular transformations.
\end{cor}

\begin{cor}\label{c:relation_string_NZ_twist_auto}
The tropicalization $(\overrightarrow{\boldsymbol{\mu}_{\bm{i}}}^{\vee})^T$ of the mutation sequence $\overrightarrow{\boldsymbol{\mu}_{\bm{i}}}^{\vee}$ gives a bijective piecewise-linear map from the string polytope $\Delta(X(w), \mathcal{L}_\lambda, v_{{\bf s}_{\bm i}}, \tau_\lambda)$ onto the Nakashima--Zelevinsky polytope $\Delta(X(w), \mathcal{L}_\lambda, v_{{\bf s}_{\bm i} ^{\rm mut}}, \tau_\lambda)$.
\end{cor}

\begin{ex}
Let $G = SL_3(\c)$, $\lambda \in P_+$, and consider the upper cluster algebra structure on $\c[U_{w_0} ^-]$. Since $X(w_0) = G/B$, we obtain the Newton--Okounkov body $\Delta(G/B, \mathcal{L}_\lambda, v_{{\bf s}}, \tau_\lambda)$ for each FZ-seed ${\bf s}$ for $\c[U_{w_0} ^-]$. In this case, there are only two FZ-seeds: ${\bf s}_{\bm i}$ and ${\bf s}_{{\bm i}^\prime}$, where ${\bm i}, {\bm i}^\prime \in R(w_0)$ are defined by ${\bm i} \coloneqq (1, 2, 1)$ and ${\bm i}^\prime \coloneqq (2, 1, 2)$. The quivers corresponding to $\varepsilon^{\bm i}$ and $\varepsilon^{{\bm i}^\prime}$ are both given as follows:

	\hfill
	\scalebox{1.0}[1.0]{
\begin{xy} 0;<1pt,0pt>:<0pt,-1pt>::
			(120,10) *+{2} ="1",
			(60,40) *+{1} ="3",
			(180,40) *+{3.} ="4",
			"3", {\ar"1"},
			"4", {\ar"3"},
		\end{xy}
	}
	\hfill
	\hfill

\noindent Then we have ${\bf s}_{{\bm i}^\prime} = \mu_1 ({\bf s}_{\bm i})$. We deduce by \cref{t:NO_body_crystal_basis} (3) and \cite[Section 1]{Lit} that the Newton--Okounkov body $\Delta(G/B, \mathcal{L}_\lambda, \tilde{v}_{\bm i} ^{\rm low}, \tau_\lambda)$ (resp., $\Delta(G/B, \mathcal{L}_\lambda, \tilde{v}_{{\bm i}^\prime} ^{\rm low}, \tau_\lambda)$) coincides with 
\begin{align*}
&\{(a_1, a_2, a_3) \in \r_{\ge 0} ^3 \mid a_3 \le \lambda_1,\ a_3 \le a_2 \le a_3 + \lambda_2,\ a_1 \le a_2 -2a_3 +\lambda_1\}\\
({\rm resp}.,\ &\{(a_1, a_2, a_3) \in \r_{\ge 0} ^3 \mid a_3 \le \lambda_2,\ a_3 \le a_2 \le a_3 + \lambda_1,\ a_1 \le a_2 -2a_3 +\lambda_2\}),
\end{align*}
where $\lambda_i \coloneqq \langle \lambda, h_i\rangle$ for $i = 1, 2$. Since we have 
\[
M_{\bm i} = M_{{\bm i}^\prime} = \begin{pmatrix}
1 & 0 & 0\\
1 & 1 & 0 \\
0 & 1 & 1
\end{pmatrix},
\]
it follows by \cref{c:relstring} (1) that $\Delta(G/B, \mathcal{L}_\lambda, v_{{\bf s}_{\bm i}}, \tau_\lambda)$ (resp., $\Delta(G/B, \mathcal{L}_\lambda, v_{{\bf s}_{{\bm i}^\prime}}, \tau_\lambda)$) coincides with the following polytope:
\begin{align*}
&\{(g_1, g_2, g_3) \in \r^3 \mid 0 \le g_3 \le \lambda_1,\ 0 \le g_2 \le \lambda_2,\ -g_2 \le g_1 \le -g_3 +\lambda_1\}\\
({\rm resp}.,\ &\{(g_1, g_2, g_3) \in \r^3 \mid 0 \le g_3 \le \lambda_2,\ 0 \le g_2 \le \lambda_1,\ -g_2 \le g_1 \le -g_3 +\lambda_2\}).
\end{align*}
We define an $\r$-linear automorphism $\omega \colon \r^3 \xrightarrow{\sim} \r^3$ by $\omega (a_1, a_2, a_3) \coloneqq (a_1, a_3, a_2)$. Then the composite map $\omega \circ \mu_1 ^T \colon \r^3 \rightarrow \r^3$ is given by 
\[
\omega \circ \mu_1 ^T (g_1, g_2, g_3) \coloneqq (-g_1, g_3 +[g_1]_+, g_2 -[-g_1]_+),
\]
which gives a bijective piecewise-linear map from $\Delta(G/B, \mathcal{L}_\lambda, v_{{\bf s}_{\bm i}}, \tau_\lambda)$ onto $\Delta(G/B, \mathcal{L}_\lambda, v_{{\bf s}_{{\bm i}^\prime}}, \tau_\lambda)$. In addition, we see that $\overrightarrow{\boldsymbol{\mu}_{\bm{i}}}^{\vee} = \overrightarrow{\boldsymbol{\mu}_{\bm{i}^\prime}}^{\vee} = \mu_1$, and hence that 
\[
{\bf s}_{\bm i} ^{\rm mut} = {\bf s}_{{\bm i}^\prime},\quad {\bf s}_{{\bm i}^\prime} ^{\rm mut} = {\bf s}_{\bm i}.
\]
If $\lambda = \varpi_1 + \varpi_2 \in P_+$, then we have $\lambda_1 = \lambda_2 = 1$. Hence the Newton--Okounkov bodies $\Delta(G/B, \mathcal{L}_\lambda, v_{{\bf s}_{\bm i}}, \tau_\lambda)$ and $\Delta(G/B, \mathcal{L}_\lambda, v_{{\bf s}_{{\bm i}^\prime}}, \tau_\lambda)$ both coincide with the following polytope:
\begin{align*}
&\{(g_1, g_2, g_3) \in \r^3 \mid 0 \le g_3 \le 1,\ 0 \le g_2 \le 1,\ -g_2 \le g_1 \le -g_3 +1\};
\end{align*}
see Figure \ref{figure_ex_NOBY_cluster_1}. In this case, $\omega \circ \mu_1 ^T$ gives a bijective piecewise-linear map from this polytope to itself.

\begin{figure}[!ht]
\begin{center}
   \includegraphics[width=5.0cm,bb=80mm 190mm 130mm 230mm,clip]{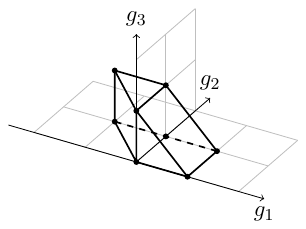}
	\caption{The Newton--Okounkov body $\Delta(G/B, \mathcal{L}_\lambda, v_{{\bf s}_{\bm i}}, \tau_\lambda)$ for $\lambda = \varpi_1 + \varpi_2$.}
	\label{figure_ex_NOBY_cluster_1}
\end{center}
\end{figure}
\end{ex}

\begin{ex}\label{ex:NO_body_string_C2}
Let $G = Sp_4(\c)$, $\lambda \in P_+$, and ${\bm i} = (1, 2, 1, 2) \in R(w_0)$. 
Then the quiver $\Gamma_{\bm{i}}$ associated with $\varepsilon^{\bm i}$ is given in \cref{e:quiverexample}. 
By \cref{t:NO_body_crystal_basis} (3) and \cite[Section 1]{Lit}, the Newton--Okounkov body $\Delta(G/B, \mathcal{L}_\lambda, \tilde{v}_{\bm i} ^{\rm low}, \tau_\lambda)$ coincides with the set of $(a_1, \ldots, a_4) \in \r_{\ge 0} ^4$ satisfying the following inequalities:
\begin{align*}
&a_4 \le \lambda_2,\ a_4 \le a_3 \le a_4 + \lambda_1,\ a_3 \le a_2 \le 2a_3 -2a_4 +\lambda_2,\ a_1 \le a_2 -2a_3 +a_4 +\lambda_1,
\end{align*}
where $\lambda_i \coloneqq \langle \lambda, h_i\rangle$ for $i = 1, 2$. Since we have 
\[
M_{\bm i} = \begin{pmatrix}
1 & 0 & 0 & 0 \\
1 & 1 & 0 & 0 \\
1 & 2 & 1 & 0 \\
0 & 1 & 1 & 1 
\end{pmatrix},
\]
it follows by \cref{c:relstring} (1) that $\Delta(G/B, \mathcal{L}_\lambda, v_{{\bf s}_{\bm i}}, \tau_\lambda)$ coincides with the set of $(g_1, \ldots, g_4) \in \r^4$ satisfying the following inequalities:
\begin{align*}
&0 \le g_4 \le \lambda_2,\ 0 \le g_3 \le \lambda_1,\ -g_3 \le g_2 \le -g_4 +\lambda_2,\ -g_2 -g_3 \le g_1 \le -g_3 +\lambda_1.
\end{align*}
In addition, by \cref{t:NO_body_crystal_basis} (2) and \cite[Theorem 6.1]{Nak1} (see also \cite[Example 3.12]{Fuj}), we see that $\Delta(G/B, \mathcal{L}_\lambda, v_{\bm i} ^{\rm low}, \tau_\lambda)$ coincides with the set of $(a_1, \ldots, a_4) \in \r_{\ge 0} ^4$ satisfying the following inequalities:
\begin{align*}
&a_4 \le \lambda_2,\ a_3 \le a_4 + \lambda_1,\ a_2 \le \min\{2a_3, a_3 +\lambda_1\},\ 2a_1 \le \min\{a_2, 2\lambda_1\}.
\end{align*}
By computing $\overrightarrow{\boldsymbol{\mu}_{\bm{i}}}^{\vee} ({\bf D}_{\bm i}) = \mu_1 \mu_2 ({\bf D}_{\bm i})$, we deduce that
\[
N_{\bm i} = \begin{pmatrix}
0 & 2 & 1 & 0 \\
0 & 1 & 1 & 0 \\
1 & 2 & 1 & 0 \\
0 & 1 & 1 & 1
\end{pmatrix}.
\]
This implies by \cref{c:relNZ} that $\Delta(G/B, \mathcal{L}_\lambda, v_{{\bf s}_{\bm i} ^{\rm mut}}, \tau_\lambda)$ coincides with the set of $(g_1, \ldots, g_4) \in \r^4$ satisfying the following inequalities:
\begin{align*}
&0 \le g_4 \le \lambda_2,\ 0 \le g_3 \le \lambda_1,\ -g_4 \le g_2 \le -g_1 -g_3 +\lambda_1,\ -g_2 -g_4 \le 2g_1 \le -2g_3 +2\lambda_1.
\end{align*}
By \cref{c:relation_string_NZ_twist_auto}, the tropicalization $(\overrightarrow{\boldsymbol{\mu}_{\bm{i}}}^{\vee})^T$ gives a bijective piecewise-linear map from $\Delta(G/B, \mathcal{L}_\lambda, v_{{\bf s}_{\bm i}}, \tau_\lambda)$ onto $\Delta(G/B, \mathcal{L}_\lambda, v_{{\bf s}_{\bm i} ^{\rm mut}}, \tau_\lambda)$. The tropicalization $(\overrightarrow{\boldsymbol{\mu}_{\bm{i}}}^{\vee})^T$ is given by 
\begin{align*}
(\overrightarrow{\boldsymbol{\mu}_{\bm{i}}}^{\vee})^T (g)& = (-g_1 -[-g_2]_+, -g_2 +2[g_1 +[g_2]_+]_+, g_3 -[-g_2]_+ -[-g_1 -[g_2]_+]_+, g_4 +[g_2]_+)
\end{align*}
for $g = (g_1, \ldots, g_4) \in \r^4$.
\end{ex}

\section{Relation with GHKK superpotential polytopes}\label{s:superpotential_polytopes}

In this section, we relate our Newton--Okounkov bodies arising from $\mathcal{A}$-cluster structures with superpotential polytopes discussed in \cite{Mag, GHKK}.
Let $U$ be the unipotent radical of $B$. 
Then the quotient space $G/U$ is called a \emph{base affine space}. 
It is not an affine variety, but a quasi-affine variety. 
We define a left $G \times H$-action on $G/U$ by 
\[(g, h) \cdot (g^\prime \bmod U) \coloneqq g g^\prime h^{-1} \bmod U\] 
for $g, g^\prime \in G$ and $h \in H$, which is well-defined since $U h^{-1} = h^{-1} U$. 
The natural right $G \times H$-action on the ring $\mathbb{C}[G/U]$ of regular functions on $G/U$ is transformed into the left action by the inverse map of $G \times H$.
Define a left $G \times G$-action on $\mathbb{C}[G]$ by $((g_1, g_2) \cdot \psi) (g) \coloneqq \psi (g_1^{-1} g g_2)$ for $\psi \in \mathbb{C}[G]$ and $g_1, g_2, g \in G$.
Then the $\c$-algebra $\mathbb{C}[G/U]$ is naturally isomorphic to the ring 
\[\mathbb{C}[G]^{\{e\} \times U} \coloneqq \{\psi \in \mathbb{C}[G] \mid (e, u) \cdot \psi = \psi\ \text{for\ all}\ u \in U\}\] 
of invariant functions, where $e \in G$ denotes the identity element.
The $\c$-algebra $\mathbb{C}[G]^{\{e\} \times U}$ is a $G \times H$-submodule of $\mathbb{C}[G]$, and the isomorphism $\mathbb{C}[G/U] \simeq \mathbb{C}[G]^{\{e\} \times U}$ is compatible with $G \times H$-actions.
By the algebraic Peter--Weyl theorem (see, for instance, \cite[Theorem 4.2.7]{GW}), we see that
\[\mathbb{C}[G] \simeq \bigoplus_{\lambda \in P_+} V(\lambda)^\ast \otimes V(\lambda)\]
as $G \times G$-modules, where $V(\lambda)^\ast \otimes V(\lambda) \hookrightarrow \mathbb{C}[G]$ is given as $f \otimes v \mapsto C_{f, v}$ for $f \in V(\lambda)^\ast$ and $v \in V(\lambda)$ (see Section \ref{ss:unipotent} for the definition of $C_{f, v}$). 
Hence it follows that 
\begin{equation}\label{eq:base_affine_isotypic}
\begin{aligned}
\mathbb{C}[G/U] \simeq \bigoplus_{\lambda \in P_+} V(\lambda)^\ast
\end{aligned}
\end{equation}
as $G$-modules, where we regard $\mathbb{C}[G/U]$ as a $G$-module by $G \simeq G \times \{e\}\ (\subseteq G \times H)$, and $\kappa_\lambda \colon V(\lambda)^\ast \hookrightarrow \mathbb{C}[G/U] \simeq \mathbb{C}[G]^{\{e\} \times U}$ is given by $f \mapsto C_{f, v_\lambda}$ for $f \in V(\lambda)^\ast$. 
For each $f \in V(\lambda)^\ast$, the matrix coefficient $C_{f, v_\lambda}$ is a weight vector with weight $\lambda$ under the action of $H \simeq \{e\} \times H\ (\subseteq G \times H)$. 
Hence \eqref{eq:base_affine_isotypic} is precisely the decomposition of the $\{e\} \times H$-module $\mathbb{C}[G/U]$ into isotypic components. 
Let us use the notation in Appendix \ref{a:doubleBruhat}, and regard the double Bruhat cell $G^{w_0, e}$ as an affine open subvariety of $G/U$ by the following open embedding: 
\[G^{w_0, e} \hookrightarrow G/U,\quad g \mapsto g \bmod U.\]
Then the $\c$-algebra $\mathbb{C}[G/U]$ is identified with a $\mathbb{C}$-subalgebra of the upper cluster algebra $\mathbb{C}[G^{w_0, e}]$.
In addition, we see that
\[G^{w_0, e} = \{g \bmod U \in G/U \mid \Delta_{\varpi_i, \varpi_i}(g), \Delta_{w_0 \varpi_i, \varpi_i}(g) \neq 0\ \text{for\ all}\ i \in I\},\]
and that 
\[\mathbb{C}[G/U][\Delta_{\varpi_i, \varpi_i}^{-1}, \Delta_{w_0 \varpi_i, \varpi_i}^{-1} \mid i \in I] \simeq \mathbb{C}[G^{w_0, e}].\]
Take a projection $\pi \colon \mathbb{C}[G^{w_0, e}] \rightarrow \c[U_{w_0}^-]$ given by \eqref{eq:Bruhatunip}. 
Since $\kappa_\lambda (V(\lambda)^\ast) = \{C_{f, v_\lambda} \mid f \in V(\lambda)^\ast\}$, we obtain the following by the proof of \cite[Lemma 4.5]{FN}. 

\begin{lem}\label{l:base_affine_and unipotent_cell_lambda}
For $\lambda \in P_+$, the composite map $\pi \circ \kappa_\lambda$ gives a $\c$-linear isomorphism from $V(\lambda)^\ast$ onto $\{(\sigma/\tau_\lambda)|_{U_{w_0}^-} \mid \sigma \in H^0(G/B, \mathcal{L}_\lambda)\}$. 
\end{lem}

In the rest of this section, we assume that $G = SL_{n+1}(\mathbb{C})$. 
Let $\widetilde{\mathcal{S}} = \{\tilde{\bf s}_t = (\widetilde{\mathbf{A}}_t, \widetilde{\varepsilon}_t)\}_{t \in \mathbb{T}}$ be the cluster pattern for the upper cluster algebra $\mathbb{C}[G^{w_0, e}]$. 
We write $\widetilde{\mathbf{A}}_t = (\widetilde{A}_{j; t})_{j \in \widetilde{J}}$ and $\widetilde{\varepsilon}_t = (\widetilde{\varepsilon}_{i, j}^{(t)})_{i \in J_{\rm uf}, j \in \widetilde{J}}$. 

\begin{thm}[{see \cite[Section 10.4]{GLS:partial}}]
Let $G = SL_{n+1}(\mathbb{C})$. 
Then the $\mathbb{C}$-algebra $\mathbb{C}[G/U]$ coincides with the $\mathbb{C}$-subalgebra of $\mathbb{C}[G^{w_0, e}]$ generated by $\{\widetilde{A}_{j; t} \mid t \in \mathbb{T},\ j \in \widetilde{J}\}$. 
\end{thm}

Let $\mathcal{A}$ be the $\mathcal{A}$-cluster variety corresponding to the upper cluster algebra $\mathbb{C}[G^{w_0, e}]$. 
Then Williams \cite[Theorem 4.16]{Wil} proved that there exists a morphism $\mathcal{A} \rightarrow G^{w_0, e}$ that induces the isomorphism in \cref{t:Double_cluster}. 
Let $W_{\rm GHKK}^T \colon \mathcal{A}^\vee(\mathbb{R}^T) \rightarrow \mathbb{R}$ be the tropicalization of the Gross--Hacking--Keel--Kontsevich superpotential $W_{\rm GHKK}$ (see \cite[Corollary 9.17]{GHKK}). 
We set
\[\Xi \coloneqq \{x \in \mathcal{A}^\vee(\mathbb{R}^T) \mid W_{\rm GHKK}^T(x) \geq 0\},\]
and write $\Xi_{\mathbb{Z}} \coloneqq \Xi \cap \mathcal{A}^\vee(\mathbb{Z}^T)$. 
Then, for each $t \in \mathbb{T}$, the set $\Xi_t\ (\subseteq \mathcal{A}^\vee_t(\mathbb{R}^T) = \r^{\widetilde{J}})$ is a rational convex polyhedral cone. 

\begin{thm}[{see \cite[Corollary 3]{Mag} and \cite[Corollary 0.20]{GHKK}}]\label{t:Mag_GHKK}
Let $G = SL_{n+1}(\mathbb{C})$. 
\begin{enumerate}
\item[{\rm (1)}] The $\mathbb{C}$-algebra $\mathbb{C}[G/U]$ has the theta function basis $\{\vartheta_q \mid q \in \Xi_{\mathbb{Z}}\}$. 
\item[{\rm (2)}] For each $q \in \Xi_\z$, the theta function $\vartheta_q$ is a weight vector with respect to the action of $H \simeq \{e\} \times H$ on $\mathbb{C}[G/U]$.
Let ${\rm wt}(\vartheta_q) \in P$ denote the corresponding weight. 
\item[{\rm (3)}] The map 
\[\Xi_{\mathbb{Z}} \rightarrow P,\quad q \mapsto {\rm wt}(\vartheta_q),\]
is naturally lifted to a function $\underline{\rm wt} \colon \mathcal{A}^\vee(\mathbb{R}^T) \rightarrow P \otimes_\z \mathbb{R}$ which is $\mathbb{R}$-linear with respect to the $\mathbb{R}$-linear structure $\mathcal{A}^\vee(\mathbb{R}^T) \simeq \mathcal{A}^{\vee} _t (\r^T) = \r^{\widetilde{J}}$ for every $t \in \mathbb{T}$. 
\item[{\rm (4)}] For each $\lambda \in P_+$, the set 
\[\{\vartheta_q \mid q \in \Xi_{\mathbb{Z}} \cap \underline{\rm wt}^{-1} (\lambda)\}\]
forms a $\mathbb{C}$-basis of $V(\lambda)^\ast$ in \eqref{eq:base_affine_isotypic}.
\end{enumerate}
\end{thm}

For each $\lambda \in P_+$, we set 
\[\Xi(\lambda) \coloneqq \Xi \cap \underline{\rm wt}^{-1} (\lambda).\]
Recall that $\widetilde{J} = \overline{I} \cup J$ for $\overline{I} = \{\overline{i} \mid i \in I\}$ and $J = \{1, \ldots, m\}$. 
Hence we have $\r^{\widetilde{J}} = \r^{\overline{I}} \oplus \r^J$. 
Let ${\rm pr}_{\r^J} \colon \r^{\widetilde{J}} \rightarrow \r^J$ be the canonical projection, and set 
\[\Xi(\tilde{\bf s}_t, \lambda) \coloneqq {\rm pr}_{\r^J} (\Xi(\lambda)_t)\]
for each $t \in \mathbb{T}$, where $\Xi(\lambda)_t \subseteq \mathcal{A}^{\vee} _t (\r^T) = \r^{\widetilde{J}}$. 
The set $\Xi(\tilde{\bf s}_t, \lambda) \subseteq \r^J$ is called a \emph{superpotential polytope}.
We set 
\[S(\tilde{\bf s}_t, \lambda) \coloneqq \bigcup_{k \in \z_{>0}} \{(k, {\bm a}) \mid {\bm a} \in \Xi(\tilde{\bf s}_t, k\lambda) \cap \z^J\} \subseteq \z_{>0} \times \z^J,\]
and denote by $C(\tilde{\bf s}_t, \lambda)\ (\subseteq \r_{\geq 0} \times \r^J)$ the smallest real closed cone containing $S(\tilde{\bf s}_t, \lambda)$. 
Since $\Xi_t$ is a rational convex polyhedral cone, we have 
\[\Xi(\tilde{\bf s}_t, \lambda) = \{{\bm a} \in \r^J \mid (1, {\bm a}) \in C(\tilde{\bf s}_t, \lambda)\}.\]
For $t \in \mathbb{T}$, let ${\bf s}_t = (\mathbf{A}_t, \varepsilon_t)$ be the seed for $\c[U_{w_0}^-]$ corresponding to $\tilde{\bf s}_t$.
We write $\mathbf{A}_t = (A_{j; t})_{j \in J}$ and $\varepsilon_t = (\varepsilon_{i, j}^{(t)})_{i \in J_{\rm uf}, j \in J}$. 

\begin{thm}\label{t:relation_with_superpotential}
Let $G = SL_{n+1}(\mathbb{C})$.
Then, for each $t \in \mathbb{T}$ and $\lambda \in P_+$, the superpotential polytope $\Xi(\tilde{\bf s}_t, \lambda)$ coincides with the Newton--Okounkov body $\Delta(G/B, \mathcal{L}_\lambda, v_{{\bf s}_t}, \tau_\lambda)$.  
\end{thm}

\begin{proof}
It suffices to prove that 
\begin{equation}\label{eq:relation_with_superpotential_semigroup}
\begin{aligned}
S(\tilde{\bf s}_t, \lambda) = S(G/B, \mathcal{L}_\lambda, v_{{\bf s}_t}, \tau_\lambda).
\end{aligned}
\end{equation}
Let $\mu \in P_+$. 
By \cref{l:base_affine_and unipotent_cell_lambda}, the set $\{\pi(\vartheta_q) \mid q \in \Xi_{\mathbb{Z}} \cap \underline{\rm wt}^{-1} (\mu)\}$ forms a $\c$-basis of $\{(\sigma/\tau_\mu)|_{U_{w_0}^-} \mid \sigma \in H^0(G/B, \mathcal{L}_\mu)\}$. 
In particular, the elements $\pi(\vartheta_q)$, $q \in \Xi_{\mathbb{Z}} \cap \underline{\rm wt}^{-1} (\mu)$, are all distinct.
Take $q \in \Xi_{\mathbb{Z}} \cap \underline{\rm wt}^{-1} (\mu)$. 
Since $\vartheta_q$ is pointed, we can write
\[\vartheta_q = \left(\prod_{j \in \widetilde{J}} \widetilde{A}_{j; t} ^{g_j}\right) \left(\sum_{{\bm a}=(a_i)_i \in \mathbb{Z}_{\geq 0}^{J_{\rm uf}}} c_{\bm a} \prod_{i \in J_{\rm uf}} \left(\prod_{j \in \widetilde{J}} \widetilde{A}_{j; t} ^{\widetilde{\varepsilon}_{i, j} ^{(t)}}\right)^{a_i}\right)\]
for some $\{c_{\bm a} \in \c \mid {\bm a} \in \mathbb{Z}_{\geq 0}^{J_{\rm uf}}\}$ such that $c_0 = 1$, where $g_{\tilde{\bf s}_t} (\vartheta_q) = (g_j)_{j \in \widetilde{J}}$.
Hence it holds by the definition of $\pi$ that
\[\pi(\vartheta_q) = \left(\prod_{j \in J} A_{j; t} ^{g_j}\right) \left(\sum_{{\bm a} = (a_i)_i \in \mathbb{Z}_{\geq 0}^{J_{\rm uf}}} c_{\bm a} \prod_{i \in J_{\rm uf}} \left(\prod_{j \in J} A_{j; t} ^{\varepsilon_{i, j} ^{(t)}}\right)^{a_i}\right),\]
which implies that $\pi(\vartheta_q)$ is pointed and that 
\begin{equation}\label{eq:projection_g-vector}
\begin{aligned}
g_{{\bf s}_t} (\pi(\vartheta_q)) = {\rm pr}_{\r^J} (g_{\tilde{\bf s}_t} (\vartheta_q)).
\end{aligned}
\end{equation}

Suppose for a contradiction that $g_{{\bf s}_t}(\pi(\vartheta_{q_1})) = g_{{\bf s}_t}(\pi(\vartheta_{q_2}))$ for some $q_1 \neq q_2$ in $\Xi_{\mathbb{Z}} \cap \underline{\rm wt}^{-1} (\mu)$. 
Then we see by \eqref{eq:projection_g-vector} that the extended $g$-vectors $g_{\tilde{\bf s}_t}(\vartheta_{q_1}), g_{\tilde{\bf s}_t}(\vartheta_{q_2}) \in \r^{\widetilde{J}} = \r^{\overline{I}} \oplus \r^J$ can differ only at coordinates in $\r^{\overline{I}}$.
Since the theta function basis is stable under the multiplication by frozen variables (see \cite[Theorem 0.3]{GHKK} and \cite[Lemma 4.7]{She}), it follows that $\vartheta_{q_1}$ is obtained from $\vartheta_{q_2}$ by multiplying a Laurent monomial in $\widetilde{D}_{\varpi_i, \varpi_i}$, $i \in I$.
However, this implies that $\pi(\vartheta_{q_1}) = \pi(\vartheta_{q_2})$, which gives a contradiction. 
Thus, we have proved that the extended $g$-vectors $g_{{\bf s}_t}(\pi(\vartheta_{q}))$, $q \in \Xi_{\mathbb{Z}} \cap \underline{\rm wt}^{-1} (\mu)$, are all distinct. 
Hence it follows by Propositions \ref{prop1_val} and \ref{p:valuation_generalizing_g} that 
\[\{v_{{\bf s}_t}(\sigma/\tau_\mu) \mid \sigma \in H^0(G/B, \mathcal{L}_\mu) \setminus \{0\}\} = \{g_{{\bf s}_t}(\pi(\vartheta_{q})) \mid q \in \Xi_{\mathbb{Z}} \cap \underline{\rm wt}^{-1} (\mu)\}.\]
Since this equality holds for all $\mu \in P_+$, we conclude \eqref{eq:relation_with_superpotential_semigroup}, which implies that 
\[\Xi(\tilde{\bf s}_t, \lambda) = \Delta(G/B, \mathcal{L}_\lambda, v_{{\bf s}_t}, \tau_\lambda).\]
\end{proof}

\cref{t:relation_with_superpotential} implies that we can construct a system of affine inequalities defining the Newton--Okounkov polytope $\Delta(G/B, \mathcal{L}_\lambda, v_{{\bf s}_t}, \tau_\lambda)$ from the superpotential $W_{\rm GHKK}$.

\begin{rem}
Let ${\bm i} \in R(w_0)$, and consider the seed $\tilde{\mathbf{s}}_{\bm i}$ for $\mathbb{C}[G^{w_0, e}]$ that corresponds to the seed $\mathbf{s}_{\bm i}$ for $\c[U_{w_0}^-]$. 
Then, under the unimodular transformation $\Delta(G/B, \mathcal{L}_\lambda, v_{{\bf s}_{\bm i}}, \tau_\lambda) \simeq \Delta_{\bm i} (\lambda)$ in \cref{c:relstring} (1), \cref{t:relation_with_superpotential} for $\Xi(\tilde{\bf s}_{\bm i}, \lambda)$ and $\Delta(G/B, \mathcal{L}_\lambda, v_{{\bf s}_{\bm i}}, \tau_\lambda)$ was previously proved by Magee \cite{Mag} and Bossinger--Fourier \cite{BF}. 
\end{rem}

\begin{rem}
In \cref{t:relation_with_superpotential}, we assumed that $G = SL_{n+1}(\mathbb{C})$ only to use \cref{t:Mag_GHKK}. 
Hence, if \cref{t:Mag_GHKK} is extended to other semisimple groups $G$, then \cref{t:relation_with_superpotential} is also extended. 
\end{rem}

\begin{ex}
Let $G = SL_4(\c)$, ${\bm i} = (1, 2, 1, 3, 2, 1) \in R(w_0)$ (see \cref{s:cluster_cone} for the labeling of the vertices of the Dynkin diagram), and consider the corresponding seed $\tilde{\mathbf{s}}_{\bm i}$ for $\mathbb{C}[G^{w_0, e}]$. 
We fix $t \in \mathbb{T}$ such that $\tilde{\bf s}_t = \tilde{\mathbf{s}}_{\bm i}$. 
Then we see by \cite[Corollary 24]{Mag} that the restriction $W_{\rm GHKK}|_{\mathcal{A}^{\vee} _t} \colon \mathcal{A}^{\vee} _t = \Spec (\mathbb{C}[X_{j; t}^{\pm 1} \mid j \in \widetilde{J}]) \rightarrow \c$ is given by 
\begin{align*}
W_{\rm GHKK}|_{\mathcal{A}^{\vee} _t} =& X_{\overline{1}; t}^{-1} + X_{\overline{1}; t}^{-1} X_{1; t}^{-1} + X_{\overline{1}; t}^{-1} X_{1; t}^{-1} X_{3; t}^{-1} + X_{\overline{2}; t}^{-1} + X_{\overline{2}; t}^{-1} X_{2; t}^{-1} + X_{\overline{3}; t}^{-1} \\
+& X_{4; t}^{-1} + X_{2; t}^{-1} X_{4; t}^{-1} + X_{1; t}^{-1} X_{2; t}^{-1}X_{4; t}^{-1} + X_{5; t}^{-1} + X_{3; t}^{-1} X_{5; t}^{-1} + X_{6; t}^{-1};
\end{align*}
see also \cite[Section 5]{BF}.
Hence the cone $\Xi_t$ coincides with the set of $(g_{\overline{1}}, g_{\overline{2}}, g_{\overline{3}}, g_1, \ldots, g_6) \in \r^{\widetilde{J}}$ satisfying the following inequalities:
\begin{align*}
&g_{\overline{1}} \geq 0, g_{\overline{1}} +g_1 \geq 0,\ g_{\overline{1}} +g_1 + g_3 \geq 0,\ g_{\overline{2}} \geq 0,\ g_{\overline{2}} + g_2 \geq 0,\ g_{\overline{3}} \geq 0,\\
&g_4 \geq 0,\ g_2 + g_4 \geq 0,\ g_1 + g_2 + g_4 \geq 0,\ g_5 \geq 0,\ g_3 + g_5 \geq 0,\ g_6 \geq 0. 
\end{align*}
In addition, under the $\mathbb{R}$-linear structure $\mathcal{A}^\vee(\mathbb{R}^T) \simeq \mathcal{A}^{\vee} _t (\r^T) = \r^{\widetilde{J}}$, we see by \cite[Proposition 31 and the paragraph before it]{Mag} that the function $\underline{\rm wt} \colon \mathcal{A}^{\vee} _t (\r^T) = \r^{\widetilde{J}} \rightarrow P \otimes_\z \r$ is written as 
\begin{align*}
\underline{\rm wt} (g_{\overline{1}}, g_{\overline{2}}, g_{\overline{3}}, g_1, \ldots, g_6) =  (g_{\overline{1}} + g_1 + g_3 + g_6) \varpi_1 + (g_{\overline{2}} + g_2 + g_5) \varpi_2 + (g_{\overline{3}} + g_4) \varpi_3.
\end{align*}
Hence for $\lambda = \lambda_1 \varpi_1 + \lambda_2 \varpi_2 + \lambda_3 \varpi_3 \in P_+$, the superpotential polytope $\Xi(\tilde{\bf s}_t, \lambda)$ is given as the set of $(g_1, \ldots, g_6) \in \r^J$ satisfying the following inequalities:
\begin{align*}
&\lambda_1 - g_1 - g_3  -g_6 \geq 0, \lambda_1 - g_3  -g_6 \geq 0,\ \lambda_1 -g_6 \geq 0,\ \lambda_2 - g_2 - g_5 \geq 0,\ \lambda_2 - g_5 \geq 0,\ \lambda_3 -g_4 \geq 0,\\
&g_4 \geq 0,\ g_2 + g_4 \geq 0,\ g_1 + g_2 + g_4 \geq 0,\ g_5 \geq 0,\ g_3 + g_5 \geq 0,\ g_6 \geq 0,
\end{align*}
which coincides with $\Delta(G/B, \mathcal{L}_\lambda, v_{{\bf s}_{\bm i}}, \tau_\lambda)$. 
Now we have $\overrightarrow{\boldsymbol{\mu}_{\bm{i}}}^{\vee} = \mu_3 \mu_1 \mu_2 \mu_3$. 
Take $t^\prime \in \mathbb{T}$ such that $\tilde{\bf s}_{t^\prime} = \mu_3 \mu_1 \mu_2 \mu_3 (\tilde{\mathbf{s}}_{\bm i})$, which corresponds to the seed $\mathbf{s}_{\bm i}^{\rm mut}$ for $\c[U_{w_0}^-]$. 
Then we can compute $W_{\rm GHKK}|_{\mathcal{A}^{\vee} _{t^\prime}}$ from $W_{\rm GHKK}|_{\mathcal{A}^{\vee} _t}$ by applying the mutation sequence $\mu_3 \mu_1 \mu_2 \mu_3$, and we see that 
\begin{align*}
W_{\rm GHKK}|_{\mathcal{A}^{\vee} _{t^\prime}} =& X_{\overline{1}; t}^{-1} + X_{\overline{2}; t}^{-1} + X_{\overline{2}; t}^{-1} X_{1; t}^{-1} + X_{\overline{3}; t}^{-1} + X_{\overline{3}; t}^{-1} X_{2; t}^{-1} + X_{\overline{3}; t}^{-1} X_{2; t}^{-1} X_{3; t}^{-1} \\
+& X_{4; t}^{-1} + X_{5; t}^{-1} + X_{3; t}^{-1} X_{5; t}^{-1} + X_{6; t}^{-1} + X_{1; t}^{-1} X_{6; t}^{-1} + X_{1; t}^{-1} X_{2; t}^{-1}X_{6; t}^{-1}.
\end{align*}
In addition, the function $\underline{\rm wt} \colon \mathcal{A}^{\vee} _{t^\prime} (\r^T) = \r^{\widetilde{J}} \rightarrow P \otimes_\z \r$ is given by 
\begin{align*}
\underline{\rm wt} (g_{\overline{1}}, g_{\overline{2}}, g_{\overline{3}}, g_1, \ldots, g_6) = (g_{\overline{1}} + g_6) \varpi_1 + (g_{\overline{2}} + g_1 + g_5) \varpi_2 + (g_{\overline{3}} + g_2 + g_3 + g_4) \varpi_3.
\end{align*}
Hence the superpotential polytope $\Xi(\tilde{\bf s}_{t^\prime}, \lambda)$ for $\lambda = \lambda_1 \varpi_1 + \lambda_2 \varpi_2 + \lambda_3 \varpi_3 \in P_+$ is the set of $(g_1, \ldots, g_6) \in \r^J$ satisfying the following inequalities:
\begin{align*}
&\lambda_1 -g_6 \geq 0, \lambda_2 - g_1  -g_5 \geq 0,\ \lambda_2 -g_5 \geq 0,\ \lambda_3 - g_2 - g_3 -g_4 \geq 0,\ \lambda_3 - g_3 - g_4 \geq 0,\ \lambda_3 -g_4 \geq 0,\\
&g_4 \geq 0,\ g_5 \geq 0,\ g_3 + g_5 \geq 0,\ g_6 \geq 0,\ g_1 + g_6 \geq 0,\ g_1 + g_2 + g_6 \geq 0,
\end{align*}
which coincides with $\Delta(G/B, \mathcal{L}_\lambda, v_{{\bf s}_{\bm i} ^{\rm mut}}, \tau_\lambda)$.
\end{ex}

\section{Generators of cluster cones}\label{s:cluster_cone}

In this section, we address the following question: which elements of $\c[U^- _w]$ do give generators of the cluster cone $C_{\bf s}$? In some specific cases when $w=w_0$, we compute the minimal generators of the cluster cone, which are given as extended $g$-vectors of specific cluster variables. 
Let $C \subset \r^m$ be a rational strongly convex polyhedral cone. Since $C$ is strongly convex, a $1$-dimensional face $\sigma$ of $C$ is a ray, i.e., a half line. Then, since $C$ is rational, the ray $\sigma$ can be written as $\sigma = \r_{\ge 0} u_\sigma$ for some primitive element $u_\sigma \in \z^m$. We call $u_\sigma$ a \emph{ray generator} of $C$. 
Denoting by $\mathcal{R}(C)$ the set of ray generators of $C$, we have
\[C = \sum_{u \in \mathcal{R}(C)} \r_{\ge 0} u.\]
Assume that $G$ is simple of classical type or of type $G_2$. 
Let us identify the index set $I$ of vertices of the Dynkin diagram with $\{1, 2, \ldots, n\}$ as follows:
\begin{align*}
&A_n\ \begin{xy}
\ar@{-} (50,0) *++!D{1} *\cir<3pt>{};
(60,0) *++!D{2} *\cir<3pt>{}="C"
\ar@{-} "C";(65,0) \ar@{.} (65,0);(70,0)^*!U{}
\ar@{-} (70,0);(75,0) *++!D{n-1} *\cir<3pt>{}="D"
\ar@{-} "D";(85,0) *++!D{n} *\cir<3pt>{}="E"
\end{xy}\hspace{-1mm},& &B_n\ \begin{xy}
\ar@{<=} (50,0) *++!D{1} *\cir<3pt>{};
(60,0) *++!D{2} *\cir<3pt>{}="C"
\ar@{-} "C";(65,0) \ar@{.} (65,0);(70,0)^*!U{}
\ar@{-} (70,0);(75,0) *++!D{n-1} *\cir<3pt>{}="D"
\ar@{-} "D";(85,0) *++!D{n} *\cir<3pt>{}="E"
\end{xy}\hspace{-1mm}, & &\\ 
&C_n\ \begin{xy}
\ar@{=>} (50,0) *++!D{1} *\cir<3pt>{};
(60,0) *++!D{2} *\cir<3pt>{}="C"
\ar@{-} "C";(65,0) \ar@{.} (65,0);(70,0)^*!U{}
\ar@{-} (70,0);(75,0) *++!D{n-1} *\cir<3pt>{}="D"
\ar@{-} "D";(85,0) *++!D{n} *\cir<3pt>{}="E"
\end{xy}\hspace{-1mm},& &D_n\ \begin{xy}
\ar@{-} (30,0) *++!D{3} *\cir<3pt>{}="C";
(35,0)
\ar@{.} (35,0);(40,0)^*!U{}
\ar@{-} (40,0);(45,0) *++!D{n-1} *\cir<3pt>{}="D"
\ar@{-} "D";(55,0) *++!D{n} *\cir<3pt>{}
\ar@{-} "C";(20,4) *++!R{1} *\cir<3pt>{}
\ar@{-} "C";(20,-4) *++!R{2} *\cir<3pt>{},
\end{xy}\hspace{-1mm},& &G_2\ \begin{xy}
\ar@3{->} (50,0) *++!D{1} *\cir<3pt>{};
(60,0) *++!D{2} *\cir<3pt>{}
\end{xy},
\end{align*}
where we set $n=2$ when we consider type $G_2$.
Define ${\bm i}_1 \in R(w_0)$ as follows. 
\begin{itemize}
\item If $G$ is of type $A_n$, then 
\[
{\bm i}_1 \coloneqq (n, n-1, n, n-2, n-1, n, \ldots, 1, 2, \ldots, n) \in I^{\frac{n(n+1)}{2}}.
\]
\item If $G$ is of type $B_n$ or $C_n$, then
\[{\bm i}_1 \coloneqq (1, \underbrace{2, 1, 2}_3, \underbrace{3, 2, 1, 2, 3}_5, \ldots, \underbrace{n, n-1, \ldots, 1, \ldots, n-1, n}_{2n-1}) \in I^{n^2}.\] 
\item If $G$ is of type $D_n$, then
\[{\bm i}_1 \coloneqq (1, 2, \underbrace{3, 1, 2, 3}_4, \underbrace{4, 3, 1, 2, 3, 4}_6, \ldots, \underbrace{n, n-1, \ldots, 3, 1, 2, 3, \ldots, n-1, n}_{2n-2}) \in I^{n (n-1)}.\]
\item If $G$ is of type $G_2$, then ${\bm i}_1 \coloneqq (1, 2, 1, 2, 1, 2) \in I^6$.
\end{itemize}
Recall from \cref{c:relation_with_string_cones} (1) that $\widetilde{C}_{{\bm i}_1}$ coincides with the string cone associated with ${\bm i}_1$. 
Littelmann \cite[Corollary 2 and Theorems 5.1, 6.1, 7.1]{Lit} gave a system of explicit linear inequalities defining the string cone $\widetilde{C}_{{\bm i}_1}$ for the reduced word ${\bm i}_1$ above. 
By this description, we obtain the following. 
\begin{itemize}
\item If $G$ is of type $A_n$, then 
\[\mathcal{R}(\widetilde{C}_{{\bm i}_1}) = \{(\underbrace{0, \ldots, 0}_{\frac{i (i-1)}{2}}, \underbrace{1, \ldots, 1}_j, 0, \ldots, 0) \in \r^{\frac{n(n+1)}{2}} \mid 1 \le i \le n,\ 1 \le j \le i\}.\]
\item If $G$ is of type $B_n$, then 
\begin{align*}
\mathcal{R}(\widetilde{C}_{{\bm i}_1}) &= \{(\underbrace{0, \ldots, 0}_{(i-1)^2}, \underbrace{1, \ldots, 1}_j, 0, \ldots, 0) \in \r^{n^2} \mid 2 \le i \le n,\ 1 \le j \le i -1\}\\
&\cup \{(\underbrace{0, \ldots, 0}_{(i-1)^2}, \underbrace{1, \ldots, 1}_{i-1}, 2, \underbrace{1, \ldots, 1}_j, 0, \ldots, 0) \in \r^{n^2} \mid 2 \le i \le n,\ 0 \le j \le i-1\}\\
&\cup \{(1, 0, 0, \ldots, 0) \in \r^{n^2}\}.
\end{align*}
\item If $G$ is of type $C_n$, then 
\[\mathcal{R}(\widetilde{C}_{{\bm i}_1}) = \{(\underbrace{0, \ldots, 0}_{(i-1)^2}, \underbrace{1, \ldots, 1}_j, 0, \ldots, 0) \in \r^{n^2} \mid 1 \le i \le n,\ 1 \le j \le 2i -1\}.\]
\item If $G$ is of type $D_n$, then 
\begin{align*}
\mathcal{R}(\widetilde{C}_{{\bm i}_1}) &= \{(\underbrace{0, \ldots, 0}_{i(i-1)}, \underbrace{1, \ldots, 1}_j, 0, \ldots, 0) \in \r^{n(n-1)} \mid 2 \le i \le n-1,\ 1 \le j \le 2i\}\\
&\cup \{(\underbrace{0, \ldots, 0}_{i(i-1)}, \underbrace{1, \ldots, 1}_{i-1}, 0, 1, 0, \ldots, 0) \in \r^{n(n-1)} \mid 2 \le i \le n-1\} \\
&\cup \{(1, 0, 0, \ldots, 0), (0, 1, 0, \ldots, 0) \in \r^{n(n-1)}\}.
\end{align*}
\item If $G$ is of type $G_2$, then 
\begin{align*}
\mathcal{R}(\widetilde{C}_{{\bm i}_1}) &= \{(1, 0, 0, 0, 0, 0), (0, 1, 0, 0, 0, 0), (0, 1, 1, 0, 0, 0),\\
&\quad\quad (0, 1, 1, 2, 0, 0), (0, 1, 1, 2, 1, 0), (0, 1, 1, 2, 1, 1)\}.
\end{align*}
\end{itemize}

For $1 \le k \le n$, let $\mathfrak{g}_k \subset \mathfrak{g}$ denote the Lie subalgebra generated by $\{e_i, f_i, h_i \mid 1 \le i \le k\}$, $W^{(k)}$ the Weyl group of $\mathfrak{g}_k$, and $w_0 ^{(k)} \in W^{(k)}$ the longest element. 
We define a set $\mathcal{D}_1$ of unipotent minors by
\[
\mathcal{D}_1 \coloneqq \bigcup_{1 \le k \le n} \{D_{w_0 ^{(k)} \varpi_k, w \varpi_k} \mid w \in W^{(k)},\ w \varpi_k \neq w_0 ^{(k)} \varpi_k\}.
\]
Then we obtain the following lemma by \cref{r:minorvariable}.

\begin{lem}\label{l:explicit_cluster_variables}
All elements of $\mathcal{D}_1$ are cluster variables of $\c[U_{w_0} ^-]$.
\end{lem}

\begin{prop}\label{p:minimal_generators_cluster_variables}
Let $G$ and ${\bm i}_1$ be as above, and use the lexicographic order $\prec$ in \cref{d:lowest_term_valuation_Schubert} to construct the valuation $v_{{\bf s}_{{\bm i}_1}}$. 
Then the equality 
\[\mathcal{R}(C_{{\bf s}_{{\bm i}_1}}) = v_{{\bf s}_{{\bm i}_1}} (\mathcal{D}_1)\]
holds. In particular, all elements of $\mathcal{R}(C_{{\bf s}_{{\bm i}_1}})$ are extended $g$-vectors of cluster variables.
\end{prop}

\begin{proof}
It suffices to prove the first assertion since the second assertion follows from the first one by \cref{c:cluster_monomial_valuation} and \cref{l:explicit_cluster_variables}. 
By \cref{t:relstring} and \cref{c:relations_string_cone}, it is enough to show that
\begin{equation}\label{eq:ray_generators_valuations}
\begin{aligned}
\mathcal{R}(\widetilde{C}_{{\bm i}_1}) = \tilde{v}_{{\bm i}_1} ^{\rm low} (\mathcal{D}_1).
\end{aligned}
\end{equation}
We write ${\bm i}_1 = (i_1, \ldots, i_m)$. 
For $1 \le k \le n$ and $w \in W^{(k)}$, the definitions of $D_{w_0 ^{(k)} \varpi_k, w \varpi_k}$ and $y_{{\bm i}_1} ^\ast \colon \c [U^- _{w_0}] \hookrightarrow \c[t_1 ^{\pm 1}, \ldots, t_m ^{\pm 1}]$ imply that 
\[
y_{{\bm i}_1} ^\ast (D_{w_0 ^{(k)} \varpi_k, w \varpi_k}) = \langle f_{w_0 ^{(k)} \varpi_k}, y_{i_1} (t_1) y_{i_2} (t_2) \cdots y_{i_m} (t_m) v_{w \varpi_k} \rangle.
\]
Hence \eqref{eq:ray_generators_valuations} follows by the graphs of the crystal bases for the fundamental representations $V(\varpi_i)$, $i \in I$, given in \cite{KasNak, KanMis}; see \cite{Kas5} for a survey on crystal bases. 
This proves the proposition.
\end{proof}

\begin{ex}
Let $G = Sp_4(\c)$.
Then we have ${\bm i}_1 = (1, 2, 1, 2) \in R(w_0)$ and
\[\mathcal{R}(\widetilde{C}_{{\bm i}_1}) = \{(1,0,0,0), (0,1,0,0), (0,1,1,0), (0,1,1,1)\}.\]
By \cref{c:relations_string_cone}, it follows that 
\[\mathcal{R}(C_{{\bf s}_{{\bm i}_1}}) = \{{\bm a} M_{{\bm i}_1} \mid {\bm a} \in \mathcal{R}(\widetilde{C}_{{\bm i}_1})\} = \{(1,0,0,0), (-1,1,0,0), (0,-1,1,0), (0,0,0,1)\},\]
where we used the matrix $M_{{\bm i}_1}$ computed in \cref{ex:NO_body_string_C2}. 
In addition, the set $\mathcal{D}_1$ is given by 
\[
\mathcal{D}_1 = \{D_{s_1 \varpi_1, \varpi_1}, D_{w_0 \varpi_2, s_1 s_2 \varpi_2}, D_{w_0 \varpi_2, s_2 \varpi_2}, D_{w_0 \varpi_2, \varpi_2}\},
\]
which corresponds to $\mathcal{R}(C_{{\bf s}_{{\bm i}_1}})$ by \cref{p:minimal_generators_cluster_variables}.
Indeed, we see that 
\begin{align*}
v_{{\bf s}_{{\bm i}_1}}(D_{s_1 \varpi_1, \varpi_1}) &= (1,0,0,0), &v_{{\bf s}_{{\bm i}_1}}(D_{w_0 \varpi_2, s_1 s_2 \varpi_2}) &= (-1,1,0,0),\\
v_{{\bf s}_{{\bm i}_1}}(D_{w_0 \varpi_2, s_2 \varpi_2}) &= (0,-1,1,0), &v_{{\bf s}_{{\bm i}_1}}(D_{w_0 \varpi_2, \varpi_2}) &= (0,0,0,1).
\end{align*}
\end{ex}

Let $i \mapsto i^\ast$ be the involution on $I$ defined by $w_0 (\alpha_i) = - \alpha_{i^\ast}$ for $i \in I$.
We define ${\bm i}_2 \in R(w_0)$ by applying this involution to all entries of ${\bm i}_1$. 
More concretely, ${\bm i}_2$ is given as follows. 
\begin{itemize}
\item If $G$ is of type $A_n$, then ${\bm i}_2 = (1, 2, 1, 3, 2, 1, \ldots, n, n-1, \ldots, 1) \in I^{\frac{n(n+1)}{2}}$. 
\item If $G$ is of type $B_n$, $C_n$, or $G_2$, then ${\bm i}_2 = {\bm i}_1$.
\item If $G$ is of type $D_n$ and $n$ is even, then ${\bm i}_2 = {\bm i}_1$. 
If $G$ is of type $D_n$ and $n$ is odd, then ${\bm i}_2$ is obtained from ${\bm i}_1$ by replacing $1$ and $2$ with $2$ and $1$, respectively.
\end{itemize}
Since the equality $C_{{\bm i}_2^{\rm op}} = \widetilde{C}_{{\bm i}_2}^{\rm op}$ holds by \cref{c:relation_with_string_cones} (3), we see that 
\[\mathcal{R} (C_{{\bm i}_2^{\rm op}}) = \mathcal{R} (\widetilde{C}_{{\bm i}_2}^{\rm op}) = \mathcal{R} (\widetilde{C}_{{\bm i}_1}^{\rm op}).\]
We define a set $\mathcal{D}_2$ of unipotent minors by
\[
\mathcal{D}_2 \coloneqq \bigcup_{1 \le k \le n} \{D_{w \varpi_k, \varpi_k} \mid w \in W^{(k)},\ w \varpi_k \neq \varpi_k\}.
\]
Then all elements of $\mathcal{D}_2$ are cluster variables of $\c[U_{w_0} ^-]$ by \cref{r:minorvariable}.
In a way similar to the proof of \cref{p:minimal_generators_cluster_variables}, we obtain the following. 

\begin{prop}\label{p:i3mut}
Let $G$ and ${\bm i}_2$ be as above, and set ${\bm i}_3 \coloneqq {\bm i}_2 ^{\rm op}$. 
Define $v_{{\bf s}_{{\bm i}_3}^{\rm mut}}$ as in \cref{t:relNZ}.
Then the equality 
\[\mathcal{R}(C_{{\bf s}_{{\bm i}_3}^{\rm mut}}) = v_{{\bf s}_{{\bm i}_3}^{\rm mut}} (\mathcal{D}_2)\]
holds. 
In particular, all elements of $\mathcal{R}(C_{{\bf s}_{{\bm i}_3}^{\rm mut}})$ are extended $g$-vectors of cluster variables.
\end{prop}

\begin{rem}
    In general, there is no $\bm{j}\in R(w_0)$ satisfying ${\bf s}_{{\bm j}}={\bf s}_{{\bm i}_3}^{\rm mut}$. For example, when $G$ is of type $B_2$, $\bm{i}_3=(2,1,2,1)$ and the quiver associated with $\varepsilon^{\bm{i}_3, {\rm mut}}$ is given in \cref{e:smut}. Here note that the FZ-seed ${\bf s}_{{\bm i}_3}$ can be identified with the FZ-seed ${\bf s}_{{\bm i}}$ associated with $G=Sp_4(\mathbb{C})$ (type $C_2$) and $\bm{i}=(1,2,1,2)$. Then we can show that it is different from the one associated with $\varepsilon^{\bm{i}}$ for $\bm{i}\in R(w_0)=\{(1,2,1,2), (2,1,2,1)\}$. 
\end{rem}

On the other hand, when $G$ is of type $A_n$, we have the following.
\begin{prop}
Let $G=SL_{n+1}(\c)$ and ${\bm i}_1, {\bm i}_3$ be as above. Then 
\[
{\bf s}_{{\bm i}_3}^{\rm mut}={\bf s}_{{\bm i}_1}
\]
up to a relabeling of the index set of FZ-seeds. 
\end{prop}
\begin{proof}
Set $m\coloneqq \frac{n(n+1)}{2}, J=\{1,\dots,m\}$, and write the sets $J_{\rm uf}$ and $J_{\rm fr}$ associated with $\bm{i}_t$ ($t=1, 3$) as $J_{\rm uf}^{(\bm{i}_t)}$ and $J_{\rm fr}^{(\bm{i}_t)}$, respectively. By the argument in the proof of \cref{t:relNZ}, we have 
    \[
D^{\rm mut}(s, \bm{i}_3)=\begin{cases}
(\eta_{w_0}^{\ast})^{-1}(D(s^{\vee}, \bm{i}_3))\prod_{i\in I}D_{w_0\varpi_i, \varpi_i}^{f_i^{(s)}}&\text{ if } s\in J_{\rm uf}^{(\bm{i}_3)},\\
D_{w_0\varpi_{i_s}, \varpi_{i_s}}&\text{ if } s\in J_{\rm fr}^{(\bm{i}_3)}
\end{cases}
\]
for some $f_i^{(s)}\in \mathbb{Z}$ ($i\in I$, $s\in J$). When $s^{\vee}\in J_{\rm uf}^{(\bm{i}_3)}$ satisfies $\xi_{\bm{i}_3}(s^{\vee})=(i, k)$ (recall Section \ref{ss:mutation_sequence}), we have 
\[
D(s^{\vee}, \bm{i}_3)=D_{[k+1, k+i], [1,i]}=D_{[k+1, n+1], [1,i]\cup [k+i+1,n+1]}
\]
Here $[a, b] \coloneqq \{a, a+1,\ldots, b\}$ for $a, b \in \z_{>0}$ with $a\leq b$, and for $K, K'\subset [1, n+1]$ with $|K|=|K'|$, $D_{K, K'}$ denotes the minor associated with the row set $K$ and the column set $K'$ on the unipotent subgroup $U^-$ consisting of the unipotent lower triangular matrices in $SL_{n+1}(\c)$. Then there exists $\widetilde{s}\in J_{\rm uf}^{(\bm{i}_1)}$ satisfying $\xi_{\bm{i}_1}(\widetilde{s})=(n-k+1, n-k+1-i)$ and 
\[
D_{[k+1, n+1], [1,i]\cup [k+i+1,n+1]}=D_{w_0\varpi_{n-k+1}, w_{\leq \widetilde{s}}\varpi_{n-k+1}}, 
\]
here we consider the element $w_{\leq \widetilde{s}}$ associated with $\bm{i}_1$. Note that 
\begin{align*}
&\xi_{\bm{i}_3}(J_{\rm uf}^{(\bm{i}_3)})=\{(i, k)\mid 1\leq i\leq n-1,\ 1\leq k\leq n-i\},\\
&\xi_{\bm{i}_1}(J_{\rm uf}^{(\bm{i}_1)})=\{(i, k)\mid 2\leq i\leq n,\ 1\leq k\leq i-1\},
\end{align*}
and the correspondence $J_{\rm uf}^{(\bm{i}_3)}\to J_{\rm uf}^{(\bm{i}_1)},\ s^{\vee}\mapsto \widetilde{s},$ gives a bijection. Hence we have 
\begin{align*}
(\eta_{w_0}^{\ast})^{-1}(D(s^{\vee}, \bm{i}_3))&=(\eta_{w_0}^{\ast})^{-1}(D_{w_0\varpi_{n-k+1}, w_{\leq \widetilde{s}}\varpi_{n-k+1}})\\
&=\frac{D_{w_{\leq \widetilde{s}}\varpi_{n-k+1}, \varpi_{n-k+1}}}{D_{w_{0}\varpi_{n-k+1}, \varpi_{n-k+1}}}\\
&=\frac{D( \widetilde{s}, \bm{i}_1)}{D_{w_{0}\varpi_{n-k+1}, \varpi_{n-k+1}}}.
\end{align*}
Therefore,
\[
D^{\rm mut}(s, \bm{i}_3)=\frac{D( \widetilde{s}, \bm{i}_1)}{D_{w_{0}\varpi_{n-k+1}, \varpi_{n-k+1}}}\prod_{i\in I}D_{w_0\varpi_i, \varpi_i}^{f_i^{(s)}}\in \c[U^-]. 
\]
By \cite[Corollary 2.3]{GLS:factorial}, it implies
\[
D^{\rm mut}(s, \bm{i}_3)=D(\widetilde{s}, \bm{i}_1).
\]
Therefore, as a set, 
\begin{align*}
   \{ D^{\rm mut}(s, \bm{i}_3)\mid s\in J\}&=\{D(s, \bm{i}_1)\mid s\in J\}. 
\end{align*}
Hence the fact \cite[Proposition 3]{CL} implies ${\bf s}_{{\bm i}_3}^{\rm mut}={\bf s}_{{\bm i}_1}$ up to a relabeling of the index set of FZ-seeds. 
\end{proof}

\appendix

\section{Investigation of specific mutation sequences}\label{a:mutation_sequence}

In this appendix, we give proofs of \cref{t:mutation_max,t:mutation_max_variable}. We adopt \cref{n:fix_w} in this appendix, and recall the notation in Section \ref{ss:mutation_sequence}. 

First it is convenient to introduce a skew-symmetrizable matrix $\varepsilon^{\bm{i}, \mathrm{ex}}=(\varepsilon_{s, t})_{s, t\in J}$ obtained from $\Gamma_{\bm{i}}$ by the rule \eqref{eq:recoverexchange}. Note that $\varepsilon_{s, t} = 0$ for $s, t\in J_{\rm fr}$, and $\varepsilon^{\bm{i}}$ is the $J_{\rm uf} \times J$-submatrix of $\varepsilon^{\bm{i}, \mathrm{ex}}$. We can consider the mutations of $\varepsilon^{\bm{i}, \mathrm{ex}}$ by \eqref{eq_mutation_for_varepsilon}, and if we apply a mutation sequence $\mu_{s_k}\cdots \mu_{s_1}$ such that $s_1, \ldots, s_k\in J_{\rm uf}$, then the $J_{\rm uf}
 \times J$-submatrix of $\mu_{s_k}\cdots \mu_{s_1}(\varepsilon^{\bm{i}, \mathrm{ex}})$ is equal to $\mu_{s_k}\cdots \mu_{s_1}(\varepsilon^{\bm{i}})$. Hence, in the following, we study mutations of $\varepsilon^{\bm{i}, \mathrm{ex}}$ (in directions $s\in J_{\rm uf}$). 

Assume that we have a skew-symmetrizable $J\times J$-matrix $\varepsilon$ which is recovered from the quiver of the following form by the rule \eqref{eq:recoverexchange}: 

\hfill
\begin{xy} 0;<1pt,0pt>:<0pt,-1pt>::
	(-73,60) *+{i},
	(-65,60) *+{\scalebox{0.65}{$>$}},	
	(-73,30) *+{j'},
	(-65,30) *+{\scalebox{0.65}{$>$}},		
	(-73,90) *+{j''},
	(-65,90) *+{\scalebox{0.65}{$>$}},			
	(-30,60) *+{\scalebox{0.9}{$s^-$}} ="-1",	
	(40,30) *+{\scalebox{0.9}{$t'$}}="0",
	(0,60) *+{\scalebox{0.9}{$s$}} ="1",
	(60,60) *+{\scalebox{0.9}{$s^+$}} ="2",
	(200,30) *+{\scalebox{0.9}{$(t')^+$}}="3",
	(120,60) *+{\scalebox{0.9}{$s^{(2)}$}} ="4",
	(135,60) *+{\cdots},
	(155,60) *+{\scalebox{0.9}{$s^{(k-1)}$}} ="10",		
	(190,60) *+{\scalebox{0.9}{$s^{(k)}$}} ="5",	
	(300,60) *+{\scalebox{0.9}{$s^{(k+1)}$}} ="6",
	(220,30) *+{\cdots},
	(260,30) *+{\cdots},
	(280,30) *+{\scalebox{0.9}{$(t')^{(l')}$}}="8",	
	(325,30) *+{\scalebox{0.9}{$(t')^{(l'+1)}$}}="9",
	(30,90) *+{\scalebox{0.9}{$t''$}} ="11",									
	(70,90) *+{\scalebox{0.9}{$(t'')^+$}} ="12",
	(87,90) *+{\cdots},
	(110,90) *+{\scalebox{0.9}{$(t'')^{(l'')}$}} ="13",
	(-50,60) *+{\cdots},	
(-20,30) *+{\cdots},
(-20,90) *+{\cdots},
(130,90) *+{\cdots},
(320,60) *+{\cdots},
(348,30) *+{\cdots},	
	(-73,120) *+{j'''},
(-65,120) *+{\scalebox{0.65}{$>$}},
	(20,120) *+{\scalebox{0.9}{$t'''$}} ="14",									
(75,120) *+{\scalebox{0.9}{$(t''')^+$}} ="15",
(95,120) *+{\cdots},
(118,120) *+{\scalebox{0.9}{$(t''')^{(l''')}$}} ="16",
(0,120) *+{\cdots},
(138,120) *+{\cdots},
(-73,0) *+{j},
(-65,0) *+{\scalebox{0.65}{$>$}},
(25,0) *+{\cdots},	
(45,0) *+{\scalebox{0.9}{$t$}}="17",
(210,0) *+{\scalebox{0.9}{$t^+$}}="18",
(248,0) *+{\scalebox{0.9}{$t^{(2)}$}}="19",
(265,0) *+{\cdots},
(285,0) *+{\scalebox{0.9}{$t^{(l)}$}}="20",	
(320,0) *+{\scalebox{0.9}{$t^{(l+1)}$}}="21",
(340,0) *+{\cdots},
	"18", {\ar"17"},	
	"19", {\ar"18"},		
	"21", {\ar"20"},	
	"1", {\ar"17"},	
	"17", {\ar"5"},	
	"5", {\ar"20"},				
	"14", {\ar"-1"},	
	"1", {\ar"14"},
	"15", {\ar"14"},
	"14", {\ar"2"},
	"2", {\ar"16"},			
	"17", {\ar"-1"},
	"-1", {\ar"1"},
	"1", {\ar"0"},
	"2", {\ar"1"},
	"4", {\ar"2"},
	"3", {\ar"0"},				
	"0", {\ar"5"},
	"6", {\ar"5"},		
	"9", {\ar"8"},
	"5", {\ar"8"},
	"5", {\ar"10"},					
	"12", {\ar"11"},
	"1", {\ar"11"},
	"11", {\ar"2"},
	"2", {\ar"13"},
\end{xy}
\hfill
\hfill

Here $r^{(k)}$ ($r\in J,\ k\in\mathbb{Z}_{>0}$) is defined as 
\begin{align*}
r^{(1)}&\coloneqq r^+,&
r^{(k)}&\coloneqq (r^{(k-1)})^+\text{ for }k\in \mathbb{Z}_{>1};
\end{align*}
note that we consider $r^+$ with respect to $\bm{i}$.

Applying the mutation $\mu_s$ to $\varepsilon$ in direction $s$, we obtain a skew-symmetrizable $J\times J$-matrix which is recovered from the following quiver by the rule \eqref{eq:recoverexchange}:

\hfill
\begin{xy} 0;<1pt,0pt>:<0pt,-1pt>::
	(-73,60) *+{i},
(-65,60) *+{\scalebox{0.65}{$>$}},	
(-73,30) *+{j'},
(-65,30) *+{\scalebox{0.65}{$>$}},		
(-73,90) *+{j''},
(-65,90) *+{\scalebox{0.65}{$>$}},			
(-30,60) *+{\scalebox{0.9}{$s^-$}} ="-1",	
(40,30) *+{\scalebox{0.9}{$t'$}}="0",
(0,60) *+{\scalebox{0.9}{$s$}} ="1",
(60,60) *+{\scalebox{0.9}{$s^+$}} ="2",
(200,30) *+{\scalebox{0.9}{$(t')^+$}}="3",
(120,60) *+{\scalebox{0.9}{$s^{(2)}$}} ="4",
(135,60) *+{\cdots},
(155,60) *+{\scalebox{0.9}{$s^{(k-1)}$}} ="10",		
(190,60) *+{\scalebox{0.9}{$s^{(k)}$}} ="5",	
(300,60) *+{\scalebox{0.9}{$s^{(k+1)}$}} ="6",
	(220,30) *+{\cdots},
(260,30) *+{\cdots},
(280,30) *+{\scalebox{0.9}{$(t')^{(l')}$}}="8",	
(325,30) *+{\scalebox{0.9}{$(t')^{(l'+1)}$}}="9",
(30,90) *+{\scalebox{0.9}{$t''$}} ="11",									
(70,90) *+{\scalebox{0.9}{$(t'')^+$}} ="12",
(87,90) *+{\cdots},
(110,90) *+{\scalebox{0.9}{$(t'')^{(l'')}$}} ="13",
(-50,60) *+{\cdots},	
(-5,30) *+{\cdots},
(-20,90) *+{\cdots},
(130,90) *+{\cdots},
(320,60) *+{\cdots},
(348,30) *+{\cdots},
	(-73,120) *+{j'''},
(-65,120) *+{\scalebox{0.65}{$>$}},
(20,120) *+{\scalebox{0.9}{$t'''$}} ="14",									
(75,120) *+{\scalebox{0.9}{$(t''')^+$}} ="15",
(95,120) *+{\cdots},
(118,120) *+{\scalebox{0.9}{$(t''')^{(l''')}$}} ="16",
(0,120) *+{\cdots},
(138,120) *+{\cdots},
(-73,0) *+{j},
(-65,0) *+{\scalebox{0.65}{$>$}},
(25,0) *+{\cdots},	
(45,0) *+{\scalebox{0.9}{$t$}}="17",
(210,0) *+{\scalebox{0.9}{$t^+$}}="18",
(248,0) *+{\scalebox{0.9}{$t^{(2)}$}}="19",
(265,0) *+{\cdots},
(285,0) *+{\scalebox{0.9}{$t^{(l)}$}}="20",	
(320,0) *+{\scalebox{0.9}{$t^{(l+1)}$}}="21",
(340,0) *+{\cdots},
"18", {\ar"17"},	
"19", {\ar"18"},		
"21", {\ar"20"},	
"17", {\ar"1"},	
"17", {\ar"5"},	
"5", {\ar"20"},	
	"14", {\ar"1"},
	"15", {\ar"14"},
	"2", {\ar"16"},
	"1", {\ar"-1"},
	"0", {\ar"1"},
	"1", {\ar"2"},
	"4", {\ar"2"},
	"3", {\ar"0"},				
	"0", {\ar"5"},
	"6", {\ar"5"},		
	"9", {\ar"8"},
	"5", {\ar"8"},
	"5", {\ar"10"},					
	"12", {\ar"11"},
	"11", {\ar"1"},
	"-1", {\ar"11"},	
	"2", {\ar"17"},
	"2", {\ar"0"},	
	"2", {\ar"13"},
	"-1", {\ar"0"},		
\end{xy}
\hfill
\hfill

Then $s^+$ in this mutated quiver is in the same situation as $s$ of the original quiver. Indeed, 
\begin{itemize}
	\item the line $j$ for new $s^+$ corresponds to the line $j$ or $j'''$ for previous $s$,
	\item the line $j'$ for new $s^+$ corresponds to the line $j$ or $j'''$ for previous $s$,
	\item the line $j''$ for new $s^+$ corresponds to the line $j'$ or $j''$ for previous $s$,
	\item the line $j'''$ for new $s^+$ corresponds to the line $j'$ or $j''$ for previous $s$.		
\end{itemize}
Moreover, even if 
\begin{itemize}
	\item $s^-$ does not exist (that is, $\xi_{\bm{i}}(s)=(i, 1)$), or 
	\item there exist arrows among the lines $j$, $j'$, $j''$, and $j'''$, or 
	\item there are more (or no) lines of the same form as the lines $j$, $j'$, $j''$, or $j'''$, 
\end{itemize}
it does not make our situation more complicated (we only need an obvious modification). Hence we can calculate $\overleftarrow{\boldsymbol{\mu}_{\bm{i}}}[1](\varepsilon^{\bm{i}, \mathrm{ex}})$ by iterated application of the above argument. More precisely, we obtain the following.

\begin{prop}\label{p:mutation_oneline}
Let $w\in W$, and $\bm{i}\in R(w)$. Assume that $m_{i_1}>1$ (that is, $\overleftarrow{\boldsymbol{\mu}_{\bm{i}}}[1]\neq \mathrm{id}$). Then the $J\times J$-matrix $\overleftarrow{\boldsymbol{\mu}_{\bm{i}}}[1](\varepsilon^{\bm{i}, \mathrm{ex}})=(\varepsilon'_{s, t})_{s, t\in J}$ is given by 
\begin{align}
	\varepsilon'_{s, t}=
\begin{cases}
-1&\text{if}\ i_s=i_1\text{ and }s=t^+\not\in J_{\rm fr}, \\
1&\text{if}\ i_s=i_1\text{ and }s=t^+\in J_{\rm fr}, \\
c_{i_t,i_s}&\text{if}\ i_s=i_1\text{ and }s^+<t<s^{(2)}<t^+, \\
-c_{i_t,i_s}&\text{if}\ i_s=i_1\text{ and }t<s^+<t^+\leq s^{(2)}, \\
1&\text{if}\ i_t=i_1\text{ and }s^+=t\not\in J_{\rm fr}, \\
-1&\text{if}\ i_t=i_1\text{ and }s^+=t\in J_{\rm fr}, \\
c_{i_t,i_s}&\text{if}\ i_t=i_1\text{ and }s<t^+<s^+\leq t^{(2)},\\
-c_{i_t,i_s}&\text{if}\ i_t=i_1\text{ and }t^+<s<t^{(2)}< s^+,\\
c_{i_t,i_s}&\text{if}\ i_s=i_1\text{ and }t<s<t^+=s^+\not\in J,\\
-c_{i_t,i_s}&\text{if}\ i_t=i_1\text{ and }s<t<s^+=t^+\not\in J,\\
\varepsilon_{s, t}&\text{if}\ i_s\neq i_1\text{ and }i_t\neq i_1,\\
0&\text{otherwise}.
\end{cases}\label{eq:mutation_oneline}
\end{align}
\end{prop}

\begin{ex}\label{e:oneline_mutatedquiverexample}
	For the examples in \cref{e:quiverexample}, the quivers associated with $\overleftarrow{\boldsymbol{\mu}_{\bm{i}}}[1](\varepsilon^{\bm{i}, \mathrm{ex}})$
	are given as follows.
	\begin{itemize}
	
\item The case of $G = SL_4(\mathbb{C})$ and $\bm{i}=(2, 1, 2, 3, 2, 1)\in R(w_0)$:  $\overleftarrow{\boldsymbol{\mu}_{\bm{i}}}[1]=\mu_3\mu_1$, and
	
	\hfill
	\begin{xy} 0;<1pt,0pt>:<0pt,-1pt>::
		(-18,0) *+{3},
		(-10,0) *+{\scalebox{0.65}{$>$}},	
		(-18,30) *+{2},
		(-10,30) *+{\scalebox{0.65}{$>$}},		
		(-18,60) *+{1},
		(-10,60) *+{\scalebox{0.65}{$>$}},			
		(0,30) *+{1} ="1",
		(40,60) *+{2} ="2",
		(80,30) *+{3} ="3",
		(120,0) *+{4} ="4",		
		(160,30) *+{5} ="5",
		(200,60) *+{6} ="6",
		"3", {\ar"1"},
		"3", {\ar"5"},
		"2", {\ar"3"},
		"6", {\ar"2"},
		"1", {\ar"4"},
		"4", {\ar"3"},
		"5", {\ar"4"},
	\end{xy}
	\hfill
	\hfill
	
\item The case of $G = Sp_4(\mathbb{C})$ and $\bm{i}=(1, 2, 1, 2)\in R(w_0)$: $\overleftarrow{\boldsymbol{\mu}_{\bm{i}}}[1]=\mu_1$, and
	
	\hfill
	\begin{xy} 0;<1pt,0pt>:<0pt,-1pt>::
		(-18,0) *+{2},
		(-10,0) *+{\scalebox{0.65}{$>$}},	
		(-18,30) *+{1},
		(-10,30) *+{\scalebox{0.65}{$>$}},		
		(0,30) *+{1} ="1",
		(40,0) *+{2} ="2",
		(80,30) *+{3} ="3",
		(120,0) *+{4} ="4",
		"2", {\ar"1"},
		"1", {\ar"3"},
		"4", {\ar"2"},
	\end{xy}
	\hfill
	\hfill

	\end{itemize}	
\end{ex}

For $\bm{i}=(i_1,\dots, i_m)\in R(w)$, set $\bm{i}^{\mathrm{op}}\coloneqq (i_m,\dots, i_1)\in R(w^{-1})$. Then we have a bijection $\mathsf{R}_{\bm{i}}\colon J \to J$ given by $\mathsf{R}_{\bm{i}} \coloneqq (\xi_{\bm{i}^{\mathrm{op}}})^{-1}\circ \xi_{\bm{i}}$. Note that $\mathsf{R}_{\bm{i}}(J_{\rm uf}) = J_{\rm uf}$ and $\mathsf{R}_{\bm{i}}(J_{\rm fr}) = J_{\rm fr}$. 

\begin{thm}\label{t:mutation_max_ex}
	Let $w\in W$, and $\bm{i}\in R(w)$. Then the $J\times J$-matrix $\overleftarrow{\boldsymbol{\mu}_{\bm{i}}}(\varepsilon^{\bm{i}, \mathrm{ex}})=(\overline{\varepsilon}_{s, t})_{s, t\in J}$ is given by 
	\begin{align}
	\overline{\varepsilon}_{s,t}&=
		\begin{cases}
	1&\text{if}\ s=t^+, \\
	-1&\text{if}\ s^+=t, \\
	c_{i_t, i_s}&\text{if}\ \mathsf{R}_{\bm{i}}(t)<\mathsf{R}_{\bm{i}}(s)<\mathsf{R}_{\bm{i}}(t^+)<\mathsf{R}_{\bm{i}}(s^+),\\
	-c_{i_t, i_s}&\text{if}\ \mathsf{R}_{\bm{i}}(s)<\mathsf{R}_{\bm{i}}(t)<\mathsf{R}_{\bm{i}}(s^+)<\mathsf{R}_{\bm{i}}(t^+), \\
	c_{i_t, i_s}&\text{if}\ s=\xi_{\bm{i}^{\mathrm{op}}}^{-1}(i_s, m_{i_s})>\xi_{\bm{i}^{\mathrm{op}}}^{-1}(i_t, m_{i_t})=t\text{ and }\xi_{\bm{i}^{\mathrm{op}}}^{-1}(i_s, 1)<\xi_{\bm{i}^{\mathrm{op}}}^{-1}(i_t, 1),\\
	-c_{i_t, i_s}&\text{if}\ t=\xi_{\bm{i}^{\mathrm{op}}}^{-1}(i_t, m_{i_t})>\xi_{\bm{i}^{\mathrm{op}}}^{-1}(i_s, m_{i_s})=s\text{ and }\xi_{\bm{i}^{\mathrm{op}}}^{-1}(i_t, 1)<\xi_{\bm{i}^{\mathrm{op}}}^{-1}(i_s, 1),\\	
	0&\text{otherwise}.
	\end{cases} 
	\end{align}
\end{thm}

Note that \cref{t:mutation_max} is an immediate consequence of \cref{t:mutation_max_ex}.

\begin{rem}\label{r:opquiver_ex}
	The quiver associated with $\overleftarrow{\boldsymbol{\mu}_{\bm{i}}}(\varepsilon^{\bm{i}, \mathrm{ex}})$ is the same as the one obtained from $\Gamma_{\bm{i}^{\mathrm{op}}}$ by reversing directions of all arrows and adding arrows $s\to t$ for $s, t$ such that
	\begin{align*}
	s=\xi_{\bm{i}^{\mathrm{op}}}^{-1}(i_s, m_{i_s})>\xi_{\bm{i}^{\mathrm{op}}}^{-1}(i_t, m_{i_t})=t \text{ and }\xi_{\bm{i}^{\mathrm{op}}}^{-1}(i_s, 1)<\xi_{\bm{i}^{\mathrm{op}}}^{-1}(i_t, 1). 
	\end{align*}
	Note that, for $i,j\in I$, the condition 
	\[
		\xi_{\bm{i}^{\mathrm{op}}}^{-1}(i, m_{i})>\xi_{\bm{i}^{\mathrm{op}}}^{-1}(j, m_{j}) \text{ and }\xi_{\bm{i}^{\mathrm{op}}}^{-1}(i, 1)<\xi_{\bm{i}^{\mathrm{op}}}^{-1}(j, 1)
	\]
	is equivalent to 
	\[
	\xi_{\bm{i}}^{-1}(i, m_{i})>\xi_{\bm{i}}^{-1}(j, m_{j}) \text{ and }\xi_{\bm{i}}^{-1}(i, 1)<\xi_{\bm{i}}^{-1}(j, 1). 
	\]
\end{rem}

\begin{ex}
	For the examples in \cref{e:quiverexample}, we obtain the following (cf.~\cref{e:mutseq}). 
	\begin{itemize}
	
\item The case of $G = SL_4(\mathbb{C})$ and $\bm{i}=(2, 1, 2, 3, 2, 1)\in R(w_0)$: we have $\overleftarrow{\boldsymbol{\mu}_{\bm{i}}} = \mu_1 \mu_2 \mu_3 \mu_1$, and the quiver associated with $\overleftarrow{\boldsymbol{\mu}_{\bm{i}}}(\varepsilon^{\bm{i}, \mathrm{ex}})$ is given by 
	
	\hfill
	\begin{xy} 0;<1pt,0pt>:<0pt,-1pt>::
		(-18,0) *+{3},
		(-10,0) *+{\scalebox{0.65}{$>$}},	
		(-18,30) *+{2},
		(-10,30) *+{\scalebox{0.65}{$>$}},		
		(-18,60) *+{1},
		(-10,60) *+{\scalebox{0.65}{$>$}},			
		(0,30) *+{1} ="1",
		(40,60) *+{2} ="2",
		(80,30) *+{3} ="3",
		(120,0) *+{4} ="4",		
		(160,30) *+{5} ="5",
		(200,60) *+{6} ="6",
		"1", {\ar"3"},
		"3", {\ar"5"},
		"3", {\ar"2"},
		"6", {\ar"3"},
		"2", {\ar"6"},
		"4", {\ar"1"},
		"5", {\ar"4"},
	\end{xy}
	\hfill
	\hfill
	
\item The case of $G = Sp_4(\c)$ and $\bm{i}=(1, 2, 1, 2)\in R(w_0)$: we have $\overleftarrow{\boldsymbol{\mu}_{\bm{i}}} = \mu_2 \mu_1$, and the quiver associated with $\overleftarrow{\boldsymbol{\mu}_{\bm{i}}}(\varepsilon^{\bm{i}, \mathrm{ex}})$ is given by 

	\hfill
	\begin{xy} 0;<1pt,0pt>:<0pt,-1pt>::
		(-18,0) *+{2},
		(-10,0) *+{\scalebox{0.65}{$>$}},	
		(-18,30) *+{1},
		(-10,30) *+{\scalebox{0.65}{$>$}},		
		(0,30) *+{1} ="1",
		(40,0) *+{2} ="2",
		(80,30) *+{3} ="3",
		(120,0) *+{4} ="4",
		"1", {\ar"2"},
		"1", {\ar"3"},
		"2", {\ar"4"},
		"4", {\ar"1"},
	\end{xy}
	\hfill
	\hfill
	
	\end{itemize}	
\end{ex}

\begin{proof}[Proof of \cref{t:mutation_max_ex}]
	We prove the theorem by induction on the length $m$ of $w$. When $m=1$, there is nothing to prove. 
	
	Suppose that $m>1$. Set $\bm{i}_{2\leq}\coloneqq (i_2,\dots, i_m)$. If $m_{i_1}=1$, then $\overleftarrow{\boldsymbol{\mu}_{\bm{i}}}[1]=\mathrm{id}$, and $1$ is not connected with other vertices by arrows. Moreover, we have $\mathsf{R}_{\bm{i}_{2\leq}}(s)=\mathsf{R}_{\bm{i}}(s+1)$ and $s^{+, 2\leq}+1=(s+1)^{+}$ for all $s\in J\setminus \{ m \}$, where, to avoid confusion, we write $s^{+, 2\leq}$ for $s^+$ with respect to $\bm{i}_{2\leq}$. Hence we obtain the desired result since the theorem holds for $s_{i_1}w$ by our induction hypothesis. 
	
	Assume that $m_{i_1}>1$. Consider the quiver $\Gamma_{\bm{i}}^{(1)}$ associated with $\overleftarrow{\boldsymbol{\mu}_{\bm{i}}}[1](\varepsilon^{\bm{i}, \mathrm{ex}})$. By \cref{p:mutation_oneline}, if we 
\begin{itemize}
	\item[(i)] remove the vertex $s\in J$ such that $\xi_{\bm{i}}(s)=(i_1, m_{i_1})$ and the arrows attached to $s$, 
	\item[(ii)] remove the arrows among $(i_1, m_{i_1}-1)$ and $(i, m_i)$, $i\in I\setminus \{i_1\}$, 
	\item[(iii)] move each vertex $s$ of $J_{\rm uf}$ such that $i_s=i_1$ to the place of $s^+$, 
\end{itemize}
then $\Gamma_{\bm{i}}^{(1)}$ becomes the quiver $\Gamma_{\bm{i}_{2\leq}}$ associated with $\varepsilon^{\bm{i}_{2\leq}, \mathrm{ex}}$. After this procedure, the calculation of $\overleftarrow{\boldsymbol{\mu}_{\bm{i}}}[m]\circ \cdots \circ \overleftarrow{\boldsymbol{\mu}_{\bm{i}}}[2]$ is nothing but that of $\overleftarrow{\boldsymbol{\mu}_{\bm{i}_{2\leq}}}(\varepsilon^{\bm{i}_{2\leq}, \mathrm{ex}})$. Moreover, in the processes (i), (ii), (iii), we do not remove the arrows connected with the vertices where mutations are performed in $\overleftarrow{\boldsymbol{\mu}_{\bm{i}}}[m]\circ \cdots \circ \overleftarrow{\boldsymbol{\mu}_{\bm{i}}}[2]$. Hence, to calculate $\overleftarrow{\boldsymbol{\mu}_{\bm{i}}}(\varepsilon^{\bm{i}, \mathrm{ex}})$, we may first consider $\overleftarrow{\boldsymbol{\mu}_{\bm{i}_{2\leq}}}(\varepsilon^{\bm{i}_{2\leq}, \mathrm{ex}})$, and next recall the arrows which are removed by (i) and (ii) via identification above. By our induction hypothesis and \cref{r:opquiver_ex}, the quiver associated with $\overleftarrow{\boldsymbol{\mu}_{\bm{i}_{2\leq}}}(\varepsilon^{\bm{i}_{2\leq}, \mathrm{ex}})$ is obtained from $\Gamma_{(\bm{i}_{2\leq})^{\mathrm{op}}}$ by reversing directions of all arrows and adding arrows $s\to t$ for $s, t$ such that
\[
s=\xi_{(\bm{i}_{2\leq})^{\mathrm{op}}}^{-1}(i_s, m'_{i_s})>\xi_{(\bm{i}_{2\leq})^{\mathrm{op}}}^{-1}(i_t, m'_{i_t})=t \text{ and }\xi_{(\bm{i}_{2\leq})^{\mathrm{op}}}^{-1}(i_s, 1)<\xi_{(\bm{i}_{2\leq})^{\mathrm{op}}}^{-1}(i_t, 1),
\] 
where 
\[
m'_i\coloneqq
\begin{cases}
m_{i_1}-1&\text{if }i=i_1,\\
m_i&\text{otherwise}.
\end{cases}
\]
Moreover, in the processes (i) and (ii), we remove the arrows $s\to t$ such that 
\begin{itemize}
	\item $s=\xi_{\bm{i}}^{-1}(i_1, m_{i_1}-1)$ and $t=\xi_{\bm{i}}^{-1}(i_1, m_{i_1})$, 
	\item $s=\xi_{\bm{i}}^{-1}(i_1, m_{i_1})>\xi_{\bm{i}}^{-1}(i_t, m_{i_t})=t$ (note that we also have $\xi_{\bm{i}}^{-1}(i_1, 1)=1<\xi_{\bm{i}}^{-1}(i_t, 1)$), 
	\item $\xi_{\bm{i}}^{-1}(i_1, m_{i_1})>\xi_{\bm{i}}^{-1}(i_s, m_{i_s})=s$ and $t=\xi_{\bm{i}}^{-1}(i_1, m_{i_1}-1)$.  	
\end{itemize}
Since $\xi_{(\bm{i}_{2\leq})^{\mathrm{op}}}=\xi_{\bm{i}^{\mathrm{op}}}\mid_{J\setminus \{m\}}$, $\xi_{\bm{i}^{\mathrm{op}}}^{-1}(i_1, m_{i_1})=m$, and $\Gamma_{(\bm{i}_{2\leq})^{\mathrm{op}}}$ is a subquiver of $\Gamma_{\bm{i}^{\mathrm{op}}}$, we obtain the desired result.  
\end{proof}

Next we calculate the explicit form of $\overleftarrow{\boldsymbol{\mu}_{\bm{i}}}(\mathbf{D}_{\bm{i}})$. Indeed, the exchange relations appearing in $\overleftarrow{\boldsymbol{\mu}_{\bm{i}}}$ correspond to the following well-known determinantal identities as pointed out in \cite{GLS:Kac-Moody}; recall \cref{n:indexplus} for the notation.

\begin{prop}[{A system of determinantal identities \cite[Theorem 1.17]{FZ:Double}}]\label{p:T-sys}
	Let $\bm{i}=(i_1,\dots, i_m)\in R(w)$. For $s, t\in J\cup\{ 0\}$, set  
	\[
	D_{\bm{i}}(s, t)\coloneqq \begin{cases}
	D_{w_{\leq s}\varpi_{i_s}, w_{\leq t}\varpi_{i_s}}&\text{if }s\in J,\\
	0&\text{if }s=0\text{ and }t\in J,\\	
	1&\text{if }s=t=0.\\
	\end{cases}
	\]
	Then, for $s, t\in J$ with $i_s=i_t$ and $s> t$, the following equality holds: 
	\begin{align}
	D_{\bm{i}}(s^-, t^-)D_{\bm{i}}(s, t)&=D_{\bm{i}}(s, t^-)D_{\bm{i}}(s^-, t)+\prod_{j\in I\setminus\{i_s\}}D_{\bm{i}}(s^-(j), t^-(j))^{-c_{j, i_s}}. \label{eq:T-sys}
	\end{align}
\end{prop}

\begin{rem}\label{r:T-systemshift}
We should note the following property of the system of identities \eqref{eq:T-sys}. Let $\bm{i}=(i_1,\dots, i_m)\in R(w)$, and $1\leq k\leq m$. Write $\bm{i}_{k\leq}\coloneqq (i_k,\dots, i_m)=(i'_1,\dots, i'_{m-k+1})$. For $s\in \{ 1,\dots, m-k+1\}$, the numbers $s^{-}$ and $s^-(j)$ ($j\in I$) associated with $\bm{i}_{k\leq}$ are denoted by $s^{-, k}$ and $s^{-, k}(j)$, respectively. Then, by \cref{p:T-sys}, we have 
\[
	D_{\bm{i}_{k\leq}}(s^{-,k}, t^{-,k})D_{\bm{i}_{k\leq}}(s, t)=D_{\bm{i}_{k\leq}}(s, t^{-, k})D_{\bm{i}_{k\leq}}(s^{-, k}, t)+\prod_{j\in I\setminus\{i'_s\}}D_{\bm{i}_{k\leq}}(s^{-, k}(j), t^{-, k}(j))^{-c_{j, i'_s}}
\]
for $s, t\in \{1,\dots, m-k+1\}$ with $i'_s=i'_t$ and $s> t$. Then the equality still holds if we substitute 
\begin{itemize}
    \item $D_{\bm{i}}(s^{-,k}+k-1, t^{-,k}+k-1)$ for $D_{\bm{i}_{k\leq}}(s^{-,k}, t^{-,k})$,
    \item $D_{\bm{i}}(s+k-1, t+k-1)$ for $D_{\bm{i}_{k\leq}}(s, t)$,
    \item $D_{\bm{i}}(s+k-1, t^{-, k}+k-1)$ for $D_{\bm{i}_{k\leq}}(s, t^{-, k})$,
    \item $D_{\bm{i}}(s^{-, k}+k-1, t+k-1)$ for $D_{\bm{i}_{k\leq}}(s^{-, k}, t)$,
    \item $D_{\bm{i}}(s^{-, k}(j)+k-1, t^{-, k}(j)+k-1)$ for $D_{\bm{i}_{k\leq}}(s^{-, k}(j), t^{-, k}(j))$ ($j\in I\setminus\{i'_s\}$). 
\end{itemize}
\end{rem}

\begin{proof}[{Proof of \cref{t:mutation_max_variable}}]
In the mutation sequence $\overleftarrow{\boldsymbol{\mu}_{\bm{i}}}$, the mutation $\mu_s$ appears $(m_{i_s}-k[s])$ times for $s\in J$. Hence it suffices to show that by each mutation $\mu_s$ in $\overleftarrow{\boldsymbol{\mu}_{\bm{i}}}$, the cluster variable at $s$ changes from an element of the form $D_{\bm{i}}(t_1, t_2)$ to $D_{\bm{i}}(t_1^+, t_2^+)$; recall the notation in \cref{p:T-sys}. Here we set $t_2^+\coloneqq \xi_{\bm{i}}^{-1}(i_{t_1}, 1)$ if $t_2=0$. 
First we consider the mutation sequence $\overleftarrow{\boldsymbol{\mu}_{\bm{i}}}[1]$. If $m_{i_1}=1$, then $\overleftarrow{\boldsymbol{\mu}_{\bm{i}}}[1]=\mathrm{id}$. If $m_{i_1}>1$, then the observation before \cref{p:mutation_oneline} shows that the desired statement holds for $\overleftarrow{\boldsymbol{\mu}_{\bm{i}}}[1]$ because each exchange relation is of the form \eqref{eq:T-sys}. More precisely, the exchange relation for the mutation $\mu_{\xi_{\bm{i}}^{-1}(i_1, k)}$ in  $\overleftarrow{\boldsymbol{\mu}_{\bm{i}}}[1]$ corresponds to the equality \eqref{eq:T-sys} for $s=\xi_{\bm{i}}^{-1}(i_1, k+1)$ and $t=1$. 

Let $\bm{i}_{2\leq}\coloneqq (i_2,\dots, i_m)=(i'_1,\dots, i'_{m-1})$. Comparing the FZ-seed $\mathbf{s}_{\bm{i}_{2\leq}}$ (whose index set is $\{1,\dots, m-1\}$) with  $\overleftarrow{\boldsymbol{\mu}_{\bm{i}}}[1](\mathbf{s}_{\bm{i}})$, we observe the following:  
\begin{itemize}
    \item[$(\ast)$] applying the modification (i)--(iii) in the proof of \cref{t:mutation_max_ex} to $\overleftarrow{\boldsymbol{\mu}_{\bm{i}}}[1](\mathbf{s}_{\bm{i}})$, we obtain the FZ-seed which can be obtained by substituting $D_{\bm{i}}(s+1, 1)$ for the cluster variable $D_{\bm{i}_{2\leq}}(s, 0)$ of $\mathbf{s}_{\bm{i}_{2\leq}}$ at $s\in \{1,\dots, m-1\}$. 
\end{itemize} 
Hence \cref{r:T-systemshift} and the observation of $\overleftarrow{\boldsymbol{\mu}_{\bm{i}}}[1]$ imply that the exchange relations appearing in the process $\overleftarrow{\boldsymbol{\mu}_{\bm{i}}}[2]$ from $\overleftarrow{\boldsymbol{\mu}_{\bm{i}}}[1](\mathbf{s}_{\bm{i}})$ are the equalities of the form \eqref{eq:T-sys}. In particular, the cluster variable at $s$ changes from an element of the form $D_{\bm{i}}(t_1, t_2)$ to $D_{\bm{i}}(t_1^+, t_2^+)$ in this process. Iterating this argument, we can calculate $\overleftarrow{\boldsymbol{\mu}_{\bm{i}}}[m] \circ \cdots \circ \overleftarrow{\boldsymbol{\mu}_{\bm{i}}}[1](\mathbf{s}_{\bm{i}})$ inductively, and we know that the cluster variable at $s$ changes from an element of the form $D_{\bm{i}}(t_1, t_2)$ to $D_{\bm{i}}(t_1^+, t_2^+)$ at each mutation.
\end{proof}

\begin{rem}\label{r:minorvariable}
By \cref{t:independence_of_rex} and the proof of \cref{t:mutation_max_variable}, 
\[
D_{w_1\varpi_i, w_2\varpi_i}\text{ such that }\begin{cases}
\ell (w)\geq \ell(w_1)> \ell(w_2),\\
\ell (w_1 s_i)<\ell(w_1),\\
\ell (w)=\ell(w_1)+\ell(w_1^{-1}w)\text{ and }\ell (w_1)=\ell(w_2)+\ell(w_2^{-1}w_1)
\end{cases}
\]
is a cluster variable of $\mathbb{C}[U_w^-]$ with respect to the cluster structure whose initial FZ-seed is $\mathbf{s}_{\bm{i}}$ ($\bm{i}\in R(w)$). Indeed, for $(j_1,\dots, j_{m_2})\in R(w_2)$, $(j_{m_2+1},\dots, j_{m_1})\in R(w_2^{-1}w_1)$, and $(j_{m_1+1},\dots, j_m)\in R(w_1^{-1}w)$, let
$\bm{j} \coloneqq (j_{1},\dots, j_{m_2}, j_{m_2+1},\dots, j_{m_1}, j_{m_1+1},\dots, j_m)\in R(w)$. Then
\[
D_{w_1\varpi_i, w_2\varpi_i}=D_{\bm{j}}(m_1, m_2).
\]
\end{rem}
\begin{rem}
The proofs of \cref{t:mutation_max,t:mutation_max_variable} work also in the quantum settings by using \cite[Proposition 5.5]{GLS:Selecta} (see also \cite[Proposition 3.24 and Remark 3.25]{KimOya:qtwist}). 
\end{rem}

\section{Double Bruhat cells and unipotent cells}\label{a:doubleBruhat}

In this appendix, we see a direct relation between the cluster structure on the coordinate ring of $U^-_w$ and that of the \emph{double Bruhat cell} $G^{w, e}\coloneqq B^- \cap B\widetilde{w}B$, 
where $\widetilde{w}\in N_G(H)$ is a lift for $w$. Indeed, Berenstein--Fomin--Zelevinsky \cite{BFZ2} and Williams \cite{Wil} proved that $\mathbb{C}[G^{w, e}]$ (more generally, $\mathbb{C}[G^{w, w'}]$ for $w, w'\in W$) has a structure of an upper cluster algebra. Actually, their upper cluster algebra structure induces an upper cluster algebra structure of $\mathbb{C}[U^-_w]$, but as far as the authors know, this procedure is not explicitly explained in the literature. Hence we demonstrate it in this appendix. We adopt an algebraic proof, and as a byproduct we show that $\mathbb{C}[U^-_w]$ also has an (ordinary) cluster algebra structure.  

\begin{rem}
Geiss--Leclerc--Schr\"{o}er \cite{GLS:Kac-Moody} verified the cluster algebra structure on  $\mathbb{C}[U^-_w]$ without going through $G^{w, e}$ under the assumption that $G$ is a \emph{symmetric} Kac--Moody group. 
Demonet \cite{Dem} extended Geiss--Leclerc--Schr\"{o}er's method to the case of symmetrizable Kac--Moody groups for specific $w$. Goodearl--Yakimov \cite{GY:Mem} showed the quantum cluster algebra structure on the quantum analogue of $\mathbb{C}[U^-_w]$ for an arbitrary symmetrizable Kac--Moody group $G$ and $w\in W$. They also announced the classical version of their results in  \cite[Section 1.2]{GY:Poisson}. 
\end{rem}

Let $w \in W$, and write 
\[
\widetilde{D}_{f, v}\coloneqq C_{f, v}|_{G^{w, e}},\quad 
\widetilde{D}_{u\lambda,u'\lambda}\coloneqq \Delta_{u\lambda,u'\lambda}|_{G^{w, e}},\quad 
e^{\mu}\coloneqq \Delta_{\mu,\mu}|_{H}
\]
for $\lambda\in P_+$, $f\in V(\lambda)^{\ast}$, $v\in V(\lambda)$, $u, u'\in W$, and $\mu\in P$. It is easy to see that the multiplication map gives an isomorphism 
\[
U^-_w\times H\xrightarrow{\sim} G^{w, e}
\]
of varieties. Hence we have an isomorphism $\mathbb{C}[G^{w, e}]\xrightarrow{\sim} 
\mathbb{C}[U^-_w]\otimes_{\mathbb{C}}\mathbb{C}[H]$ of $\c$-algebras given by 
\[
\widetilde{D}_{f, v}\mapsto D_{f, v}\otimes e^{\mu}
\]
for $f\in V(\lambda)^{\ast}$, $\lambda\in P_+$, and a weight vector $v\in V(\lambda)$ of weight $\mu\in P$. In particular, there exists a $\mathbb{C}$-algebra isomorphism 
\begin{align}
    \mathbb{C}[G^{w, e}]/(\widetilde{D}_{\varpi_i, \varpi_i}-1\mid i\in I)\xrightarrow{\sim}  \mathbb{C}[U^-_w],\label{eq:Bruhatunip}
\end{align}
where $(\widetilde{D}_{\varpi_i, \varpi_i}-1\mid i\in I)$ is the ideal of $\mathbb{C}[G^{w, e}]$ generated by $\{\widetilde{D}_{\varpi_i, \varpi_i}-1\mid i\in I\}$. 

Let us see the upper cluster algebra structure on $G^{w, e}$. Fix $\bm{i}=(i_1,\dots, i_m)\in R(w)$. We set $\overline{I}\coloneqq \{\overline{i}\mid i\in I\}$, and write 
	\begin{align*}
		\overline{i}^{+} &\coloneqq \min(\{m+1\}\cup\{ 1\leq j\leq m \mid i_{j}=i\})
	\end{align*}
	for $i\in I$. Let $i_{\overline{i}}\coloneqq i$, and consider that $\overline{i}<j$ for all $1 \leq j \leq m$. We set 
	\[
	\widetilde{J} \coloneqq \overline{I} \cup \{1,\dots,m\}\ \text{and}\ \widetilde{J}_{\rm fr} \coloneqq \overline{I}\cup \{j\in J\mid j^+=m+1\}. 
	\]
	Define a $J_{\rm uf} \times \widetilde{J}$-integer matrix $\widetilde{\varepsilon}^{\bm{i}}=(\widetilde{\varepsilon}_{s, t})_{s\in J_{\rm uf}, t\in \widetilde{J}}$ by
	\[
	\widetilde{\varepsilon}_{s, t}\coloneqq
	\begin{cases}
	-1&\text{if}\ s=t^+, \\
	-c_{i_t, i_s}&\text{if}\ t<s<t^+<s^+,\\
	1&\text{if}\ s^+=t, \\
	c_{i_t, i_s}&\text{if}\ s<t<s^+<t^+, \\
	0&\text{otherwise}.
	\end{cases}
	\]
	For $s\in \widetilde{J}$, we set  
\[
\widetilde{D}(s, \bm{i})\coloneqq \begin{cases}
\widetilde{D}_{w_{\leq s}\varpi_{i_s}, \varpi_{i_s}}&\text{if }s\in \{1,\dots,m\},\\
\widetilde{D}_{\varpi_{i_s}, \varpi_{i_s}}&\text{if }s\in \overline{I}.
\end{cases}
\]
Let us consider the upper cluster algebra $\mathscr{U}(\widetilde{\mathbf{s}}_{t_0})$ whose initial FZ-seed is given as $\widetilde{\mathbf{s}}_{t_0} = (\widetilde{\mathbf{A}}_{t_0} = (\widetilde{A}_{s; t_0})_{s \in \widetilde{J}}, \widetilde{\varepsilon}^{\bm{i}})$. 
\begin{thm}[{\cite[Theorem 2.10]{BFZ2} and \cite[Theorem 4.16]{Wil}}]\label{t:Double_cluster}
There exists a $\mathbb{C}$-algebra isomorphism 
\[
\mathscr{U}(\widetilde{\mathbf{s}}_{t_0})\xrightarrow{\sim} \mathbb{C}[G^{w, e}]\ \text{given by }\widetilde{A}_{s; t_0}\mapsto \widetilde{D}(s, \bm{i})\ \text{for }s\in \widetilde{J}.
\]
\end{thm}

\begin{ex}
	Recall the quiver description of the exchange matrix appearing in \cref{e:initial}. When $G = SL_5(\mathbb{C})$ and $\bm{i} = (1, 2, 1, 3, 2, 1, 4, 3, 2, 1)$, the initial FZ-seed $(\widetilde{\mathbf{A}}_{t_0}, \widetilde{\varepsilon}^{\bm{i}})$ is described as follows:
	
	\hfill
	\scalebox{0.7}[0.7]{
		\begin{xy} 0;<1pt,0pt>:<0pt,-1pt>::
			(60,0) *+{\widetilde{D}_{\varpi_4, \varpi_4}} ="10",
			(0,30) *+{\widetilde{D}_{\varpi_3, \varpi_3}} ="11",
			(-60,60) *+{\widetilde{D}_{\varpi_2, \varpi_2}} ="12",
			(-120,90) *+{\widetilde{D}_{\varpi_1, \varpi_1}} ="13",
			(180,0) *+{\widetilde{D}_{w_0 \varpi_4, \varpi_4}} ="0",
			(120,30) *+{\widetilde{D}_{s_1 s_2 s_1 s_3 \varpi_3, \varpi_3}} ="1",
			(240,30) *+{\widetilde{D}_{w_0 \varpi_3, \varpi_3}} ="2",
			(60,60) *+{\widetilde{D}_{s_1 s_2 \varpi_2, \varpi_2}} ="3",
			(180,60) *+{\widetilde{D}_{s_1 s_2 s_1 s_3 s_2 \varpi_2, \varpi_2}} ="4",
			(300,60) *+{\widetilde{D}_{w_0 \varpi_2, \varpi_2}} ="5",
			(0,90) *+{\widetilde{D}_{s_1 \varpi_1, \varpi_1}} ="6",
			(120,90) *+{\widetilde{D}_{s_1 s_2 s_1 \varpi_1, \varpi_1}} ="7",
			(240,90) *+{\widetilde{D}_{s_1 s_2 s_1 s_3 s_2 s_1 \varpi_1, \varpi_1}} ="8",
			(360,90) *+{\widetilde{D}_{w_0 \varpi_1, \varpi_1}} ="9",
			"6", {\ar"13"},
			"3", {\ar"12"},
			"1", {\ar"11"},
			"10", {\ar"1"},
			"11", {\ar"3"},
			"12", {\ar"6"},			
			"1", {\ar"0"},
			"3", {\ar"1"},
			"1", {\ar"4"},
			"4", {\ar"2"},
			"6", {\ar"3"},
			"3", {\ar"7"},
			"7", {\ar"4"},
			"4", {\ar"8"},
			"8", {\ar"5"},
			"2", {\ar"1"},
			"4", {\ar"3"},
			"5", {\ar"4"},
			"7", {\ar"6"},
			"8", {\ar"7"},
			"9", {\ar"8"},
		\end{xy}
	}
	\hfill
	\hfill
	
\end{ex}

Our main statement in this appendix is that 
\cref{t:Double_cluster} implies the following theorem. 

\begin{thm}\label{t:Bruhat-unip}
There is a $\mathbb{C}$-algebra isomorphism $\mathscr{U}(\mathbf{A}_{t_0}, \varepsilon^{\bm{i}})\xrightarrow{\sim} \mathbb{C}[U_w^-]$ given by 
\[
A_{s; t_0}\mapsto D(s, \bm{i})
\]
for $s\in J$. Moreover, $\mathscr{U}(\mathbf{A}_{t_0}, \varepsilon^{\bm{i}})=\mathscr{A}(\mathbf{A}_{t_0}, \varepsilon^{\bm{i}})$.
\end{thm}

\begin{proof}
The fraction field $\mathcal{F}_w$ of the coordinate ring $\mathbb{C}[U^-_w]$ is isomorphic to the field of rational functions in $|J|$-variables, and $\{D(s, \bm{i})\mid s\in J\}$ forms a free generating set of $\mathcal{F}_w$ (cf.~\cref{t:Chamber_Ansatz}). Hence the upper cluster algebra $\mathscr{U}(\mathbf{A}_{t_0}, \varepsilon^{\bm{i}})$ can be realized as a $\mathbb{C}$-subalgebra of $\mathcal{F}_w$ in such a manner that the initial FZ-seed is given as 
\[
(\mathbf{A}_{t_0}\coloneqq (D(s, \bm{i})\mid s\in J), \varepsilon^{\bm{i}}). 
\]
Moreover, we consider $\mathscr{U}(\widetilde{\mathbf{A}}_{t_0}, \widetilde{\varepsilon}^{\bm{i}})$ as a $\c$-subalgebra of the fraction field of $\mathbb{C}[G^{w, e}]$ by \cref{t:Double_cluster}. By \eqref{eq:Bruhatunip}, we have a $\mathbb{C}$-algebra homomorphism 
\[
\pi\colon \mathbb{C}[G^{w, e}]\twoheadrightarrow \mathbb{C}[U_w^-]\hookrightarrow \mathcal{F}_w. 
\]
We shall prove that the image $\Image\pi$ of $\pi$ is contained in $\mathscr{U}(\mathbf{A}_{t_0}, \varepsilon^{\bm{i}})$. It suffices to show that 
\begin{align}
\pi(\widetilde{A}_{s; t})=\begin{cases}
A_{s; t}&\text{if }s\in J,\\
1&\text{if }s\in \overline{I}
\end{cases}\label{eq:specializationimage}
\end{align}
for all $t\in\mathbb{T}$ and $s\in\widetilde{J}$. We show it by induction on the distance from $t_0$ in $\mathbb{T}$. When $t=t_0$, \eqref{eq:specializationimage} is clear from the definition of the initial FZ-seed. Let $t, t'\in \mathbb{T}$ and $k\in J_{\rm uf}$ such that $t\overset{k}{\text{---}} t'$, and suppose that \eqref{eq:specializationimage} holds at $t$. First it is clear from the definition of exchange relations that 
\[
\pi(\widetilde{A}_{s; t'})=\pi(\widetilde{A}_{s; t})=A_{s; t}=A_{s; t'}
\]
for $s\in \widetilde{J}\setminus \{k\}$, where we set $A_{s; t'}=A_{s; t}=1$ for $s\in \overline{I}$. Next, by our induction hypothesis and the definition of exchange relations again, we have
\begin{align*}
\pi(\widetilde{A}_{k; t})\pi(\widetilde{A}_{k; t'})&=\pi(\widetilde{A}_{k; t}\widetilde{A}_{k; t'})\\
&=\pi\left(\prod_{s\in \widetilde{J}}\widetilde{A}_{s; t}^{[\varepsilon_{k, s}^{(t)}]_+}+\prod_{s\in \widetilde{J}}\widetilde{A}_{s; t}^{[-\varepsilon_{k, s}^{(t)}]_+}\right)\\
&=\prod_{s\in J}A_{s; t}^{[\varepsilon_{k,s}^{(t)}]_+}+\prod_{s\in J}A_{s; t}^{[-\varepsilon_{k,s}^{(t)}]_+}\\
&=A_{k; t}A_{k; t'}=\pi(\widetilde{A}_{k; t})A_{k; t'}. 
\end{align*}
Hence we have $\pi(\widetilde{A}_{k; t'})=A_{k; t'}$. Thus, we obtain \eqref{eq:specializationimage}, and $\mathbb{C}[U_w^-]\simeq \Image\pi\subset \mathscr{U}(\mathbf{A}_{t_0}, \varepsilon^{\bm{i}})$. 

By \eqref{eq:specializationimage} again, $\pi$ induces a surjective $\c$-algebra homomorphism between the ordinary cluster algebras 
\[
\mathscr{A}(\widetilde{\mathbf{A}}_{t_0}, \widetilde{\varepsilon}^{\bm{i}})\twoheadrightarrow \mathscr{A}(\mathbf{A}_{t_0}, \varepsilon^{\bm{i}}). 
\]
Here remark that $\mathscr{A}(\widetilde{\mathbf{A}}_{t_0}, \widetilde{\varepsilon}^{\bm{i}})\subset\mathscr{U}(\widetilde{\mathbf{A}}_{t_0}, \widetilde{\varepsilon}^{\bm{i}})$ and $\mathscr{A}(\mathbf{A}_{t_0}, \varepsilon^{\bm{i}})\subset \mathscr{U}(\mathbf{A}_{t_0}, \varepsilon^{\bm{i}})$. Hence $\mathscr{A}(\mathbf{A}_{t_0}, \varepsilon^{\bm{i}})$ is contained in $\Image\pi\simeq \mathbb{C}[U_w^-]$. By \cite[Proposition 8.5]{GLS:Kac-Moody} (or \cite[Corollary 2.22]{KimOya}) and \cref{r:minorvariable}, we conclude that the $\c$-algebra generators of $\mathbb{C}[U_w^-]$ can be obtained as cluster variables of $\mathscr{A}(\mathbf{A}_{t_0}, \varepsilon^{\bm{i}})$. These imply that $\mathscr{A}(\mathbf{A}_{t_0}, \varepsilon^{\bm{i}})=\mathbb{C}[U_w^-]$. Moreover, since $\mathbb{C}[U^-_w]$ is isomorphic to a certain localization of a polynomial ring by \cite[Proposition 8.5]{GLS:Kac-Moody} (or \cite[Corollary 2.22]{KimOya}), $\mathbb{C}[U^-_w]$ is a unique factorization domain. Hence, by \cite[Corollary 1.5]{GLS:factorial}, we obtain 
\[
\mathbb{C}[U_w^-]=\mathscr{A}(\mathbf{A}_{t_0}, \varepsilon^{\bm{i}})=\mathscr{U}(\mathbf{A}_{t_0}, \varepsilon^{\bm{i}}).
\]
\end{proof}

\section{Upper global bases}\label{a:upper_global_bases}

In this appendix, we briefly summarize some properties of the upper global basis together with precise references.  
For $w \in W$, we regard the intersection $U^- \cap X(w)$ as a closed subvariety of $U^-$, and denote by 
\[
\pi_w \colon \c[U^-] \twoheadrightarrow \c[U^- \cap X(w)]
\]
the restriction map. 
Lusztig \cite{Lus_can, Lus_quivers, Lus1} and Kashiwara \cite{Kas1,Kas2,Kas3} constructed a specific $\c$-basis  ${\bf B}^{\rm up}\subset \c[U^-]$ of $\c[U^-]$ via the quantized enveloping algebra $U_q (\mathfrak{g})$ associated with $\mathfrak{g}$. 
This is called (the specialization at $q = 1$ of) the \emph{dual canonical basis}/\emph{upper global basis} of $\c[U^-]$. 

\begin{thm}[{\cite[Propositions 3.2.3 and 3.2.5]{Kas4} (see also \cite[Corollary 3.20]{FO})}]\label{t:Demazure}
For $w\in W$, there uniquely exists a subset ${\bf B}^{\rm up}[w] \subset {\bf B}^{\rm up}$ having the following properties {\rm (1)}--{\rm (3)}.
\begin{enumerate}
\item[{\rm (1)}] The set $\{\pi_w (b) \mid b \in {\bf B}^{\rm up}[w]\}$ forms a $\c$-basis of $\c[U^- \cap X(w)]$.
\item[{\rm (2)}] The equality $\pi_w (b) = 0$ holds for all $b \in {\bf B}^{\rm up} \setminus {\bf B}^{\rm up}[w]$.
\item[{\rm (3)}] For $\lambda\in P_+$, there uniquely exists a subset ${\bf B}^{\rm up}_w[\lambda] \subset {\bf B}^{\rm up}[w]$ such that 
\[
\{\sigma / \tau_{\lambda} \mid \sigma \in H^0(X(w), \mathcal{L}_{\lambda})\} = \sum_{b \in {\bf B}^{\rm up}_w[\lambda]} \c \pi_w (b).
\]
\end{enumerate}
\end{thm}

\begin{rem}
\cref{t:Demazure} is verified in a quantum setting in \cite{Kas4}. We rephrase the results of \cite{Kas4} specializing $q$ to $1$ here. 
\end{rem}

Since $U^- _w$ is an open subvariety of $U^- \cap X(w)$, given by $D_{w \varpi_i, \varpi_i} \neq 0$ for $i \in I$, we have 
\[
\c[U^- \cap X(w)][D_{w \varpi_i, \varpi_i}^{-1}\mid i\in I]\simeq \c[U^- _w]. 
\]
For $b \in {\bf B}^{\rm up}[w]$, write $b_w\coloneqq \pi_w (b)$. Through the isomorphism above, the set 
\[
{\bf B}_w ^{\rm up} \coloneqq \{ b_w \cdot \prod_{i \in I} D_{w \varpi_i, \varpi_i} ^{-a_i} \mid b \in {\bf B}^{\rm up}[w],\ (a_i)_{i \in I} \in \z_{\ge 0} ^I\}
\]
forms a $\c$-basis of $\c[U^- _w]$, which is called (the specialization at $q = 1$ of) the \emph{dual canonical basis}/\emph{upper global basis} of $\c[U_w ^-]$ (see \cite[Proposition 4.5 and Definition 4.6]{KimOya}).
By \cref{t:Demazure} (3), this basis ${\bf B}_w ^{\rm up}$ has the property $({\rm T})_4$ in Section \ref{ss:bases_tropical}. 
In addition, Kashiwara--Kim \cite[Lemma 3.6, Lemma 3.12, and Theorem 3.16]{KasKim} and Qin \cite[Theorem 9.5.1]{Qin3} proved that ${\bf B}_w ^{\rm up}$ has the properties $({\rm T})_1$--$({\rm T})_3$ in Section \ref{ss:bases_tropical}, which implies \cref{t:existence_of_bases}.

\begin{rem}
In a quantum setting, the upper global basis gives a common triangular basis in the sense of \cite[Definition 6.1.3]{Qin}; see \cite[Proposition 3.19 and Remark 3.20]{KasKim} and \cite[Theorem 9.5.1]{Qin3}.
\end{rem}

\begin{rem}
When the Cartan matrix of $\mathfrak{g}$ is symmetric, the opposite dominance order $\preceq_{\varepsilon_t} ^{\rm op}$ in \cref{d:main_valuation} is the same as the partial order $\preceq_{\mathscr{S}}$ in \cite[Section 3.3]{KasKim} by \cite[Proposition 3.3]{KasKim}.
In particular, the extended $g$-vector $g_t$ in \cref{d:weakly_pointed} corresponds to ${\bf g}_{\mathscr{S}} ^L$ defined in \cite[Definition 3.8]{KasKim} by \cite[Lemma 3.6]{KasKim}.
\end{rem}

\bibliographystyle{jplain} 
\def\cprime{$'$} 

\end{document}